\documentclass[12pt]{amsart}
\usepackage{times}
\usepackage{anysize}
\marginsize{2.7cm}{2.44cm}{2.0cm}{3.0cm}
\usepackage[small,it]{caption}
 \usepackage[latin1]{inputenc}
 \usepackage[dvips]{graphicx}
 \usepackage{wrapfig}
 \usepackage{amsmath}
 \usepackage{amsthm}
 \usepackage{amsfonts}
 \usepackage{amssymb}
 \usepackage{layout
} \usepackage{verbatim}
 \usepackage{alltt}
\usepackage{yfonts}

\usepackage[all]{xy}
\usepackage{setspace}

\newtheorem*{thma}{Theorem A}
\newtheorem*{thmb}{Theorem B}
\newtheorem*{claim}{Claim}

\newcommand{\hli}{\hbox{H}_{\al}}
\newcommand{\hLi}{\hbox{H}_{\all}}
\newcommand{\hhh}{\bigwedge^{r-1}\textbf{\textup{Hom}}(\pmb{\mathcal{L}_s},\oo[[\Gamma\times\pmb{\Delta}]])}

\newcommand{\all}{\mathbb{L}}
\newcommand{\LL}{\Lambda}
\newcommand{\TT}{\mathbb{T}}
\newcommand{\QQ}{\mathbb{Q}}
\newcommand{\FF}{\mathcal{F}}

\newcommand{\lra}{\longrightarrow}
\newcommand{\ZZ}{\mathbb{Z}}
\newcommand{\PP}{\mathcal{P}}
\newcommand{\MM}{\mathcal{M}}
\newcommand{\Gal}{\textup{Gal}}
\newcommand{\KS}{\textbf{\textup{KS}}}
\newcommand{\ES}{\textbf{\textup{ES}}}
\newcommand{\NN}{\mathcal{N}}
\newcommand{\ra}{\rightarrow}
\newcommand{\xx}{\mathbf{X}}
\newcommand{\be}{\begin{equation}}
\newcommand{\ee}{\end{equation}}

\newcommand{\XX}{\mathcal{X}}
\newcommand{\kk}{\mathcal{K}}
\newcommand{\al}{\mathcal{L}}

\newcommand{\oo}{\mathcal{O}}

\newcommand{\qq}{\hbox{\frakfamily q}}
\newcommand{\vv}{\mathbb{V}}
\newcommand{\mm}{\hbox{\frakfamily m}}

\newcommand{\FFc}{\mathcal{F}_{\textup{\lowercase{can}}}}
\newcommand{\Hom}{\textup{Hom}}

\newcommand{\BK}{\textup{BK}}
\newcommand{\Gr}{\textup{Gr}}
\newcommand{\hiwasawa}{\mathbf{H.Iw.}}
\newcommand{\hord}{\mathbf{H.P}}
\newcommand{\hntz}{\mathbf{H.TZ}}
\newcommand{\htam}{\mathbf{H.T}}
\newcommand{\hne}{\mathbf{H.nE}}
\newcommand{\hone}{\mathbf{H.1}}

\newcommand{\hthree}{\mathbf{H.3}}

\newcommand{\hfive}{\mathbf{H.5}}

\newcommand{\hwd}{\mathbf{H.D}}
\newcommand{\hpss}{\mathbf{H.pS}}

\numberwithin{equation}{section}
\newtheorem{thm}{Theorem}[section]
\newtheorem{lemma}[thm]{Lemma}
\newenvironment{define}{\par\medskip\noindent\refstepcounter{thm}
\bgroup{\hspace*{-0.15 cm}\bf{Definition}
\thethm.}\bgroup}{\egroup \egroup\par\medskip}\newtheorem{prop}[thm]{Proposition}
\newtheorem{cor}[thm]{Corollary}
\newenvironment{rem}{\par\medskip\noindent\refstepcounter{thm}
\bgroup{\hspace*{-0.15 cm}\bf{Remark} \thethm.}\bgroup}{\egroup
\egroup\par\medskip} \parskip 2pt

\newenvironment{example}{\par\medskip\noindent\refstepcounter{thm}
\bgroup{\hspace*{-0.15 cm}\bf{Example}
\thethm.}\bgroup}{\egroup \egroup\par\medskip}

\newtheorem{conj}{Conjecture}

\newcounter{Athm}[section]

\renewcommand{\theAthm} {\arabic{Athm}}

\begin{document}
\title{{O}\lowercase{n} {E}\lowercase{uler systems of rank $r$ and their} {K}\lowercase{olyvagin systems}}

\author{K\^az\i m B\"uy\"ukboduk}

\address{K\^az\i m B\"uy\"ukboduk \hfill\break\indent Max Planck Institut f\"ur Mathematik  \hfill\break\indent Vivatsgasse 7 \hfill\break\indent 53111 Bonn
\hfill\break\indent Germany}
\curraddr{
\hfill\break\indent Ko\c{c} University, Mathematics \hfill\break\indent Rumeli Feneri Yolu \hfill\break\indent 34450 Sar\i yer / \.Istanbul
\hfill\break\indent Turkey}

\begin{abstract}
In this paper we set up a general Kolyvagin system machinery for
Euler systems of rank $r$ (in the sense of Perrin-Riou) associated
to a large class of Galois representations, building on our previous
work on Kolyvagin systems of Rubin-Stark units and generalizing the
results of Kato, Rubin and Perrin-Riou. Our machinery produces a
bound on the size of the classical Selmer group attached to a Galoýs
representation $T$ (that satisfies certain technical hypotheses) in
terms of a certain $r\times r$ determinant; a bound which remarkably
goes hand in hand with Bloch-Kato conjectures. At the end, we
present an application based on a conjecture of Perrin-Riou on
$p$-adic $L$-functions, which lends further evidence to Bloch-Kato
conjectures.

\end{abstract}

\maketitle
\tableofcontents
\section{Introduction}
\label{sec:intro}
Fix once and for all an odd prime $p$, a totally real number field $k$ and an algebraic closure $\overline{k}$ of $k$. Let
$\oo$ be the ring of integers of a finite extension $\Phi$ of $\QQ_p$ with $\mm$ being its maximal ideal
and $\mathbb{F}=\oo/\mm$ its residue field. Suppose $T$ is a geometric, continuous representation of $G_k=\textup{Gal}(\overline{k}/k)$ (i.e., a free $\oo$-module of finite rank which is endowed with a continuous action of $G_k$, unramified outside a finite set of primes of $k$ and De Rham at $p$). Set $V=T\otimes\Phi$ and $T^{\mathcal{D}}=\textup{Hom}_{\oo}(T,\oo)(1)$. Following Bloch and Kato~\cite{bk}, one may define Selmer groups $H^1_f(k,T)$ and $H^1_f(k,T^{\mathcal{D}})$ attached to $T$ and $T^{\mathcal{D}}$. When $V$ is the $p$-adic realization of a motive $\mathcal{M}$, one can also define an $L$-function $L(\mathcal{M},s)$ attached to $\mathcal{M}$, which is (conjecturally) the analytic counterpart of the Selmer groups $H^1_f(k,T)$ and $H^1_f(k,T^{\mathcal{D}})$. The exact relation between the special value of $L(\mathcal{M},s)$ at $s=0$ and the Selmer groups $H^1_f(k,T)$ and $H^1_f(k,T^{\mathcal{D}})$ is the subject of the Bloch-Kato conjecture, which is a vast generalization of, on the one hand, the Birch and Swinnerton-Dyer conjecture for elliptic curves and on the other  hand, the class number formulas. 

In this context, Euler system/Kolyvagin system machinery (as developed, following Kolyvagin's~\cite{kol} original ideas, by Kato~\cite{katoes}, Rubin~\cite{r00}, Perrin-Riou~\cite{pr-es} and later enhanced by Mazur and Rubin~\cite{mr02}) has been used to prove important results towards the Bloch-Kato conjecture in various different settings. For instance:
\begin{itemize}
\item[(i)] When  $k=\QQ$ and $T=\oo(1)\otimes\chi^{-1}$ (where $\chi: G_\QQ\ra \oo^\times$ is an even Dirichlet character and $\oo(1)=\oo\otimes_{\ZZ_p}\ZZ_p(1)$), then cyclotomic unit Euler system may be used to prove Gras' conjecture and a refined class number formula (see~\cite[\S III.2.1]{r00}).
\item[(ii)] When $k=\QQ$ and $T=T_p(E)$  is the $p$-adic Tate module of an elliptic curve $E_{/\QQ}$, then Kato~\cite{ka1} has constructed an Euler system in order to obtain strong evidence towards the Birch and Swinnerton-Dyer conjecture.
\end{itemize}
One common feature of the $G_k$-representations $T$ mentioned in the examples (i) and (ii) above is that
$$r=r(T):=\textup{dim}_{\Phi} (\textup{Ind}_{k/\QQ} V)^-=1,$$
where $(\textup{Ind}_{k/\QQ} V)^-$ is the $-1$-eigenspace for a
complex conjugation acting on $\textup{Ind}_{k/\QQ} V$. For a
general Galois representation $T$ for which $r(T)=1$, Mazur and
Rubin~\cite{mr02} developed an Euler system/Kolyvagin system
machinery so as to determine the structure of the relevant Selmer
group completely in terms of an Euler system\footnote{In his works
prior to~\cite{mr02}, Kolyvagin also describes the structure of the
Selmer group of an elliptic curve in terms of the relevant Euler
system, c.f., \cite[Theorems C and E]{Kol89}, \cite{kol} and
\cite[Theorem 1]{Kol91}.}. When $r(T)=1$, they prove that the module
of Kolyvagin systems is often cyclic (Theorem 5.2.10 of loc.cit.)
and it is therefore possible to choose `the best' Kolyvagin system
(which they call a \emph{primitive Kolyvagin system}) that may be
used to obtain the best possible bound on the associated Selmer
group. Furthermore, in the setting of the examples (i) and (ii)
above, a primitive Kolyvagin system is expected to be obtained from
an Euler system via Kolyvagin's descent (c.f.,~\cite[Theorem 3.2.4,
Remark 6.1.8 and Remark 6.2.5]{mr02}).

When $r>1$ the whole picture is more complicated. First of all, the Selmer groups in which the Kolyvagin system classes (that descended from an Euler system) live in are too large to be controlled by a single Kolyvagin system. Secondly, in contrast with~\cite[Theorem 5.2.10]{mr02}, the module of Kolyvagin systems is no longer cyclic. The purpose of this article is to overcome these difficulties and develop a satisfactory Euler system/Kolyvagin system machinery adapted to the case $r>1$. In this case, one should  start with an \emph{Euler system of rank $r$} in the sense of~\cite[Definition 1.2.2]{pr-es} (rather than an Euler system in the sense of~\cite[Definition II.1.1]{r00}), which in turn may be used as follows:
\begin{enumerate}
\item[(1)] Following the recipe of~\cite[\S1.2.3]{pr-es}, first obtain \emph{many} Euler systems (in the sense of~\cite[Definition II.1.1]{r00}; these correspond to \emph{Euler systems of rank $one$} in the terminology of~\cite[Definition 1.2.1]{pr-es}).
\item[(2)] Then apply Kolyvagin's descent on these Euler systems of rank one to obtain sufficiently many Kolyvagin systems which will control the Selmer groups in question.
\end{enumerate}

 As we remarked above, the basic technical problem we face is that the module of Kolyvagin systems is no longer cyclic\footnote{Howard~\cite[Appendix B]{mr02} shows that the $\mathbb{F}$-vector space of Kolyvagin systems for the residual representation $T/\mm T$ is infinite dimensional.} and the procedure above leaves us with many choices of Kolyvagin systems (whereas when $r=1$, there is a `best' Kolyvagin system).  We now briefly outline our method to tackle this issue. We first define a \emph{modified Selmer structure} (\S\ref{subsub:lock} and \S\ref{subsub:localconditionatpoverkinfty}), which is coarser than the Bloch-Kato Selmer structure but finer than what Mazur and Rubin call the canonical Selmer structure (see Example \ref{example:canonical selmer} and Definition~\ref{def:line} below). Our modified Selmer structure has the property that the module of Kolyvagin systems for it is a free $\oo$-module of rank one.  Next, we construct (\S\ref{subsec:choosehoms}) these Kolyvagin systems (for the modified Selmer structure we defined) from an Euler system of rank $r$ by refining the arguments of~\cite[\S1.2.3]{pr-es} and~\cite[Theorem 3.2.4]{mr02}. The point is that, the argument of Perrin-Riou~\cite[\S1.2.3]{pr-es} applied to an Euler system of rank $r$ does not a priori yield a Kolyvagin system for the modified Selmer structure, but only for the coarser canonical Selmer structure. Finally, we use these Kolyvagin systems for the modified Selmer structure to bound the modified Selmer group, which we combine with a global duality argument to obtain a bound for the Bloch-Kato Selmer groups in terms of an $r\times r$ determinant (\S\ref{sec:applications}).

Only in this paragraph, let $T$ denote the rank one $G_k$-representation $T=\oo(1)\otimes\chi^{-1}$, where $\chi$ is a totally even character $\chi:G_k \ra \oo^\times$ of finite prime-to-$p$ order.
In this case, 
 we have $r(T)=[k:\QQ]$. When $k\neq \QQ$, the machinery of~\cite{mr02} is not sufficient as it is to treat this example. An Euler system of rank $r=r(T)$ in this setting is obtained from the (conjectural) Rubin-Stark elements~\cite{ru96}. The author has studied this example extensively in~\cite{kbbstark,kbbiwasawa} and has developed a Kolyvagin system machinery in order to utilize this example of an Euler system of rank $r>1$ (which he used to prove, for example, Gras-type conjectures for totally real fields). Note that the representation $T=\oo(1)\otimes\chi^{-1}$ is \emph{totally odd} in the sense that
$$ \left(\textup{Ind}_{k/\QQ}T\right)^- =  \textup{Ind}_{k/\QQ}T.$$
This property is essential for the treatment of~\cite{kbbstark,kbbiwasawa,kbbstick}. In this article, we generalize the  methods of~\cite{kbbstark,kbbiwasawa,kbbstick} in order to develop an appropriate Kolyvagin system machinery for a Galois representations $T$ that satisfies a certain \emph{Pan\v{c}i\v{s}kin's Condition} (see hypotheses $\hpss$(b) and $\mathbf{{H}.P}$ in \S\ref{subsec:notationhypo}). Many important Galois representations fall in this category (c.f., Remark~\ref{rem:selfdualarepanciskin}):
\begin{enumerate}
\item $T=T_p(E)$ is the $p$-adic Tate module of an elliptic curve $E/k$; $k\neq \QQ$, 
\item $A/\QQ$ is an abelian variety of dimension $g>1$ and $T=T_p(A)$.
\end{enumerate}
(The two examples of $T$ above satisfy the hypothesis $\hpss$ below, and when $E$ (resp., $A$) has good ordinary reduction at all primes of $k$ above $p$ (resp., at $p$), they both satisfy $\mathbf{{H}.P}$.)

Before we state the main results of this paper, we fix our notation and set the hypotheses which we will refer to in the main body of our article.

\subsection{Notation and Hypotheses}
\label{subsec:notationhypo}
For any field $K$, let $G_K$ be the Galois group of a fixed separable closure of $K$. Throughout, $k$ is a fixed totally real number field and $k_\infty$ is the cyclotomic $\ZZ_p$-extension of $k$. Set $\Gamma=\Gal(k_\infty/k)$.  We write $k_n$ for the unique sub-extension of $k_\infty/k$ of degree $p^n$, and set $\Gamma_n=\Gal(k_n/k)$. Our first hypothesis which we will assume for our Iwasawa theoretic results is the following:

$(\hiwasawa)$ Every prime $\wp$ of  $k$ above $p$ totally ramifies in $k_\infty/k$.

For any prime $\lambda $ of $ k$, we fix a decomposition group
$\mathcal{D}_{\lambda} \subset G_k$. We will occasionally identify
$\mathcal{D}_\lambda$ with the absolute Galois group of the
completion $k_\lambda$. We denote the inertia subgroup inside
$\mathcal{D}_\lambda$ by $\mathcal{I}_\lambda$. We write
$\textup{Fr}_\lambda \in \mathcal{D}_\lambda/\mathcal{I}_{\lambda}$
for the arithmetic Frobenius element.

Let $\oo$ be the ring of integers of a finite extension $\Phi$ of $\QQ_p$ with $\mm$ being its maximal ideal and $\mathbb{F}=\oo/\mm$ its residue field. We define  $\LL=\oo[[\Gamma]]$ to be the cyclotomic Iwasawa algebra. Write $\pmb{\mu}_{p^n}$ for the (Galois module of) $p^n$-th roots of unity, and set $\ZZ_p(1)=\varprojlim \pmb{\mu}_{p^n}$ and $\pmb{\mu}_{p^\infty}=\varinjlim  \pmb{\mu}_{p^n}$. We define $\oo(1):=\oo \otimes_{\ZZ_p}\ZZ_p(1)$, and for any $\oo[[G_k]]$-module $M$, we write $M(1):=M\otimes_{\oo}\oo(1)$ (allowing $G_k$ act both on $M$ and $\oo(1)$). We also define $M^*=\Hom(M,\Phi/\oo)(1)$, the Cartier dual of $M$; and $M^{\vee}=\Hom(M,\Phi/\oo)$, the Pontryagin dual of $M$; and $M^{\mathcal{D}}=\Hom_{\oo}(M,\oo)(1)$. 

For any field $K$ and a topological abelian group $A$ which is
endowed with a continuous action of $G_K$, we write $H^i(K,A)$ for
the $i$-th group cohomology $H^i(G_K,A)$ computed with continuous
cochains. We also define
$$A^{\wedge}:=\Hom(\Hom(A,\QQ_p/\ZZ_p),\QQ_p/\ZZ_p)$$ (with continuous homomorphisms, when $A$ is as above) to be the $p$-adic completion of $A$.

For any commutative ring $R$, an ideal $I \subset R$ and an $R$-module $A$, we write $A[I]$ for the submodule of $A$ consisting of elements that are killed by all $I$. For $x\in R$, we write $A[x]$ for $A[Rx]$.

Let $T$ be a free $\oo$-module of finite rank, endowed with a continuous action of $G_k$. 
Let 
$$d=\textup{rank}_\oo(\textup{Ind}_{k/\QQ}T)=[k:\QQ]\cdot \textup{rank}_\oo\,T,$$
$$d_+=\textup{rank}_\oo(\textup{Ind}_{k/\QQ}T)^+=\sum_{v|\infty}\textup{rank}_{\oo}\,H^0(k_v,T),$$
 where $k_v$ stands for the completion of $k$ at the infinite place $v$. 
 We define $$r=d-d_+=\textup{rank}_\oo(\textup{Ind}_{k/\QQ}T)^-.$$ Note that $r$ above is exactly what Perrin-Riou~\cite{pr-es} calls $d_-$.

Write $ \TT=T\otimes_{\oo}\LL$, where we allow $G_k$ act on both $T$
and $\LL$. (The action of $G_k$ on $\LL$ is induced from the
canonical surjection $G_k\twoheadrightarrow \Gamma$.) Define
$V=T\otimes_{\oo} \Phi$, and $V^*=\Hom(V,\Phi)(1)$.

Fix a set $\PP$ of (non-archimedean) primes of $k$ which does \emph{not} contain any prime above $p$ and any prime at which $T$ is ramified. Following~\cite[Definition 3.1.6]{mr02}, we define $\PP_s$ ($s \in \ZZ^+$) as the set of primes $\lambda$ of $k$ at which $T$ is unramified, which do not lie above $p$ and which satisfy:
\begin{enumerate}
\item $T/(\mm^s T + (\textup{Fr}_{\lambda}-1)T)$ is a free $\oo/\mm^s$-module of rank one,
\item $I_\lambda:=\textup{span}_{\oo}\left\{\mathbf{N}\lambda-1, \textup{det}(1-\textup{Fr}_\lambda|T)\right\} \subset \mm^s$.

\end{enumerate}

For any group $\Delta$, and a $\oo[\Delta]$-module $M$, we write $\wedge^s M$ for the $s$th exterior power of $M$ computed in the category of $\oo[\Delta]$-modules. For example, we will be dealing below  with exterior powers of the form $\wedge^s H^i(K,M)$, where $K$ is a finite extension of $k$ with Galois group $\Delta$, and $M$ is an $\oo[[G_k]]$-module. This naturally  makes $H^i(K,M)$ an $\oo[\Delta]$-module and $\wedge^s H^i(K,M)$ is calculated in the category of $\oo[\Delta]$-modules.

Let $B_{\textup{cris}}$ (resp., $B_{\textup{dR}}$) be Fontaine's crystalline (resp., de Rham) period ring; see~\ref{subsec:PRpadic} below and references therein for further details on the basics of $p$-adic Hodge theory. For a finite extension $K$ of $\QQ_p$ and a $\QQ_p[[G_K]]$-module $W$ of finite $\QQ_p$-dimension, we define 
\begin{itemize}
\item $H^1_f(K,W)=\ker (H^1(K,W)\lra H^1(K,B_{\textup{cris}}\otimes W)),$
\item $D_{\textup{dR}}(W)=D_{\textup{dR}}(K, W)=H^0(K,W\otimes_{\QQ_p} B_{\textup{dR}})$ to be the filtered $K$-vector space associated to $W$, with associated filtration $\dots\subset D_{\textup{dR}}^{i+1}(W)\subset D_{\textup{dR}}^i(W) \subset \dots $  
\end{itemize}
We set $H^1_f(K,T)=\ker (H^1(K,T)\ra H^1(K,B_{\textup{cris}}\otimes V))$.

Below we record a list of properties which will play a role in what follows:

\begin{itemize}
\item[($\mathbf{H.1}$)] $T/\mm T$ is an absolutely irreducible $\mathbb{F} [[G_k]]$-representation.
\item[($\mathbf{H.2}$)] There is a $\tau \in G_k$ such that $\tau=1 \hbox{ on } \pmb{\mu} _{p ^{\infty}}$ and the $\oo$-module $T/(\tau -1)T$ is free of rank one.
\item[($\mathbf{H.3}$)] $H^{1}(k(T  , \pmb{\mu} _{p ^{\infty}})  , T/\mm T)$ = $H^{1}(k (T ,\pmb{\mu} _{p ^{\infty}})  , T ^{*}[\mm])$ =  0, where $k(T, \pmb{\mu} _{p ^{\infty}})=k(T)(\pmb{\mu} _{p ^{\infty}}) \subset \overline{k}$, and $k(T)$ is the smallest extension of $k$ such that  the $G_k$-action on $T$ factors through $\textup{Gal}(k(T)/k)$.

\item[($\mathbf{H.4}$)] $p>4$.
\item[($\mathbf{H.5}$)] The set of primes $\PP$ satisfies $\PP_t \subset \PP \subset \PP_1$ for some $t \in \ZZ^+$.
\item[($\mathbf{{H}.T}$)] (\emph{Tamagawa Condition}) $(T\otimes \Phi/\oo) ^{\mathcal{I}_{\lambda}}$ is $\oo$-divisible for any prime $\lambda$ of $k$ prime to $p$.
\item[($\mathbf{{H}.nE}$)] (\emph{Non-exceptionality}) $H^0(k_p,T^*):=\oplus_{\wp|p}H^0(k_\wp,T^*)=0.$
\item[($\hwd$)](\emph{A condition on `denominators'}) $H^0(k_{\wp,\infty},T)=0$ for every $\wp|p$.
\item[($\hpss$)] \begin{itemize}
\item[(a)]The representation $V$ is potentially semistable (in the sense of~\cite[\S I.2]{FPR91}) at any place $\wp$ dividing $p$.
\item[(b)] ${\displaystyle \textup{rank}_{\oo}\, \left(\displaystyle{\oplus_{\wp|p}} \,H^1_f(k_\wp,T)\right) = d_+.}$
\end{itemize}
\item[($\mathbf{{H}.P}$)] (\emph{strong Pan\v{c}i\v{s}kin's Condition}) The Galois representation $T$ satisfies a strong Pan\v{c}i\v{s}kin Condition at all primes $\wp \subset k$ above $p$ in the following sense: 
\begin{itemize}
\item[(i)] There exists an exact sequence of $\oo[G_{k_\wp}]$-modules (that are free as $\oo$-modules)
$$0\lra\textup{F}_\wp^+T\lra T \lra \textup{F}_\wp^-T \lra 0$$  
such that $\textup{F}_\wp^{\pm}V:=\Phi\otimes_{\oo}\textup{F}_\wp^\pm T$ are potentially semistable, and 
$$D^0_{\textup{dR}}(\textup{F}_\wp^{+}V)=0=D_{\textup{dR}}(\textup{F}_\wp^{-}V)/D^0_{\textup{dR}}(\textup{F}_\wp^{-}V).$$  
\item[(ii)] $\displaystyle \sum_{\wp|p}\, [k_\wp:\QQ_p]\cdot \textup{rank}_{\oo}\, \textup {F}_\wp^+T = d_+.$
\end{itemize}
We set $\textup{F}_\wp^+\TT:= \textup {F} _\wp ^+T\otimes_{\oo}\LL\,\,\,;\,\,\, \textup{F} _\wp ^-\TT:=\TT/\textup {F} _\wp ^+\TT=\textup{F} _\wp ^-T\otimes_{\oo}\LL.$
\end{itemize}
\begin{rem}
 \label{rem:hpsb}
 Suppose that $H^0(k_p,T)=0$ (which follows from $\hwd$, when assumed) and suppose $\hpss$(a) holds. It follows from~\cite[Corollary 3.8.4]{bk} that $\hpss$(b) is equivalent to the assertion that 
 $$\sum_{\wp|p} \dim_{\QQ_p} D_{\textup{dR}}(k_\wp,V)/D_{\textup{dR}}^0(k_\wp,V)=d_+.$$ 
 \end{rem}

 \begin{rem}
 \label{rem:selfdualarepanciskin}
 In this paragraph, let $T$ be a self-dual $G_k$-representation, in the sense that there is a skew-symetric isomorphism
  \be\label{eqn:defselfdual}T\stackrel{\sim}{\lra} T^{\mathcal{D}}= \Hom_{\oo}(T,\oo)(1).\ee
  When $T$ is self-dual, $d$ is even and $d_+=d/2$.
  The aim of this Remark is to explain how the hypotheses $\hpss$(b) and $\mathbf{{H}.P}$(ii) may be verified for self-dual Galois representations.
  \begin{itemize}
  \item[(a)] In this case, $\hpss$(b) follows from $\hpss$(a) under the hypothesis $\mathbf{{H}.nE}$, using Lemma \ref{lemma:free for k}, as the Bloch-Kato subgroups $H^1_f(k_\wp,V)$ and $H^1_f(k_\wp,V^*)$ are orthogonal compliments of each other under the local Tate pairing (see~\cite[Proposition 3.8]{bk}). 
  \item[(b)] Assume now that the $G_k$ is \emph{nearly ordinary} at every prime $\wp$ of $k$ above $p$, in the sense that there is an exact sequence of $\oo[G_{k_\wp}]$-modules (which are free as $\oo$-modules)
  $$0\lra\textup{F}_\wp^+T\lra T \lra \textup{F}_\wp^-T \lra 0,$$
  such that,  
  $$\textup{F}_\wp^{\pm}T \stackrel{\sim}{\lra} \Hom_{\oo}(\textup{F}_\wp^\mp T,\oo)(1)$$
as $G_{k_\wp}$-modules (under the isomorphism induced from~(\ref{eqn:defselfdual}) above). In this case, one may check easily that $T$ satisfies $\mathbf{{H}.P}$(ii) above. 
  \end{itemize}
 \end{rem}

 \begin{rem}
 \label{rem:onpanciskin}
 In this Remark, we record further properties of a representation $T$ that satisfies the hypothesis $\mathbf{{H}.P}$(i).
 \begin{itemize}
 \item[(i)] The modules $\textup{F}_\wp^\pm T$ are uniquely determined (\cite[6.7]{nekheights1}).
 \item[(ii)] The $G_{k_\wp}$-representation $V$ is potentially semistable (\cite[1.28]{nekheights1}).  
 \item[(iii)] The dual representation $T^{\mathcal{D}}=\Hom_{\oo}(T,\oo)(1)$ also satisfies Pan\v{c}i\v{s}kin's condition, with $\textup{F}_\wp^\pm T^{\mathcal{D}}=(\textup{F}_\wp^\mp T)^{\mathcal{D}}$.
 \item[(iv)] $H^0(k_\wp, \textup{F}_\wp^- V)=D_{\textup{cris}}(\textup{F}_\wp^- V)^{\varphi=1}$, where $\varphi$ is the crystalline Frobenius (\cite[Proposition 3.3.2(1)]{nekheights2}). 
 \end{itemize}
 \end{rem}
  The final hypothesis we record here is:
 \begin{itemize}
 \item[($\mathbf{H.TZ}$)] (\emph{Trivial zero condition}) Under $\mathbf{{H.P}}$(i),  
$$\bigoplus_{\wp|p}H^0(k_\wp, \textup{F}_\wp^\pm T^{\mathcal{D}}\otimes \Phi/\oo)=0=\bigoplus_{\wp|p} H^0(k_\wp, \textup{F}_\wp^\pm T\otimes\Phi/\oo)$$
\end{itemize}
We will not need the truth of \emph{all} of these hypotheses which we recorded above for all of our results. We will carefully state which of these hypotheses are in effect before stating each claim.

Let us continue to comment on these hypotheses. The hypotheses  $\mathbf{H.1}$-$\mathbf{5}$ are already present in~\cite[\S3.5]{mr02}. A variant of the hypotheses $\mathbf{H.6}$ of loc.cit. will appear shortly (in fact, we will show that it holds for the cases of interest, as long as we assume $\mathbf{{H}.T}$, $\mathbf{{H}.nE}$ and $\mathbf{H.TZ}$ for the Iwasawa theoretic results). Hypothesis $\hpss$(a) is needed in this work to ensure that the Bloch-Kato local conditions for $V$ and $V^*=\textup{Hom}(V,\Phi(1))$ are orthogonal compliments of each other. It is equivalent to asking that $V$ is de Rham at $p$ (thanks to the work of Berger~\cite{berger}), and any Galois representation `coming from geometry' is de Rham at $p$. For applications, $\hpss$(a) should therefore pose no condition. Hypothesis $\mathbf{{H}.T}$ is closely related to the assertion that the Tamagawa factor at $\lambda$ is prime to $p$. This assumption is only required to prove~Theorem \ref{prop:modifedKSinftyrk1}, i.e., the existence of $\LL$-adic Kolyvagin systems. The hypotheses $\mathbf{{H}.P}$ and $\mathbf{H.TZ}$ are required for the Iwasawa theoretic applications, and  $\mathbf{H.TZ}$ ensures that the associated (conjectural) $p$-adic $L$-function has no trivial zeros at the characters of finite order of $\Gamma$. 
 Hypotheses $\mathbf{{H}.nE}$ and $\hwd$ are present so that the
structure of the local cohomology groups is `nice' (c.f.,
\S\ref{sub:locp}). $\hwd$ has to do with the denominators of
$L$-functions (c.f., \cite[Remark VII.2.5]{r00}) and
$\mathbf{{H}.nE}$ is a generalization of Mazur's~\cite{mazur-anom}
notion of $p$ being a \emph{non-anomalous prime for an abelian
variety} $A$: $\mathbf{{H}.nE}$ for the $p$-adic Tate module
$T=T_p(A)$ of an abelian variety $A_{/\QQ}$ is the condition that
$p$ is not an anomalous prime for $A$. It is expected that
$\mathbf{{H}.nE}$ and  $\hwd$ are often (but  not always) true, for
instance, when $T=T_p(A)$ is as above, Mazur~\cite{mazur-anom}
explains that anomalous primes should be sparse.
 
 Before we go on with the statements of our main results, we record the following facts which may be used to compare the `Bloch-Kato local conditions' to `Greenberg local conditions':  
\begin{rem}
\label{rem:greenbergvsBK}
\begin{itemize}
\item[(i)] If one assumes $\mathbf{{H}.P}$(i) and $H^0(k_p,T)=0$, then the conditions $\hpss$(b) and $\mathbf{{H}.P}$(ii) are equivalent to each other: It follows from~\cite[Corollary 3.8.4]{bk} that the $\QQ_p$-vector spaces $H^1_f(k_\wp,V)$ and $D_{\textup{dR}}(k_\wp,V)/D^0_{\textup{dR}}(k_\wp,V)=D_{\textup{dR}}(k_\wp,\textup{F}_\wp^+V)$ have the same dimension, and the dimension of the latter is equal to $[k_\wp:\QQ_p]\,\textup{rank}_{\oo}\textup{F}_\wp^+T$. 

\item[(ii)] The following facts are proved in~\cite[1.32]{nekheights1} and \cite[3.3.2]{nekheights2}: Suppose $K$ is a finite extension of $\QQ_p$ and $V$ is a $\Phi$-vector space equipped with a continuous $G_K$ action, that satisfies Pan\v{c}i\v{s}kin's condition. If $K^\prime/K$ is an extension over which $V$ becomes semistable and is such that 
$$D_{\textup{cris}}(V\mid_{G_{K^\prime}})^{\varphi=1}=D_{\textup{cris}}(V^*\mid_{G_{K^\prime}})^{\varphi=1}=0,$$  
then there is an exact sequence
$$0\lra H^0(K, \textup{F}_\wp^-V)\lra H^1(K,\textup{F}_\wp^+V)\lra H^1_f(K,V)\lra 0.$$
If in addition we have 
\be
\label{eqn:trivialzero}
H^0(K, \textup{F}_\wp^-V)=0,
\ee
then $H^1(K,\textup{F}_\wp^+V)\stackrel{\sim}{\lra}H^1_f(K,V)$. 
\end{itemize}
\end{rem} 
\subsection{Statements of the Main Results}
For a Galois representation $T$ as above, assume that the hypotheses $\hone$-$\hfive$, $\hne$ and $\hwd$ hold true. Suppose that $\mathbf{c}^{(r)}=\{c^{(r)}_K\}$ is an Euler system of rank $r$, in the sense of Definition~\ref{def:weaklyintegralESofrankr} below. For any number field $F$, define
$$\textup{loc}_p^s: H^1(F,T)\lra H^1_s(F_p,T)$$
by fixing embeddings $\iota_\wp: \overline{F} \hookrightarrow \overline{F_\wp}$ for every $\wp$ above $p$. Here $H^1_s(F_p,T)$ is the \emph{singular quotient} $H^1(F_p,T)/H^1_f(F_p,T)$ (see~\S\ref{subsub:lock}) of $H^1(F_p,T)$. We write $\textup{loc}_p^s$ also for the induced map $\wedge^r H^1(F,T)\ra\wedge^r H^1_s(F_p,T).$
Let $H^1_{\FF_{\textup{BK}}^*}(k,T^*)$ denote the Bloch-Kato Selmer group attached to $T^*$ (see~\S\ref{subsubsec:compareselmer1}).
\begin{thma}[Corollary~\ref{cor:maink}]
In addition to the hypotheses above, suppose that ${\hpss}$ holds for $T$. Then
$$\#H^1_{\FF_{\textup{BK}}^*}(k,T^*)\leq [\wedge^r H^1(k,T): \oo\cdot\textup{loc}_p^s(c^{(r)}_k)].$$
\end{thma}
See Theorem~\ref{thm:mainkinfty}  and Theorem~\ref{thm:cotorsion} below for our Iwasawa theoretic main result, which proves that the characteristic ideal of an appropriately defined Greenberg Selmer group divides the characteristic ideal of a certain $\LL$-module determined by the Euler system $\mathbf{c}^{(r)}$.

We illustrate one concrete application of our technical results, which relies on Perrin-Riou's conjectures~\cite{pr-ast} on $p$-adic $L$-functions (see Conjectures~\ref{conj:prpadicL} and \ref{conj:specialelts} below). Suppose that $V=T\otimes\Phi$ is the $p$-adic realization of a pure, self-dual motive $\MM$ defined over $k$. Assume in addition that $V$ is crystalline at $p$, and that $1$ is {not} an eigenvalue for the Frobenius acting on $D_{\textup{cris}}(V)$. Let $L(\MM,s)$ denote the complex $L$-function associated to $\mathcal{M}$.
\begin{thmb}[Theorem~\ref{thm:mainconjappl}]
Assume Conjectures~\ref{conj:prpadicL} and \ref{conj:specialelts} of Perrin-Riou, as stated in \S\ref{subsubsec:PRlog} and \S\ref{subsubsec:PRpadicES} below. If $L(\MM,0)\neq 0$, then the Bloch-Kato Selmer group $H^1_{\FF_{\BK}^*}(k,T^*)$ is finite.
\end{thmb}
Furthermore, the proof of~Theorem~\ref{thm:mainconjappl} gives a bound on the size $H^1_{\FF_{\BK}^*}(k,T^*)$ that is explicitly related to the $L$-value, which goes hand in hand with the Bloch-Kato conjectures.

Although the existence of the Euler system of rank $r$ which is used to prove Corollary~\ref{cor:maink}, Theorems~\ref{thm:mainkinfty} and~\ref{thm:mainconjappl}  is conjectural, the existence of the derived classes (which play an essential role in the proofs) is \emph{not}, thanks to~\cite[Theorem 5.2.10]{mr02} (resp.,~\cite[Theorem 3.23]{kbb} in the Iwasawa theoretic setting). Note however that the Kolyvagin system classes whose existence is proved abstractly have a priori no link to $L$-values. This is one reason why one desires to descend from an Euler system of rank $r$ (which is conjecturally related to $L$-values yet we do not know if it exists) to a Kolyvagin system (resp., to a $\LL$-adic Kolyvagin system), whose existence was proved by Howard, Mazur and Rubin (resp., by the author in the Iwasawa theoretic setting).
 \section{Preliminaries: Local conditions and Selmer groups}
 \subsection{Selmer structures on $T$}
The notation that we have set above is in effect.

We first recall Mazur and Rubin's definition of a \emph{Selmer structure}, in particular the \emph{canonical Selmer structure} on $T$ and $\TT$.

 Let $R$ be a complete local noetherian $\oo$-algebra, and let $M$ be a $R[[G_k]]$-module which is free of finite rank over $R$. In this paper, we will be interested in the case when $R=\LL$ or its certain quotients, and $M$ is $\TT$ or its relevant quotients by an ideal of $\LL$. (For example, taking the quotient by the augmentation ideal of $\LL$ will give us $\oo$ and the representation $T$.)

\subsubsection{Selmer structures and Selmer groups}
\label{sec:selmer structure}
Notation from \S\ref{subsec:notationhypo} is in effect in this section.
\begin{define}
\label{selmer structure}
A \emph{Selmer structure} $\FF$ on $M$ is a collection of the following data:
\begin{itemize}
\item a finite set $\Sigma(\FF)$ of places of $k$, including all infinite places and primes above $p$, and all primes where $M$ is ramified.
\item for every $\lambda \in \Sigma(\FF)$ a local condition on $M$ (which we view now as a $R[[\mathcal{D}_{\lambda}]]$-module), i.e., a choice of $R$-submodule $$H^1_{\FF}(k_{\lambda},M) \subset H^1(k_{\lambda},M).$$
 \end{itemize}
If $\lambda \notin \Sigma(\FF)$ we will also write
$H^1_{\FF}(k_{\lambda},M)=H^1_{f}(k_{\lambda},M)$, where the module
$H^1_{{f}}(k_{\lambda},M)$ is the \emph{finite} part of
$H^1(k_{\lambda},M)$, defined as in~\cite[Definition 1.1.6]{mr02}.
\end{define}

\begin{define}
 The \emph{semi-local cohomology group} at a rational
prime $\ell$ is defined by setting
$$H^i(k_\ell,M):=\bigoplus_{\lambda|\ell} H^i(k_\lambda,M),$$
where the direct sum is over all primes $\lambda$ of $k$ lying above
$\ell$.
\end{define}
 Let $\lambda$ be a finite prime of $k$. There is the perfect local
Tate pairing $$<\,,\,>_\lambda\,:H^1(k_\lambda,M) \times
H^1(k_\lambda,M^*) \lra H^2(k_\lambda,\Phi/\oo(1))
\stackrel{\sim}{\lra}\Phi/\oo,$$ where we recall that $M^*$ stands
for the Cartier dual of $M$. For a Selmer structure $\FF$ on $M$,
define $H^1_{\FF^*}(k_\lambda,M^*):=H^1_\FF(k_\lambda,M)^\perp$ as
the orthogonal complement of $H^1_\FF(k_\lambda,M)$ with respect to
the local Tate pairing. The Selmer structure $\FF^*$ on $M^*$ (with
$\Sigma(\FF)=\Sigma(\FF^*)$) defined in this way will be called the
\emph{dual Selmer structure}.

For examples of local conditions see~\cite[Definitions 1.1.6 and 3.2.1]{mr02}.
\begin{define}
\label{selmer group}
If $\FF$ is a Selmer structure on $M$, we define the \emph{Selmer module} $H^1_{\FF}(k,M)$ as
 $$H^1_{\FF}(k,M):=\ker\left(H^1(\textup{Gal}(k_{\Sigma(\FF)}/k),M) \lra \bigoplus_{\lambda \in \Sigma(\FF)}H^1(k_{\lambda},M)/H^1_{\FF}(k_{\lambda},M)\right),$$
 where $k_{\Sigma(\FF)}$ is the maximal extension of $k$ which is unramified outside $\Sigma(\FF)$. We also define the dual Selmer structure in a similar fashion; just replace $M$ by $M^*$ and $\FF$ by $\FF^*$ above.
\end{define}

\begin{example}
\label{example:canonical selmer}
In this example we recall~\cite[Definitions 3.2.1 and  5.3.2]{mr02} which we use quite frequently in this paper.
\begin{itemize}
\item[(i)] Let $R=\oo$ and let ${M}$ be a free $R$-module endowed with a continuous action of $G_k$, which is unramified outside a finite set of places of $k$.  We define a Selmer structure $\FFc$ on ${M}$ by setting $\Sigma(\FFc)=\{\lambda: M \hbox{ is ramified at } \lambda\}\cup\{\wp|p\}\cup\{v|\infty\}$, and
\begin{itemize}
\item if $\lambda \in \Sigma(\FFc)$, $\lambda\nmid p\infty$, we define the local condition at $\lambda$ to be
$$H^1_{\FFc}(k_\lambda, {M})=\ker(H^1(k_\lambda, {M}) \lra H^1(k_\lambda^{\textup{unr}},M\otimes \Phi)),$$
where $k_\lambda^{\textup{unr}}$ is the maximal unramified extension of $k_\lambda$,
\item if $\wp|p$, we define the local condition at $\wp$ to be
$$H^1_{\FFc}(k_\wp, {M})=H^1(k_\wp,M).$$
\end{itemize}
The Selmer structure $\FFc$ is called the \emph{canonical Selmer structure} on ${M}$.
\item[(ii)] Let now $R=\LL$ be the cyclotomic Iwasawa algebra, and let $\mathbb{M}$ be a free $R$-module endowed with a continuous action of $G_k$, which is unramified outside a finite set of places of $k$.  We define a Selmer structure $\FF_\LL$ on $\mathbb {M}$ by setting $$\Sigma(\FF_\LL)=\{\lambda: \mathbb{M} \hbox{ is ramified at } \lambda\}\cup\{\wp\subset k: \wp|p\}\cup\{v|\infty\},$$
 and $H^1_{\FF_\LL}(k_\lambda, \mathbb {M})=H^1(k_\lambda, \mathbb{M})$ for every $\lambda \in \Sigma(\FF_\LL)$. The Selmer structure $\FF_{\Lambda}$ is called the \emph{canonical Selmer structure} on $\mathbb{M}$.
 \end{itemize}
We still denote the induced Selmer structure on the quotients
$\mathbb{M}/I\mathbb{M}$ by $\FF_\LL$, which is obtained by
\emph{propagating} $\FF_\LL$ on $\mathbb{M}$ (see~\cite[Example
1.1.2]{mr02}). Note for $\lambda\in \FF_{\LL}$ that
$H^1_{\FF_\LL}(k_{\lambda}, \mathbb{M}/I\mathbb{M})$ will not always
be the same as $H^1(k_{\lambda}, \mathbb{M}/I\mathbb{M})$. In
particular, when $I$ is the augmentation ideal inside $\Lambda$, the
Selmer structure $\FF_\LL$ on $\mathbb{M}$ will not always propagate
to $\FFc$ on $M:=\mathbb{M}\otimes_{\LL}\LL/I.$ However, when $M=T$
and $\mathbb{T}=T\otimes_{\oo}\LL$ as in~\ref{sec:intro}, $\FF_\LL$
on $\TT$ \emph{does} propagate  to $\FFc$ on $T$, under the
hypotheses $\htam$ and $\hne$.

\end{example}
\begin{rem}
\label{rem:canonical selmer} We say that an element $f\in \LL$ is
distinguished if $\LL/(f)$ is a free $\oo$-module of finite rank.
When $R=\LL$ and $\ \mathbb{T}=T\otimes_{\oo}\LL$ (which is one of
the cases of interest), the Selmer structure $\FFc$ defined
in~\cite[ \S2.1]{kbb} on the quotients $T\otimes_{\oo}\LL/(f)$ may
be identified, under the hypotheses $\htam$ and $\hne$, by the
propagation of $\FF_\LL$  to the quotients $T\otimes_{\oo}\LL/(f)$,
for every distinguished $f\in \LL$. Indeed, for every prime
$\lambda\subset k$, the submodule
$$H^1_{\FFc}(k_\lambda,T\otimes_{\oo}\LL/(f))\subset
H^1(k_\lambda,T\otimes_{\oo}\LL/(f))$$ is the image of the canonical
map $H^1(k_\lambda,T\otimes_{\oo}\LL) \ra
H^1(k_\lambda,T\otimes_{\oo}\LL/(f))$, by the proofs
of~\cite[Proposition 2.10 and 2.12]{kbb}. By definition,
$H^1_{\FF_\LL}(k_\lambda,T\otimes_{\oo}\LL/(f))$ is exactly the same
thing.
\end{rem}

\begin{define}
\label{def:selmer triple}
A \emph{Selmer triple} is a triple $(M,\FF,\PP)$, where $\FF$ is a Selmer structure on $M$ and $\PP$ is a set of primes as in~\S\ref{subsec:notationhypo}, i.e., a set of non-archimedean primes of $k$ disjoint from $\Sigma(\FF)$.
\end{define}

 \subsection{Modifying local conditions at $p$}
\label{sub:locp} When the core Selmer rank of a Selmer structure (in
the sense of~\cite[Definition 4.1.11]{mr02}, see also
\S\ref{subsec:KS1} below) is greater than one, it produces a Selmer
group which is difficult to control using the Kolyvagin system
machinery of Mazur and Rubin. We will see in~\S\ref{subsec:KS1} that
$\FFc$ on $T$ ({resp.}, $\FF_{\LL}$ on $\TT=T\otimes_{\oo}\LL$) has
core Selmer rank $r:=d-d_+$ (under the hypotheses $\hne$).
Hence, to be able to utilize the Kolyvagin system machinery of Mazur
and Rubin, we need to modify the Selmer structures $\FFc$ and
$\FF_\LL$ appropriately. This is what we do in this section.

\subsubsection{Local conditions at $p$ over $k$}
\label{subsub:lock}
 \begin{lemma}
 \label{lemma:free for k}
If ${\hne}$ (both for $T$ and $T^\mathcal{D}$) and ${\hwd}$ hold, then the $\oo$-module $H^1(k_p,T):=\oplus_{\wp|p}H^1(k_\wp,T)$
 is free  of rank $d$.
 \end{lemma}

 \begin{proof}
 We start with the remark that, thanks to ${\hne}$ for $T^\mathcal{D}$ we have $H^0(k_p,T\otimes\Phi/\oo)=H^1(k_p,T)_{\textup{tors}}=0$, and thus the $\oo$-module $H^1(k_p,T)$ is free.
 
All the references below are to~\cite[Appendix]{kbbiwasawa} and the results quoted here are originally due to Perrin-Riou.

By Theorem A.8(i),  $\LL$-torsion submodule $H^1(k_p,\TT)_{\textup{tors}}$ is isomorphic to $\oplus_{\wp|p}T^{H_{k_\wp}}$, where $H_{k_\wp}=\Gal(\overline{k_\wp}/k_{\wp,\infty})$, and this module is trivial thanks to $\hwd$. Theorem A.8(ii) now concludes that  the $\LL$-module  $H^1(k_p,\TT)$ is free rank $d$. Furthermore,
$$\textup{coker}[H^1(k_p,\TT)\lra H^1(k_p,T)]=H^2(k_p,\TT)[\gamma-1],$$
 where $\gamma$ is any topological generator of $\Gamma$. Since we assumed $\hne$\, holds, it follows from~\cite[Lemma 2.11]{kbb} that $H^2(k_p,\TT)=0,$ hence the map
 $$H^1(k_p,\TT)\lra H^1(k_p,T)$$ is surjective. Lemma now follows.
 \end{proof}
\begin{rem}
\label{ref:referee} Let $R$ be any complete local noetherian ring
with a finite residue field, and let $M$ be a free $R$-module of
finite rank which is endowed with a continuous action of
$G_{k_{\wp}}$. Suppose that
$H^2(k_\wp,M)=0=H^2(k_\wp,\textup{Hom}_{R}(M,R)(1))=0$. The
anonymous referee has kindly pointed out that the freeness of
$H^1(k_{\wp},M)$ in this very general setting may be deduced
following~\cite[Prop. 4.2.9]{nekovar06}: The cohomology
$H^{\bullet}(k_\wp,M)$ is represented by a perfect complex of
$R$-modules (i.e., projective, hence free, $R$-modules of finite
type) concentrated in degrees $0,1$ and $2$. In particular, since we
assume that $H^2(k_\wp,M)=0$, then this complex may be taken in
degrees $0$ and $1$. Similarly, the cohomology
$H^{\bullet}(k_\wp,\textup{Hom}_{R}(M,R)(1))$ is represented by a
perfect complex of $R$-modules concentrated in degrees $0$ and $1$.
When the coefficient ring $R$ is Gorenstein, the two complexes
$H^{\bullet}(k_\wp,M)$ and
$H^{\bullet}(k_\wp,\textup{Hom}_{R}(M,R)(1))$ are related by the
duality functor $\textup{RHom}_{R}(-,R)[-2]$ (c.f., \cite[Prop.
5.2.4]{nekovar06}). As a result, each of these two complexes is also
represented by a perfect complex concentrated in degrees $2-1=1$ and
$2-0=2$, hence by a single projective (hence free) $R$-module of
finite type in degree $1$.
\end{rem}
 Bloch and Kato~\cite[\S3]{bk} define a subspace $H^1_f(k_\wp,V)\subset H^1(k_\wp,V)$ by letting
 $$H^1_f(k_\wp,V):=\ker\left(H^1(k_\wp,V) \lra H^1(k_\wp,V\otimes B_{\textup{cris}}) \right),$$
where $B_{\textup{cris}}$ is Fontaine's crystalline period ring. We propagate the \emph{Bloch-Kato local condition} $H^1_f(k_\wp,V)$ on $V$ to $T$:
  \begin{align*}
 H^1_f(k_\wp,T)&:=\ker\left(H^1(k_\wp,T)\lra \frac{H^1(k_\wp,V)}{H^1_f(k_\wp,V)}\right)\\&=\ker\left(H^1(k_\wp,T) \lra H^1(k_\wp,V\otimes B_{\textup{cris}})\right)
 \end{align*}
 We define the singular quotient as $H^1_s(k_\wp,T):=H^1(k_\wp,T)/H^1_f(k_\wp,T)$. Note that $H^1_s(k_\wp,T)$ is a free $\oo$-module as it injects, by definition, into $H^1(k_\wp,V)/H^1_f(k_\wp,V)$.

Assume until the end of \S\ref{subsub:lock} that $V$ satisfies
$\hne$ (both for $T$ and $T^\mathcal{D}$), $\hwd$ and  $\hpss$.  We then immediately have: 

\begin{prop}
\label{prop:finiteconditionhalfrank}
The $\oo$-module $H^1_f(k_p,T):=\oplus_{\wp|p}H^1_f(k_\wp,T)$ (resp., the module $H^1_s(k_p,T):=\oplus_{\wp|p} H^1_s(k_\wp,T)$) is free are free of rank $d_+$ (resp., of rank $r$).
\end{prop}
Fix an $\oo$-rank one direct summand $\mathcal{L} \subset H^1(k_p,T)$ such that $\al \cap H^1_f(k_p,T)=\{0\}.$ We will also write $\al$ for the (isomorphic) image of $\al$ inside $H^1_s(k_p,T)$ under the surjection $$H^1(k_p,T)\lra H^1_s(k_p,T).$$

 \begin{define}
 \label{def:line}
 Define the \emph{$\al$-modified Selmer structure} $\FF_\al$ on $T$ as follows:

 \begin{itemize}
 \item $\Sigma(\FF_\al)=\Sigma(\FFc)$,
 \item if $\lambda \nmid p$, then $H^1_{\FF_\al}(k_\lambda, T)=H^1_{\FFc}(k_\lambda,T)$,
 \item $H^1_{\FF_\al}(k_p,T):=H^1_f(k_p,T)\oplus\al \subset H^1(k_p,T)=H^1_{\FFc}(k_p,T)$.
 \end{itemize}
 \end{define}

 \subsubsection{Local conditions at $p$ over $k_\infty$}
 \label{subsub:localconditionatpoverkinfty}
 Recall that $k_{\infty}$ denotes the cyclotomic $\ZZ_p$-extension of $k$, and $\Gamma=\Gal(k_{\infty}/k)$. Assume that the hypothesis $\mathbf{H.Iw.}$ holds in this section. Let $k_{\wp}$ denote the completion of $k$ at $\wp$, and let  $k_{\wp,\infty}$ denote the cyclotomic $\ZZ_p$-extension of $k_{\wp}$. We may therefore identify $\Gal(k_{\wp,\infty}/k_{\wp})$ with $\Gamma$ for all $\wp|p$ and henceforth $\Gamma$ will stand for any of these Galois groups. Let $\LL=\oo[[\Gamma]]$ be the cyclotomic Iwasawa algebra, as usual. We also fix a topological generator $\gamma$ of $\Gamma$, and set $\xx=\gamma-1$ (and we occasionally identify $\LL$ with the power series ring $\oo[[\xx]]$).
 \begin{lemma}
 \label{lemma:free for k_infty}
 Suppose ${\hne}$  (both for $T$ and $T^\mathcal{D}$)  and ${\hwd}$\, hold. Then $H^1(k_p,\TT):=\oplus_{\wp|p}H^1(k_\wp,\TT) $ is a free $\LL$-module of rank $d$.
 \end{lemma}

 \begin{proof}
This is already proved in the first part of the proof of Lemma~\ref{lemma:free for k}.
 \end{proof}
 Assume $\hord$\, and $\hntz$\, until the end of \S\ref{subsub:localconditionatpoverkinfty}. We define the \emph{Greenberg local conditions} at $p$ by setting 
 $$H^1_{\textup{Gr}}(k_\wp,\TT):=\ker\left(H^1(k_\wp,\TT)\lra H^1(k_\wp,\textup{F}_\wp^-\TT)\right).$$
  By definition, there is an exact sequence of $\LL$-modules
 \be\label{eqn:minussequence} 0\lra \textup{F}_\wp^-\TT \stackrel{\gamma-1}{\lra} \textup{F}_\wp^-\TT \lra \textup{F}_\wp^-T \lra 0.\ee
  Taking $G_{k_\wp}$-invariance of the sequence (\ref{eqn:minussequence}) and using $\hntz$\, and Nakayama's lemma, we conclude that $H^0(k_{\wp},\textup{F}^-\TT )=0$. This in turn implies that the map
  $$H^1(k_\wp,\textup{F}^+\TT ) \lra H^1(k_\wp,\TT)$$ (induced from the $G_{k_\wp}$-cohomology of the sequence $0\ra \textup{F}_\wp^+\TT \ra \TT \ra \textup{F}_\wp^-\TT \ra 0)$ is injective and the image of $H^1(k_\wp,\textup{F}^+\TT )$ is exactly $H^1_\textup{Gr}(k_\wp,\TT)$.

 \begin{prop}\label{prop:greenbergrank} Let $r_\wp:=[k_{\wp}:\QQ_p]\cdot\textup{rank}_{\oo}\,\textup{F}_\wp^+T$.
\begin{itemize}  \item[(i)]$H^1(k_\wp,\textup{F}^+\TT)$ is a free $\LL$-module of rank $r_\wp$.
\item[(ii)] The natural map $$H^1(k_\wp,\textup{F}^+\TT) \lra H^1(k_\wp,\textup{F}^+T)$$ is surjective.
\item[(iii)] $H^1(k_\wp,\textup{F}^+T)$ is a free $\oo$-module of rank $r_\wp$.
\end{itemize}
 \end{prop}
 \begin{proof}
 The long exact sequence of the $G_{k_\wp}$-cohomology yields an exact sequence
 $$H^0(k_\wp,\textup{F}_\wp^-\TT)\lra H^1(k_\wp,\textup{F}_\wp^+\TT)\lra H^1(k_\wp,\TT).$$
 As explained above, one may deduce from $\hntz$\, that $H^0(k_\wp,\textup{F}_\wp^-\TT)=0$, so it follows from  Lemma~\ref{lemma:free for k_infty} that $H^1(k_\wp,\textup{F}_\wp^+\TT)$ is $\LL$-torsion free. (i) now follows from \cite[Theorem A.8(ii)]{kbbiwasawa}.

Long exact sequence of the $G_{k_\wp}$-cohomology of the sequence
$$0\lra \textup{F}_\wp^+\TT \stackrel{\gamma-1}{\lra}  \textup{F}_\wp^+\TT  \lra \textup{F}_\wp^+T \lra 0$$
 gives
 $$\textup{coker}\left(H^1(k_\wp,\textup{F}_\wp^+\TT) \lra H^1(k_\wp,\textup{F}_\wp^+T)\right)=H^2(k_\wp,\textup{F}_\wp^+\TT)[\gamma-1].$$
  As in the proof of Lemma~\ref{lemma:free for k},
  $$H^2(k_\wp,\textup{F}_\wp^+\TT)[\gamma-1]=0 \iff H^0(k_\wp,(\textup{F}_\wp^+T)^*)=H^0(k_\wp,\textup{F}_\wp^-T^{\mathcal D}\otimes\Phi/\oo)=0,$$
   and the latter vanishing follows from the hypotheses $\hntz$. This completes the proof of (ii). (iii) follows at once from (i) and (ii).
    \end{proof}

 \begin{cor} The $\LL$-module $H^1_{\textup{Gr}}(k_p,\TT):=\bigoplus_{\wp|p} H^1_{\textup{Gr}}(k_\wp,\TT)$ is free of rank $d_+$.
 \end{cor}

 \begin{proof} Since we assumed $\hord$(ii),  $\sum_{\wp|p}r_\wp=d_+.$
 \end{proof}
 \begin{define}
 \label{def:line over k_infty}
 Fix a $\LL$-rank one direct summand $\mathbb{L} \subset H^1(k_p,\TT)$ such that $\all\cap H^1_{\textup{Gr}}(k_p,\TT)=\{0\}$. Define the \emph{$\all$-modified Selmer structure} $\FF_\all$ on $\TT$ as follows:

 \begin{itemize}
 \item $\Sigma(\FF_\all)=\Sigma(\FF_\LL)$,
 \item if $\lambda \nmid p$, define $H^1_{\FF_\all}(k_\lambda, \TT)=H^1_{\FF_\LL}(k_\lambda,\TT)$,
 \item $H^1_{\FF_\all}(k_p,\TT):=H^1_{\textup{Gr}}(k_p,\TT)\oplus\all \subset H^1(k_p,\TT)=H^1_{\FF_\LL}(k_p,\TT)$.
 \end{itemize}
 \end{define}

\begin{rem}
\label{rem:BKvsGr}
Note that we used two different approaches to choose local conditions in \S\ref{subsub:lock} (over $k$) and in \S\ref{subsub:localconditionatpoverkinfty} (over $k_\infty$). Starting from $H^1_{\textup{Gr}}(k_p,\TT)$, we may consider the image of $H^1_{\textup{Gr}}(k_p,\TT)$ under the canonical map
$$H^1(k_p,\TT) \lra H^1(k_p,T)$$
 and denote this image by $H^1_{\textup{Gr}}(k_p,T)\subset H^1(k_p,T)$. The choice $H^1_{\textup{Gr}}(k_p,T)\subset H^1(k_p,T)$ will be called the \emph{Greenberg local condition} on $T$. It is easy to see (thanks to Proposition~\ref{prop:greenbergrank}(ii) and (iii)) that $H^1_{\textup{Gr}}(k_p,T)$ coincides with the image of
 $ H^1(k_p,\textup{F}^+T) \hookrightarrow H^1(k_p,T).$
  In several cases of interest, the Selmer group determined by the Bloch-Kato definition agrees with the Selmer group determined by the Greenberg definition; see Remark~\ref{rem:greenbergvsBK} above.

\end{rem}

\subsection{Global duality and a comparison of Selmer groups}
In this section, we compare classical Selmer groups (which we wish to relate to the $L$-values) to modified Selmer groups (for which we are able to apply the Kolyvagin system machinery and compute in terms of $L$-values). The necessary tool to accomplish this comparison is the Poitou-Tate global duality.
\subsubsection{Comparison over $k$}
\label{subsubsec:compareselmer1}
We first define the classical (Bloch-Kato) Selmer structure and Selmer group for $T$ ({resp.}, for $T^*$). Let $\FF_{\BK}$ denote the Selmer structure on $T$ given by
\begin{itemize}
 \item $\Sigma(\FF_{\BK})=\Sigma(\FFc)=\Sigma(\FF_\al)$,
 \item For $\lambda \nmid p$, $H^1_{\FF_{\BK}}(k_\lambda, T)=H^1_{\FFc}(k_\lambda,T)=H^1_{\FF_\al}(k_\lambda, T)$,
 \item $H^1_{\FF_{\BK}}(k_p,T)=H^1_f(k_p,T)\subset H^1(k_p,T)=H^1_{\FFc}(k_p,T)$.
 \end{itemize}

This induces the dual Selmer structure $\FF_{\BK}^*$ on $T^*$. Then, by definition, we have the following exact sequences:
$$\xymatrix@R=.3cm{0\ar[r]&H^1_{\FF_{\BK}}(k,T)\ar[r]& H^1_{\FF_{\al}}(k,T)\ar[r]^(.6){ \textup{loc}_p^{s}}&\al\\
0\ar[r]&H^1_{\FF_{\al}^*}(k,T^*)\ar[r]& H^1_{\FF_{\BK}^*}(k,T^*)\ar[r]^(.5){\textup{loc}_p^*}&\frac{H^1_{\FF_{\BK}^*}(k_p,T^*)}{H^1_{\FF_{\al}^*}(k_p,T^*)}}$$
where $\textup{loc}_p^{s}$ is the compositum $\textup{loc}_p^{s}: H^1(k,T)\ra H^1(k_p,T)\ra H^1_s(k_p,T).$
The Poitou-Tate global duality theorem (c.f., \cite[Theorem I.7.3]{r00}, \cite[Theorem I.4.10]{milne}, \cite[Theorem 2.3.4]{mr02}) allows us to compare the image of $\textup{loc}_p^{s}$ to the image of $\textup{loc}_p^*$:

\begin{prop}
\label{prop:compareselmeroverk}
There is an exact sequence
$$0\ra \frac{ H^1_{\FF_{\al}}(k,T)}{H^1_{\FF_{\BK}}(k,T)}\stackrel{\textup{loc}_p^{s}}{\lra}\al\stackrel{(\textup{loc}_p^*)^{\vee}}{\lra}  \left( H^1_{\FF_{\BK}^*}(k,T^*)\right)^{\vee}\ra\left(H^1_{\FF_{\al}^*}(k,T^*)\right)^{\vee}\ra 0,$$ where the map $(\textup{loc}_p^*)^{\vee}$ is induced from localization at $p$ and the local Tate pairing between $H^1(k_p,T)$ and $H^1(k_p,T^*)$.
\end{prop}
\begin{cor}
\label{cor:weakoutcomeofcomparison}
The quotient $H^1_{\FF_{\BK}^*}(k,T^*)/H^1_{\FF_{\al}^*}(k,T^*)$ is finite iff $\textup{loc}_p^s(H^1_{\FF_{\al}}(k,T)) \neq0$.
\end{cor}
\begin{proof}
Since $\al$ is a free $\oo$-module of rank one, this is immediate from Proposition~\ref{prop:compareselmeroverk}.
\end{proof}

\begin{cor}
\label{cor:compareoverkwithclass}
Suppose $H^1_{\FF_{\BK}}(k,T)=0$. Let $c \in H^1_{\FF_{\al}}(k,T)$ be any class. Then the following sequence is exact:
$$ 
0\ra {\frac{H^1_{\FF_{\al}}(k,T)}{\oo\cdot c}}\stackrel{\textup{loc}_p^s}{\lra}\frac{\al}{\oo\cdot \textup{loc}_p^s(c)} {\lra} \left(H^1_{\FF_{\BK}^*}(k,T^*)\right)^{\vee}\lra{\left(H^1_{\FF_{\al}^*}(k,T^*)\right)^{\vee}}\ra 0.$$
\end{cor}
\begin{proof}
Note that the assumption $H^1_{\FF_{\BK}}(k,T)=0$ forces the map $\textup{loc}^s_p:H^1_{\FF_{\al}}(k,T) \ra \al$ to be injective. Corollary follows from Proposition~\ref{prop:compareselmeroverk}.
\end{proof}

\begin{rem}
\label{rem:assumptionBKvanishes}
The assumption that $H^1_{\FF_{\BK}}(k,T)=0$ may seem like an unreasonably strong assumption at the moment, however, we will be able to rephrase this assumption in terms of an Euler system of rank $r$ later on.
\end{rem}
\subsubsection{Comparison over $k_\infty$}
\label{subsubsec:compareselmer2}
For a fixed topological generator  $\gamma$ of $\Gamma$, set $\gamma_n:=\gamma^{p^n}$, and define $\al_n\subset H^1(k_p,\TT/(\gamma_n-1)\TT)$ to be the image of $\all$ under the map
$$H^1(k_p,\TT) \lra H^1(k_p,\TT/(\gamma_n-1)\TT).$$
 Let $\FF_{\al_n}$ denote the Selmer structure on $\TT/(\gamma_n-1)\TT$, which is obtained by propagating the Selmer structure $\FF_{\all}$ on $\TT$ to its quotient $\TT/(\gamma_n-1)\TT$. The propagated Selmer structure from  $\FF_{\textup{Gr}}$ on $\TT$  to the quotient $\TT/(\gamma_n-1)\TT$ will still be denoted by $\FF_{\textup{Gr}}$.

 By Shapiro's lemma, there is a canonical isomorphism
 $H^1(k,\TT/(\gamma_n-1)\TT) \stackrel{\frak{s}}{\ra} H^1(k_n,T)$, and for every prime
 $\lambda \subset k$, a canonical isomorphism
 $H^1(k_\lambda,\TT/(\gamma_n-1)\TT)\stackrel{\frak{s}_\lambda}{\lra}H^1((k_n)_\lambda,T)$;
 c.f.,~\cite[Appendix B.4 and B.5]{r00}. For $\FF=\FF_{\textup{Gr}} \hbox{ or }
 \FF_{\al_n}$, we define the submodule
 $$H^1_{\FF}((k_n)_\lambda,T) \subset
 H^1((k_n)_\lambda,T)$$
  as the image of
 $H^1_{\FF}(k,\TT/(\gamma_n-1)\TT)$ under the isomorphism
 $\frak{s}_\lambda$.

Repeating the argument of Proposition~\ref{prop:compareselmeroverk} for each field $k_n$ (instead of $k$) with Selmer structures $\FF_{\Gr}$ and $\FF_{\al_n}$ and passing to inverse limit we obtain the following:

\begin{prop}
\label{prop:compare selmer over k_infty}
The following sequences of $\LL$-modules are exact:
\begin{itemize}
\item[(i)]$0\ra\frac{H^1_{\FF_{\all}}(k,\TT)}{H^1_{\FF_{\Gr}}(k,\TT)}\stackrel{\textup{loc}_p^s}{\lra}{\all}{\lra} \left(H^1_{\FF_{\Gr}^*}(k,\TT^*)\right)^{\vee}\lra{\left(H^1_{\FF_{\all}^*}(k,\TT^*)\right)^{\vee}}\ra 0.$
\end{itemize}
If further $H^1_{\FF_{\Gr}}(k,T)$ defined in Remark~\ref{rem:BKvsGr} vanishes, then,
\begin{itemize}
\item[(ii)] for any class $c \in H^1_{\FF_{\all}}(k,\TT)$,
$$0\lra {\frac{H^1_{\FF_{\all}}(k,\TT)}{\LL\cdot c}}\stackrel{\textup{loc}_p^s}{\lra}\frac{\all}{\LL\cdot \textup{loc}_p^s(c)} {\lra} \left(H^1_{\FF_{\Gr}^*}(k,\TT^*)\right)^{\vee}\lra{\left(H^1_{\FF_{\all}^*}(k,\TT^*)\right)^{\vee}}\lra 0.$$
\end{itemize}
\end{prop}

\begin{proof}
We give a sketch of the proof. As in Proposition~\ref{prop:compareselmeroverk}, we have an exact sequence
$$0\lra\frac{H^1_{\FF_{\al_n}}(k_n,T)}{H^1_{\FF_{\Gr}}(k_n,T)}\stackrel{\textup{loc}_p^s}{\lra}{\al_n}{\lra} \left(H^1_{\FF_{\Gr}^*}(k_n,T^*)\right)^{\vee}\lra{\left(H^1_{\FF_{\al_n}^*}(k_n,T^*)\right)^{\vee}}\lra 0$$
 for each $n$. Passing to inverse limit (and making use of~\cite[Proposition B.1.1]{r00}) we obtain the exact sequence of (i).

For (ii), note that there is an injection $H^1_{\FF_{\Gr}}(k,\TT)/(\gamma-1) \hookrightarrow H^1_{\FF_{\Gr}}(k,T)$ induced from the exact sequence
$$H^1(k,\TT)\stackrel{\gamma-1}{\lra} H^1(k,\TT)\lra H^1(k,T).$$
 Therefore, our assumption that $H^1_{\FF_{\Gr}}(k,T)=0$ implies, by Nakayama's lemma, that $H^1_{\FF_{\Gr}}(k,\TT)=0$. (ii) now follows from (i).
\end{proof}

\subsection{Kolyvagin systems for modified Selmer structures}
\label{subsec:KS1}
Throughout \S\ref{subsec:KS1} we assume that the hypotheses $\mathbf{H.1}$-$\mathbf{5}$ hold for $T$.  Assume in addition that $\hne$ (both for $T$ and $T^{\mathcal{D}}$), $\hwd$\, and $\htam$\, 
hold.

One may apply~\cite[Lemma 3.7.1]{mr02} to verify that all the three Selmer triples  $(T,\FF_{\BK},\PP)$, $(T,\FF_{\Gr},\PP)$ and $(T,\FF_{\al},\PP)$ satisfy the hypothesis $\mathbf{H.6}$ of \cite[\S 3.5]{mr02} (with base field $\QQ$ in their treatment replaced by $k$). Therefore, the existence of Kolyvagin systems for these Selmer structures will be decided by their \emph{core Selmer ranks} (c.f., \cite[Definition 4.1.8 and 4.1.11]{mr02}). Let $\XX(T,\FF)$ denote the core Selmer rank of the Selmer structure $\FF$ on $T$, for $\FF=\FF_{\BK}, \FF_{\Gr} \hbox{ or } \FF_{\al}$. 
\begin{prop}
\label{rem:canonicalSelmerrank}
$\XX(T,\FFc)=r$.
\end{prop}
\begin{proof}
This follows using~\cite[Theorem 5.2.15]{mr02}, since we assumed $\hne$.
\end{proof}
\begin{prop}
\label{prop:BKGrcoreselmerrank}
If $\hpss$ holds (resp., $\hord$ and $\hntz$ hold), then $\XX(T,\FF_{\BK})=0$ (resp., $\XX(T,\FF_{\Gr})=0$).
\end{prop}
\begin{proof}
By~ \cite[Definition 5.2.4 and Proposition 5.2.5]{mr02} and \cite[Proposition 1.6]{wiles}
\begin{align*}
&\XX(T,\FF)=\textup{dim}_{\mathbb{F}}\,H^1_{\FF}(k,T/\mm T)-\textup{dim}_{\mathbb{F}}\,H^1_{\FF^*}(k,T^*[\mm])=\\
&\textup{dim}_{\mathbb{F}}\,H^0(k,T/\mm T)-\textup{dim}_{\mathbb{F}}\,H^0(k,T^*[\mm])
-\sum_{\lambda \in \Sigma(\FF)}\textup{dim}_{\mathbb{F}}\,H^0(k_{\lambda},T/\mm T)-\textup{dim}_{\mathbb{F}}\,H^1_{\FF}(k_{\lambda},T/\mm T)
\end{align*}
Applying this formula with $\FF=\FFc$ and  $\FF=\FF_{\BK}$ 
we see that
$$
\XX(T,\FFc)-\XX(T,\FF_{\BK})=\textup{dim}_{\mathbb{F}}\,H^1_{\FFc}(k_{p},T/\mm T)-\textup{dim}_{\mathbb{F}}\,H^1_{\FF_{\BK}}(k_{p},T/\mm T)
$$
 and this equals $d-d_+=r$. We have by Proposition~\ref{rem:canonicalSelmerrank} that $\XX(T,\FFc)=r$ and the proof follows.

Identical proof works also for the Greenberg condition under $\hord$ and $\hntz$.

\end{proof}

\begin{prop}
\label{modifiedcorerank}
If $\hpss$ holds, then $\XX(T,\FF_{\al})=1$.
\end{prop}
\begin{proof}
Mimicking the proof of Proposition~\ref{prop:BKGrcoreselmerrank}, we obtain the identity
$$
\XX(T,\FF_{\al})-\XX(T,\FF_{\BK})=\textup{dim}_{\mathbb{F}}\,H^1_{\FF_{\al}}(k_{p},T/\mm T)-\textup{dim}_{\mathbb{F}}\,H^1_{\FF_{\BK}}(k_{p},T/\mm T)
$$
 and this equals to one by the very definition of the $\al$-modified Selmer structure. We already know by Proposition~\ref{prop:BKGrcoreselmerrank} that $\XX(T,\FF_{\BK})=0$ and the proof follows.

Note that if we assumed $\hord$ and $\hntz$ (instead of assuming $\hpss$), and used $\FF_{\Gr}$ of Remark~\ref{rem:BKvsGr} (instead of $\FF_{\BK}$) in order to define $H^1_{\FF_{\al}}(k_p,T)=H^1_{\Gr}(k_p,T)\oplus\al$, the same proof would lead us to the identical result about $\XX(T,\FF_{\al})$, for the Selmer structure $\FF_{\al}$ which is obtained by modifying the Greenberg local conditions.
\end{proof}

\subsubsection{Kolyvagin systems over $k$}
We write  $\textup{\textbf{KS}}(T,\FF_{\al},\PP)$ for the $\oo$-module of Kolyvagin systems for the Selmer triple $(T,\FF_{\al},\PP)$. We refer the reader to \cite[Definition 3.1.3]{mr02} for a definition of this module. Assume that the hypotheses $\mathbf{H.1}$-$\mathbf{5}$, $\hne$ and $\hpss$ ($\hord$ and $\hntz$ instead of $\hpss$  whenever we refer to Greenberg's local conditions) hold.
\begin{prop}
\label{prop:modifiedKSrank1}
The $\oo$-module $\textup{\textbf{KS}}(T,\FF_{\al},\PP)$ is free of rank \emph{one}, generated by a Kolyvagin system $\kappa \in \textup{\textbf{KS}}(T,\FF_{\al},\PP)$ whose image \textup{(}under the canonical map induced from reduction mod $\mm$\textup{)} in $\textup{\textbf{KS}}(T/\mm T,\FF_{\al},\PP)$ is non-zero.
\end{prop}

\begin{proof}
This is immediate after Proposition~\ref{modifiedcorerank} and~\cite[Theorem 5.2.10]{mr02}.
\end{proof}

\begin{rem}
\label{rem:question}
Note that the {choice} of a rank one direct summand $\al \subset H^1(k_p,T)$ makes our approach somewhat unnatural. This issue is addressed in this paragraph. Put \be\label{eqn:decomposeH1} H^1(k_p,T)=\bigoplus_{i=1}^{r}\al_i\oplus H^1_f(k_p,T)\ee (where each $\al_i$ is a free $\oo$-submodule of $H^1(k_p,T)$ of rank one) and consider \be\label{eqn:sum} \sum_{i=1}^{r}\KS(T,\FF_{\al_i},\PP)\subset \KS(T,\FFc,\PP).\ee
\begin{claim}
The sum in (\ref{eqn:sum}) is in fact a direct sum.
\end{claim}
\begin{proof}
Assume contrary: Suppose there exist $\pmb{\kappa}^i \in \KS(T,\FF_{\al_i},\PP)$ such that $\sum_{i=1}^r\pmb{\kappa}^i=0$, and not all $\pmb{\kappa}^i=0$; say without loss of generality $\pmb{\kappa}^1\neq 0$. Then
$$\pmb{\kappa}^1=-\sum_{i \neq 1} \kappa^{i} \in \sum_{i \neq 1} \KS(T,\FF_{\al_i},\PP).$$
 This means, for every $\eta \in \NN(\PP)$ (:=\,square free products of primes in $\PP$)
\be\label{eqn:kappa1equality}\al_1/I_\eta\al_1\ni\textup{loc}_p^s(\kappa_\eta^1)=-\sum_{i \neq 1} \textup{loc}_p^s(\kappa_\eta^i) \in \bigoplus_{i\neq 1} \al_i/I_\eta \al_i.\ee
Here  $I_\eta:=\prod_{\lambda|\eta} I_\lambda \subset \oo,$
 and for $\lambda \in \PP$, the ideal $I_\lambda \subset \oo$ is as defined in the introduction. The equality of~(\ref{eqn:kappa1equality})  takes place in
 $$H^1_s(k_p,T/I_\eta T):=\frac{H^1(k_p,T/I_\eta T)}{H^1_f(k_p,T/I_\eta T)},$$
  where $H^1_f(k_p,T/I_\eta T)$ is the image of $H^1_f(k_p,T)$ under the surjective (thanks to $\hne$) map
  $$H^1(k_p,T)\twoheadrightarrow H^1(k_p,T/I_{\eta} T),$$
   which is induced from the surjection $T\twoheadrightarrow T/I_\eta T$. We therefore have a decomposition
   $$H^1(k_p,T/I_\eta T) \cong H^1_f(k_p,T/I_\eta T) \oplus \bigoplus_{i=1}^r \al_i/I_\eta \al_i.$$ Since $\left(\bigoplus_{i\neq1} \al_i/I_\eta \al_i\right)\cap\al_1/I_\eta \al_1=\{0\},$ it follows from~(\ref{eqn:kappa1equality}) that $\textup{loc}_p^s(\kappa_\eta^1)=0$, i.e.,
    $$\textup{loc}_p(\kappa_\eta^1) \in H^1_f(k_p,T/I_\eta T)$$
     for every $\eta \in \NN(\PP)$. This means $\pmb{\kappa}^1 \in \KS(T,\FF_{\BK},\PP)$. On the other hand $\KS(T,\FF_{\BK},\PP)=0$ by Proposition~\ref{prop:BKGrcoreselmerrank} and \cite[Theorem 5.2.10(i)]{mr02}, hence we proved $\pmb{\kappa}^1=0$, a contradiction.
\end{proof}

As in~\cite[Remark 1.27]{kbbstick}, we pose the following: \\
\textbf{Question:} Is the direct sum
$$\bigoplus_{i=1}^{r}\KS(T,\FF_{\al_i},\PP)\subset \KS(T,\FFc,\PP)$$
 independent of the choice of the decomposition (\ref{eqn:decomposeH1})?

When the answer to this question is affirmative, we would have a \emph{canonical} rank $r$ submodule of $ \KS(T,\FFc,\PP)$. It would be even more interesting to see if this rank $r$ submodule descends from Euler systems (via the Euler systems to Kolyvagin systems map of Mazur and Rubin~\cite[Theorem 3.2.4]{mr02}). Below, we construct such a (rank $r$) submodule of $ \KS(T,\FFc,\PP)$ which descends from an Euler system of rank $r$ (in case it exists); however, this module \emph{does} depend on the decomposition~(\ref{eqn:decomposeH1}).

\end{rem}

\subsubsection{Kolyvagin systems over $k_\infty$}
For every $s,m \in \ZZ^+$ and for a fixed topological generator of $\gamma$ of $\Gamma$, write $T_{s,m}=\TT/(p^s,(\gamma-1)^m)$.
\begin{define}(Compare to~\cite[Definition 3.1.6]{mr02}.)
\label{def:KSboth}
Define the module of \emph{$\LL$-adic Kolyvagin systems} as
$$\overline{\KS}(T\otimes_{\oo}\LL,\FF_{\all},\PP):=\varprojlim_{s,m}\varinjlim_{j} \KS(T_{s,m},\FF_{\all},\PP_j),$$
where $\KS(T_{s,m},\FF_{\all},\PP_j)$ is the module of Kolyvagin
systems for the Selmer structure $\FF_{\all}$ on the representation
$T_{s,m}$, as in~\cite[Definition 3.1.3]{mr02}.
\end{define}
The analogue of~\cite[Theorem 5.2.10]{mr02}, which we used to prove Proposition~\ref{prop:modifiedKSrank1}, for the big Galois representation $\TT$ has been proved by the author in~\cite[Theorem 3.23]{kbb}. Under the hypotheses $\mathbf{H.1}$-$\mathbf{5}$, $\hne$, $\htam$ and $\hord$, this result together with Proposition~\ref{prop:modifiedKSrank1} (now using $\FF_{\Gr}$ on $\TT$ (instead of $\FF_{\BK}$) to define $\FF_{\al}$ on $T$) can be used to show:
\begin{prop}
\label{prop:modifedKSinftyrk1}
The $\LL$-module of Kolyvagin Systems $\overline{\KS}(\TT,\FF_{\all},\PP)$ for the Selmer structure $\FF_{\all}$ on $\TT$  is free of rank one. Furthermore, the canonical map
$$\overline{\KS}(\TT,\FF_{\all},\PP) \lra{\KS}(T,\FF_{\mathcal{L}},\PP)$$
 is surjective.
\end{prop}

\begin{proof}
Theorem 3.23 of~\cite{kbb} is proved for the canonical Selmer structure $\FF_\LL=\FFc$ on $\TT$, under the condition that $\XX(T,\FFc)=1$. Under the running hypotheses, which in particular imply (Proposition~\ref{modifiedcorerank}) that $\XX(T,\FF_\al)=1$, the proof of~\cite[Theorem 3.23]{kbb} applies \emph{verbatim}  for the Selmer structure $\FF_{\all}$ on $\TT$.
\end{proof}


 \section{Euler systems of rank $r$ and the Euler systems to Kolyvagin systems map}
 \label{sec:ES}
Suppose $k,T,r$ and $\PP$ are as in \S\ref{subsec:notationhypo}. We define $\NN=\NN(\PP)$ as the collection of ideals
$$\NN(\PP)=\{\tau= \frak{q}_1\cdots\qq_s \subset k\,\,\,\big{|}\,\,\, \qq_i \in \PP \hbox{ are distinct prime ideals}\}.$$
As before, we write $k(\qq)$ for the maximal $p$-extension of $k$ inside the ray class field of $k$ modulo $\qq$ and let $\textup{Fr}_{\qq}$ denote an arithmetic  Frobenius at $\qq$ in $G_k$. If $\tau=\qq_1\cdots\qq_s$ is an ideal in $\NN$, we let $k(\tau)$ denote the compositum $k(\tau):=k(\qq_1)\cdots k(\qq_s),$
  and set $k_n(\tau):=k_n\cdot k(\tau)$. We define
  $$\mathfrak{C}=\{k_n(\tau): \tau \in \NN, n \in \ZZ_{\geq 0}\},$$ and $\kk=\bigcup_{F\in \mathfrak{C}} F$. We set $\Delta^\tau=\Gal(k(\tau)/k)$ and $\Delta_n^\tau=\Gal(k_n(\tau)/k))=\Delta^\tau\times\Gamma_n$. Finally,
  $$P_{\qq}(x):=\textup{det}(1-\textup{Fr}_{\qq}^{-1}\cdot x|T^{\mathcal{D}}) \in \oo[x]$$ 
  is the Euler factor at the prime $\qq \in \PP$ associated with the dual Galois representation $T^{\mathcal{D}}=\Hom(T,\oo)(1)$.

 For any finite group $G$ and a finitely generated $\oo[G]$-module $M$, we define (following~\cite[\S1.2]{ru96})
 \begin{align*}
 \wedge^r_0M:=\{m \in \,&\Phi\otimes\wedge^rM: (\psi_1\wedge\cdots\wedge \psi_r)(m) \in \oo[G]\\ &\hbox{ for every }\psi_1,\dots,\psi_r \in \Hom_{\oo[G]}(M,\oo[G])\},
 \end{align*} where the exterior power is calculated in the category of $\oo[G]$-modules.
 \begin{define}
 \label{def:weaklyintegralESofrankr}
 An \emph{Euler system of rank r} is a collection $\mathbf{c}=\{c_{k_n(\tau)} \}$ such that
 \begin{itemize}
 \item[(i)] $c_{k_n(\tau)}\in \wedge^r_0H^1(k_n(\tau),T)$,
\item[(ii)] for $\tau^\prime|\tau$ and $n \geq n^\prime$, we have
$\textup{Cor}^r_{{k_n(\tau)/k_{n\prime}({\tau^\prime})}}\left(c_{k_n(\tau)}\right)=\left(\prod_{\substack{\qq|\tau\\ \qq\nmid \tau^\prime}}P_{\qq}(\textup{Fr}_{\qq}^{-1})\right)c_{k_{n^\prime}(\tau^\prime)},$ where $\textup{Cor}^r_{{k_n(\tau)/k_{n\prime}({\tau^\prime})}}$ is the map induced from the corestriction $$\textup{Cor}_{{k_n(\tau)/k_{n\prime}({\tau^\prime})}}: H^1(k_n(\tau),T) \lra H^1(k_{n^\prime}(\tau^\prime),T).$$
\end{itemize}
We note that the $\wedge^r H^1(k_n(\tau),T)$ is the $r$-th exterior power of the $\oo[\Delta_n^\tau]$-module $H^1(k_n(\tau),T)$ in the category of $\oo[\Delta_n^\tau]$-modules.
 \end{define}

 \begin{rem}
 \label{rem:weakintegrality}
 Note that we demand the collection $\mathbf{c}$ to be integral in a weaker sense than~\cite[\S1.2.2]{pr-es}, i.e., we allow our classes to have \emph{denominators}. This, of course, is inspired from~\cite{ru96}, and  this weaker version is sufficient for our purposes.
 \end{rem}

 \begin{rem}
 \label{rem:examplesesrankr}
 We describe a conjectural example of an Euler system of rank $r$ in~\S\ref{subsec:PRpadic} below, which we obtain from Perrin-Riou's conjectures on $p$-adic $L$-functions. See also~\cite{kbbstark, kbbiwasawa} where an Euler system of rank $r$ for the multiplicative group $\mathbb{G}_m$ is studied extensively. In a forthcoming work, we consider an Euler system of rank $r$ for Hecke characters of CM fields, which is obtained from Rubin-Stark elements, in order to study CM abelian varieties of higher dimension.
 \end{rem}
 \begin{rem}
 \label{rem:locallyintegral}
For any number field $K$, let
$$\textup{loc}_p: \wedge^r_0 H^1(K,T)\lra \wedge^r_0 H^1(K_p,T)$$
 ({resp.},
 $$\textup{loc}_p^s: \wedge^r_0 H^1(K,T) \lra \wedge^r_0 H^1_s(K_p,T))$$
  be the map induced from
  $$H^1(K,T)\lra H^1(K_p,T)$$
    ({resp.}, from the compositum
    $$H^1(K,T)\lra H^1(K_p,T)\lra H^1_s(K_p,T)).$$
     Suppose $\mathbf{c}=\{c_{k_n(\tau)}\}$ is an Euler system of rank $r$. Then our results regarding the freeness of the semi-local cohomology from~\S\ref{subsub:lock} and \S\ref{subsub:localconditionatpoverkinfty} above, together with~\cite[Example 1 on page 38]{ru96} show that
\be\label{eqn:newnew}\wedge^r_0 H^1(K_p,T)=\wedge^r H^1(K_p,T) \hbox{ and } \wedge^r_0 H^1_s(K_p,T)=\wedge^r H^1_s(K_p,T).\ee
Here the equalities are induced from the canonical inclusion $\wedge^rM \hookrightarrow \QQ_p\otimes\wedge^rM$. It follows from \ref{eqn:newnew} that
$$\textup{loc}_p(c_{k_n}) \in \wedge^r H^1((k_n)_p,T) \hbox{ and } \textup{loc}_p^s(c_{k_n}) \in \wedge^rH^1_s((k_n)_p,T).$$
\end{rem}
 \begin{rem}
 \label{rem:comparisonofeulerfactors}
 The `Euler factors' $P_{\qq}(\textup{Fr}_{\qq}^{-1})$ which appear in the distribution relation (ii) above matches with  the Euler factors in~\cite{pr-es,r00} but differ from the Euler factors chosen in~\cite[Definition 3.2.3]{mr02}. However, thanks to~\cite[\S IX.6]{r00}, it is possible to go back and forth between these two choices and~\cite[Theorem 3.2.4]{mr02} still applies.
 \end{rem}
 \begin{rem}
 \label{rem:ESrank1}
 Suppose $r=1$. In this case
 $$\wedge^r_0H^1(K,T)=\wedge^r H^1(K,T)=H^1(K,T)$$
  for any number field $K \subset \kk$ (where the first equality is~\cite[Proposition 1.2(ii)]{ru96}) and our definition agrees with Perrin-Riou's definition~\cite[\S1.2.1]{pr-es} of an \emph{Euler system of rank one}; and these both agree with Rubin's~\cite[Definition II.1.1 and Remark II.1.4]{r00} definition of an \emph{Euler system}. We also will henceforth call an Euler system of rank one simply an `Euler system'.
  \end{rem}

 \subsection{Euler systems to Kolyvagin systems map}
 \label{subsec:ESKSmap}
We first recall what Mazur and Rubin call the \emph{Euler systems to Kolyvagin systems map}. Suppose $T,\PP$ and $\kk$ are as in the beginning of \S\ref{sec:ES}. Let $\ES(T)=\ES(T,\kk)$ denote the collection of Euler systems (i.e., \emph{Euler systems of rank one} in the sense of Definition~\ref{def:weaklyintegralESofrankr}) for $(T,\kk)$. Fix a topological generator $\gamma$ of $\Gamma$ and set $\gamma_n=\gamma^{p^n}$, and let $\mm_{\LL}$ be the maximal ideal of $\LL=\oo[[\Gamma]]$.
\begin{define}
\label{def:alternativeKS}
For $\mathbb{F}=\FF_{\LL}$ or $\FF_{\all}$, we set
$$\overline{\KS}^{\prime}(\mathbb{T},\mathbb{F},\PP):=\varprojlim_{m,n}\varinjlim_j \textup{\KS}(\mathbb{T}/(p^m,\gamma_{n}-1)\mathbb{T},\mathbb{F}, \PP_{j}),$$
where $\textup{\KS}(\mathbb{T}/(p^m,\gamma_{n}-1)\mathbb{T},\mathbb{F}, \PP_{j})$ is the $\LL/(p^m,\gamma_n-1)$-module of Kolyvagin systems (in the sense of~\cite[Definition 3.1.3]{mr02}) for the Selmer structure $\mathbb{F}$ propagated to the quotient $\mathbb{T}/(p^m,\gamma_{n}-1)\mathbb{T}$.
\end{define}
\begin{rem}
\label{1-kolsys} We introduced the module $\overline{\KS}^{\prime}(\mathbb{T},\mathbb{F},\PP)$ above because, after applying Kolyvagin's descent procedure~\cite[\S IV]{r00} (modified appropriately in~\cite[Appendix A]{mr02}) on an Euler system, one obtains elements of $\overline{\KS}^{\prime}(\mathbb{T},\FF_{\LL},\PP)$. On the other hand, it is not hard to see that the module  $\overline{\KS}^{\prime}(\mathbb{T},\mathbb{F},\PP)$ defined above is naturally isomorphic to the module $\overline{\KS}(\mathbb{T},\mathbb{F},\PP)$ of Definition~\ref{def:KSboth}, using the fact that each of the collections of ideals $\{\langle p^m,\gamma_n-1\rangle\}_{_{m,n}}$ and $\{\langle p^m,(\gamma-1)^n\rangle\}_{_{m,n}}$ forms a base of neighborhoods at $zero$.  Furthermore, using the fact that the collection $\{\mm_\LL^\alpha\}_{\alpha\in \ZZ^+}$ also forms a base of neighborhoods at $zero$, one may identify these two modules of Kolyvagin systems with the generalized module of Kolyvagin systems defined in~\cite[Definition 3.1.6]{mr02}. By slight abuse,  we will write  $\overline{\KS}(\mathbb{T},\mathbb{F},\PP)$ for any of the three modules of Kolyvagin systems given by three different definitions (i.e., by Definitions~\ref{def:KSboth} and \ref{def:alternativeKS} here; and by~\cite[Definition 3.1.6]{mr02}). For our purposes in this section, we will use Definition~\ref{def:alternativeKS} as the description of this module (mainly because, one naturally lands in this module after applying Kolyvagin-Mazur-Rubin descent~\cite[Theorem 5.3.3]{mr02} on an Euler system).
 \end{rem}

 Consider the following hypotheses:
\begin{itemize}
\item[$\mathbf{KS1.}$] $T/(\textup{Fr}_{\qq}-1)T$ is a cyclic $\oo$-module for every $\qq \in \PP$.
\item[$\mathbf{KS2.}$] $\textup{Fr}_{\qq}^{p^k}-1$ is injective on $T$ for every $\qq \in \PP$ and $k\geq 0$.
\end{itemize}

\begin{thm}\textup{(\cite[Theorem 3.2.4 \& 5.3.3]{mr02})}
\label{thm:ESKSmain} Suppose the hypotheses  $\mathbf{KS1}$-$\mathbf{2}$ hold. Then there are canonical maps
\begin{itemize}
\item $\ES(T) \lra \overline{\KS}(T,\FFc,\PP)$,
\item $\ES(T) \lra \overline{\KS}(\TT,\FF_{\LL},\PP)$
\end{itemize}
with the properties that
\begin{itemize}
\item if $\mathbf{c}$ maps to $\pmb{\kappa} \in \overline{\KS}(T,\FFc,\PP)$ then $\kappa_1=c_k$,
\item  if $\mathbf{c}$ maps to $\pmb{\kappa}^{\textup{Iw}} \in \overline{\KS}(\TT,\FF_{\LL},\PP)$ then
$$\kappa_1^{\textup{Iw}}=\{c_{k_n}\} \in \varprojlim_n H^1(k_n,T)=H^1(k,\TT).$$
\end{itemize}
\end{thm}
Starting from an Euler system of rank $r$, one first applies Perrin-Riou's procedure ~\cite[\S1.2.3]{pr-es} (based on an idea due to Rubin~\cite[\S6]{ru96}) to obtain an Euler system.  After this, we would like to apply the Euler systems to Kolyvagin systems map (Theorem~\ref{thm:ESKSmain}) on these Euler systems. Note however that Theorem~\ref{thm:ESKSmain} will only give rise to Kolyvagin systems for the coarser Selmer structures $\FF_\LL$ and $\FFc$ (rather than the finer Selmer structures $\FF_\all$ and $\FF_\al$).

Let $\ES^{(r)}(T)=\ES^{(r)}(T,\kk)$ denote the collection of Euler systems of rank $r$. The previous paragraph is summarized in the diagram below:
$$\xymatrix{\ES^{(r)}(T) \ar@{.>}[rd]_{{\mathcal{R}}}\ar[r]^{\textup{[PR98]}}&\ES(T)\ar[r]^(.4){\textup{[MR02]}}&\overline{\KS}(\TT,\FF_{\LL},\PP) \ar[r]&\overline{\KS}(T,\FFc,\PP)\\
&\pmb{(?)} \ar@{^{(}->}[u]\ar@{.>}[r]^(.3){\mathcal{D}_{\LL}}\ar@{.>}@/_1.8pc/[rr]_(.43){\mathcal{D}}&\overline{\KS}(\TT,\FF_{\all},\PP)\ar[r]\ar@{^{(}->}[u]&\overline{\KS}(T,\FF_{\al},\PP)\ar@{^{(}->}[u]
}$$

To be able to obtain Kolyvagin systems for the modified Selmer structures $\FF_\all$ and $\FF_\al$, we need to analyze the structure of semi-local cohomology groups for $\TT$ and $T$ at $p$, over various ray class fields of $k$. This is carried out in \S\ref{subsec:localstructure}. We then apply the results of \S\ref{subsec:localstructure} in~\S\ref{subsec:choosehoms} to choose carefully a map $\mathcal{R}$ such that  the image of the map $\mathcal{R}$ determines the correct submodule $\pmb{(?)} \subset \ES(T)$, on which the Euler systems to Kolyvagin systems map (Theorem~\ref{thm:ESKSmain}) restricts to what we call $\mathcal{D}_{\LL}$ and $\mathcal{D}$; and gives (see \S\ref{subsec:kolsys2}) Kolyvagin systems for the modified Selmer structures $\FF_{\all}$ and $\FF_\al$.

\subsection{Semi-local preparation}
\label{subsec:localstructure}
Throughout \S\ref{subsec:localstructure} we will assume $\hne$ (both for $T$ and $T^\mathcal{D}$) and $\hwd$\, hold true.

\begin{lemma}
\label{lem:referee} Let $\tau \in \NN(\PP)$ and let $k(\tau)$ be as
above. Then $H^0(k(\tau)_p,X^*)=0,$ for $X=T, T^\mathcal{D}$.
\end{lemma}
The original argument to prove Lemma~\ref{lem:referee} was
defective. The proof which we include below has been kindly pointed
to us by the anonymous referee.
\begin{proof}
Let $v$ be any prime of $k(\tau)$ above $p$. Write
$\mathcal{D}_v^\tau$ for the decomposition group of $v$ inside
$\Gal(k(\tau)/k):=\Delta^\tau$. We may identify $\mathcal{D}_v^\tau
\subset \Delta^\tau$ with the local Galois group
$\Gal(k(\tau)_v/k_\wp)$ where $\wp\subset k$ is the prime below $v$.
If $\mathcal{D}_v^\tau$ is trivial, then
$H^0(k(\tau)_v,X^*)=H^0(k_\wp,X^*)$ and Lemma follows from $\hne$.
If $\mathcal{D}_v^\tau$ is not trivial, then it is a non-trivial
$p$-group, hence the order of $H^0\left(k(\tau)_v,X^*[p]\right)$ is
congruent modulo $p$ to the order of
$$H^0\left(k(\tau)_v,X^*[p]\right)^{\mathcal{D}_v^\tau}=H^0\left(k_\wp,X^*[p]\right)=0,$$
thus $H^0\left(k(\tau)_v,X^*[p]\right)=0$ as well.
\end{proof}

 \begin{lemma}
 \label{lem:surj}
 \begin{itemize} For every $k_n(\tau) \in \mathfrak{C}$, the corestriction maps
 \item[(i)] $H^1(k_n(\tau)_p,T)\lra H^1(k(\tau)_p,T)$,
 \item[(ii)] $H^1(k(\tau)_p,T)\lra H^1(k_p,T)$,
 \item[(iii)] $H^1(k_n(\tau)_p,T)\lra H^1(k_p,T)$
  \end{itemize}
   on the semi-local cohomology at $p$ are all surjective.
 \end{lemma}
 \begin{proof}
 The cokernel of the map
 $$H^1(k(\tau),\TT)=\varprojlim_n H^1(k_n(\tau)_p,T)\lra H^1(k(\tau)_p,T)$$
  is given by $H^2(k(\tau)_p,\TT)[\gamma-1]$, where $\gamma$ is any topological generator of
  $\Gamma=\Gal(k_\infty/k)$. Since it is known that~(c.f., \cite[Proposition 3.2.1]{pr}) $H^2(k(\tau)_p,\TT)$ is a
  finitely generated $\oo$-module, it follows that
  $$H^2(k(\tau)_p,\TT)[\gamma-1]=0 \iff H^2(k(\tau)_p,\TT)/(\gamma-1)=0.$$ Since the cohomological dimension of
  the absolute Galois group of any local field is 2,
  $$H^2(k(\tau)_p,\TT)/(\gamma-1)\cong H^2(k(\tau)_p,\TT/(\gamma-1))=H^2(k(\tau)_p,T).$$
   It therefore suffices to check that $$H^2(k(\tau)_p,T):=\bigoplus_{v|p}H^2(k(\tau)_v,T)=0,$$
   which, via local duality is equivalent to checking that $H^0(k(\tau)_v,T^*)=0$ for each $v|p$. The proof of (i)
   now follows from
   Lemma~\ref{lem:referee}.

Now set $T_\tau:=\textup{Ind}_{k(\tau)/k}\,T$. The semi-local version of Shapiro's lemma
(which is explained in~\cite[\S A.5]{r00}) gives an isomorphism
$$H^1(k(\tau)_p, T) \cong H^1(k_p,T_\tau)$$ and the map
$$\textup{Cor}_\tau: \,H^1(k_p,T_\tau)\cong H^1(k(\tau)_p,T)\lra H^1(k_p,T)$$ is induced from the augmentation sequence
$$0\lra \mathcal{A}_\tau\cdot T_\tau\lra T_\tau\lra T\lra 0,$$ where $\mathcal{A}_\tau$ is the augmentation ideal of the
local ring $\oo[\Delta^\tau]$. The argument above shows that the cokernel of $\textup{Cor}_\tau$ is dual to
$$H^0(k_p,(\mathcal{A}_\tau\cdot T_\tau)^*):=\oplus
_{\wp|p}H^0(k_\wp, (\mathcal{A}_\tau\cdot T_\tau)^*).$$
 Furthermore, $(\mathcal{A}_\tau\cdot T_\tau)^{*}: =
 \Hom(\mathcal{A}_\tau\cdot T_\tau,\Phi/\oo(1))=\Hom(\mathcal{A}_\tau\cdot T_\tau,\Phi/\oo)\otimes \oo(1)$
  and,
\be\label{eqn:gorlikeprop}\Hom_{\oo}(\mathcal{A}_\tau\cdot
T_\tau,\Phi/\oo)\stackrel{\sim}{\lra}\mathcal{A}_\tau\cdot\Hom_{\oo}(T_\tau,\Phi/\oo)\ee
  by the Claim
  that we prove below, hence
  $$H^0(k_p,(\mathcal{A}_\tau\cdot T_\tau)^*) \hookrightarrow H^0(k_p,T_\tau^{*}).$$
   It therefore suffices to show that $H^0(k_p,T_\tau^*)=0$.
   By local duality this is equivalent to proving $H^2(k_p,T_\tau)=0$, which by the semi-local
   Shapiro's Lemma equivalent to show $H^2(k(\tau)_p,T)=0$, which again by local duality equivalent
   to the statement $H^0(k(\tau)_p,T^*)=0$;
   and this we have already verified in Lemma~\ref{lem:referee}. Thus, the proof of (ii) follows once
   we verify (\ref{eqn:gorlikeprop}) in the Claim below.

(iii) clearly follows from (i) and (ii).
  \end{proof}
  We now verify that (\ref{eqn:gorlikeprop}) holds true:
\begin{claim}
\label{claim:refgorremark} $\Hom_{\oo}(\mathcal{A}_\tau\cdot
T_\tau,\Phi/\oo)\stackrel{\sim}{\lra}\mathcal{A}_\tau\cdot\Hom_{\oo}(T_\tau,\Phi/\oo).$
\end{claim}

\begin{proof}
Write $\Delta^\tau=\mathcal{G}_1\times\cdots\times \mathcal{G}_s$ as a product of non-trivial cyclic groups. For each $1\leq i \leq s$, choose a generator $\delta_i$ of $\mathcal{G}_i$. Set $R_i=\oo[\mathcal{G}_1\times\cdots\times \mathcal{G}_i]$. Thanks to the direct sum decomposition above, we think of $R_i$ as both a quotient ring and a subring as $R_{j}$ for $j \geq i$. We define $\mathcal{A}_i=\ker(R_i \ra \oo)$, the augmentation ideal of $R_i$. Define also $R_i^{(1)}=\oo[\mathcal{G}_i]$ and $\mathcal{A}_i^{(1)}=\ker(R_i^{(1)} \ra \oo)$.

 We then have a surjective map
$T_\tau \stackrel{\delta_i-1}{\lra} \mathcal{A}_i^{(1)}\cdot T_\tau$  and an induced injection 
$$ \xymatrix@R=.085in{
\Hom_{\oo}(\mathcal{A}_i^{(1)}\cdot
T_\tau,\Phi/\oo)\ar[r]^{\circ[\delta_i-1]} & \mathcal{A}_i^{(1)}\cdot\Hom_{\oo}(T_\tau,\Phi/\oo)\\
\phi \ar@{|->}[r]& \left\{x \mapsto \phi((\delta_i-1)x)\right
\}
}$$
Define
$$\xymatrix@R=.085in@C=0.1in{
\frak{d}_i:&\mathcal{A}_i^{(1)}\cdot\Hom_{\oo}(T_\tau,\Phi/\oo)\ar[rrrr]&&&& \Hom_{\oo}(\mathcal{A}_i^{(1)}\cdot
T_\tau,\Phi/\oo)\\
&(\delta_i^{-1}-1)\psi\ar@{|->}[rrrr] &&&& \psi \big{|}_{\mathcal{A}_i^{(1)}\cdot T_\tau}
}$$
It is easy to verify that $\frak{d}_i$ is well-defined and the two maps $\frak{d}_i$ and $\circ[\delta_i-1]$ are
mutual inverses of each other. This proves that: 
\begin{itemize}
\item[(i)] $\mathcal{A}_k^{(1)}\cdot\Hom_{\oo}(T_\tau,\Phi/\oo) \stackrel{\sim}{\lra} \Hom_{\oo}(\mathcal{A}_k^{(1)}\cdot
T_\tau,\Phi/\oo)$, for all $k$,
\item[(ii)] $\mathcal{A}_1\cdot\Hom_{\oo}(T_\tau,\Phi/\oo) \stackrel{\sim}{\lra} \Hom_{\oo}(\mathcal{A}_1\cdot
T_\tau,\Phi/\oo)$.

Now the proof that 
$$\mathcal{A}_k\cdot\Hom_{\oo}(T_\tau,\Phi/\oo) \stackrel{\sim}{\lra} \Hom_{\oo}(\mathcal{A}_k\cdot
T_\tau,\Phi/\oo)$$ 
for every $k$ (in particular, for $k=s$, i.e., for $\mathcal{A}_s=\mathcal{A}_\tau$) follows by induction on $k$, invoking $\textup{(i)}$ above at each step.

\end{itemize}
\end{proof}
 \begin{prop}
 \label{prop:semilocalstructure}
 \begin{enumerate}
  For every $\tau \in \NN(\PP)$:
\item[(i)]  The semi-local cohomology group $H^1(k(\tau)_p,T)$ is a free $\oo[\Delta^\tau]$-module of rank $d$.
 \item[(ii)] For every $n\in\ZZ_{\geq 0}$, the semi-local cohomology group $H^1(k_n(\tau)_p,T)$ is a free $\oo[\Delta^\tau_n]$-module of rank $d$.
 \end{enumerate}
 \end{prop}

 \begin{proof}
 We start with the remark that $H^1(k(\tau)_p,T)$ is a free $\oo$-module of rank $d\cdot|\Delta^\tau|$. Indeed, this may be proved following the proof of Lemma~\ref{lemma:free for k} (again relying on the hypotheses $\hne$ (both for $T$ and $T^{\mathcal{D}}$) and $\hwd$). Further, we know thanks to Lemma~\ref{lem:surj} that the map $H^1(k(\tau)_p,T)\ra H^1(k_p,T)$ (which could be thought of as reduction modulo the augmentation ideal $\mathcal{A}_\tau \subset \oo[\Delta^\tau]$) is surjective. Nakayama's Lemma and Lemma~\ref{lemma:free for k} therefore imply that $H^1(k(\tau)_p,T)$ is generated by (at most) $d$ elements over the ring $\oo[\Delta^\tau]$. Let $\frak{B}=\{x_1, x_2,\dots,x_{d}\}$ be any set of such generators. To prove (i), it suffices to check that the $x_i$'s do not admit any non-trivial $\oo[\Delta^\tau]$-linear relation. Assume contrary, and suppose there is a non-trivial relation
 \be\label{eqn:relation}
 \sum_{i=1}^{d}\alpha_i x_i=0, \,\,\, \alpha_i \in \oo[\Delta^\tau].
 \ee
 Write $S=\{\delta x_j: \delta \in \Delta^\tau, 1\leq j\leq d\},$ note that by our assumption on the set  $\frak{B}$, the set $S$ generates $H^1(k(\tau)_p,T)$ as an $\oo$-module, and $|S|=d\cdot|\Delta^\tau|=\textup{rank}_{\oo} \, H^1(k(\tau)_p,T)$.  Equation (\ref{eqn:relation}) can be rewritten as 
 $$\sum_{\delta,j} a_{\delta,j}\cdot \delta x_j=0$$
  with $a_{\delta,j} \in \oo$. Since we already know that $H^1(k(\tau)_p,T)$ is $\oo$-torsion free, we may assume without loss of generality that $a_{\delta_0,j_0} \in \oo^\times$ for some $\delta_0,j_0$. This in turn implies that 
 $$\delta_0x_{j_0} \in \textup{span}_{\oo}(S-\{\delta_0x_{j_0}\}),$$
  hence $H^1(k(\tau)_p,T)$ is generated by $S-\{\delta_0x_{j_0}\}$. This, however, is a contradiction since we already know that the $\oo$-rank of $H^1(k(\tau)_p,T)$ is $d\cdot|\Delta^\tau|=|S|$, hence it cannot be generated by $|S|-1$ elements over $\oo$. The proof of (i) is now complete.

 (ii) is proved in an identical fashion, now considering the \emph{augmentation map}
 $$H^1(k_n(\tau)_p,T)\lra H^1(k(\tau)_p,T),$$ which is surjective thanks to Lemma~\ref{lem:surj}.
 \end{proof}

 Let $\kk_0 \subset \kk$ be the composite of all fields $k(\tau)$, where $\tau$ runs through the set $\NN=\NN(\PP)$. Set $\pmb{\Delta}:=\Gal(\kk_0/k)$.
 \begin{cor}
 \label{cor:thick free}
  $\varprojlim_{n,\tau} H^1(k_n(\tau)_p,T)$ is a free $\oo[[\Gamma\times\pmb{\Delta}]]$-module of rank $d$ and the natural projection maps $$\varprojlim_{n,\tau} H^1(k_n(\tau)_p,T) \lra H^1(k_{m}(\eta)_p,T)$$ are surjective for all $m \in \ZZ_{\geq 0}$ and $\eta \in \NN$.
 \end{cor}
 \begin{proof}
 Immediate after Proposition~\ref{prop:semilocalstructure}.
 \end{proof}

\subsection{Choosing the correct homomorphisms}
\label{subsec:choosehoms}
In this section we use the results from~\S\ref{subsec:localstructure} to choose useful homomorphisms which will be utilized in~\S\ref{subsec:kolsys2} to construct Kolyvagin systems for the modified Selmer structure $\FF_{\al}$ (resp., $\FF_{\all}$) on $T$ (resp., on $\TT$). This will be carried out in two steps: Under the hypotheses $\hpss$\, on $T$, we will make our choice of homomorphisms in~\S\ref{subsubsec:pss} and use the results of this section in~\S\ref{subsubsec:KSoverkbis} to construct an element of $\overline{\KS}(T,\FF_{\al},\PP)$ out of an Euler system of rank $r$. For the Iwasawa theoretic results, we will assume $\hord$, and we will show how to choose the useful homomorphisms in~\S\ref{subsubsec:ordcase}. This choice will be utilized in~\S\ref{subsubsec:KSoverkinftybis} to construct an element of $\overline{\KS}(\TT,\FF_{\all},T)$ starting from an Euler system of rank $r$.
\subsubsection{Choice of Homomorphisms: Potentially semi-stable case}
\label{subsubsec:pss}
In this section, we assume that the hypothesis $\hpss$ holds along with $\hne$\, and $\hwd$ (which we need to prove Proposition~\ref{prop:finiteconditionhalfrank}). Let $\al \subset H^1(k_p,T)$ be as in~\S\ref{subsub:lock}. As before,  we denote the (isomorphic) image of $\al$ under $H^1(k_p,T)\ra H^1_s(k_p,T)$ also by $\al$.

\begin{prop}
\label{prop:linearalgebra1}
There exists a decomposition of  the $\oo[[\Gamma\times\pmb{\Delta}]]$-module of rank-$d$
$$\mathbb{V}_p:=\varprojlim_{n,\tau} H^1(k_n(\tau)_p,T)=\pmb{\al^{+}} \oplus\pmb{\al_s}$$
with a distinguished rank one direct summand $\pmb{\al}\subset \pmb{\al_s}$ with the following properties:

Under the maps induced from the corestriction map $$\varprojlim_{n,\tau} H^1(k_n(\tau)_p,T)\lra H^1(k_p,T),$$
\begin{enumerate}
\item $\pmb{\al^{+}}$ (resp., $\pmb{\al_s}$) is a free $\oo[[\Gamma\times\pmb{\Delta}]]$-module of rank $d_+$ (resp., of rank $r$),
\item $\pmb{\al^{+}}$ projects onto $H^1_f(k_p,T)$, and  $\pmb{\al_s}$ onto $H^1_s(k_p,T)$,
\item $\pmb{\al}$ projects onto $\al$.
\end{enumerate}
\end{prop}
\begin{proof}
By Lemma~\ref{lemma:free for k}, the $\oo$-module $H^1(k_p,T)$ is free  of rank $d$, and by Proposition~\ref{prop:finiteconditionhalfrank}, the $\oo$-module $H^1_f(k_p,T)$ (resp., $H^1_s(k_p,T)$) is free of rank $d_+$ (resp., of rank $r$). Fix a basis
$\{e_1, \dots, e_{d_+}, e_{1+d_+},\dots, e_{d}\}$ of $H^1(k_p,T)$ such that
\begin{itemize}
\item $\{e_1, \dots, e_{d_+}\}$ is a basis for $H^1_f(k_p,T)$,
\item $e_{1+d_+}$ generates the $\oo$-line $\al$,
\item the isomorphic image of $\{e_{1+d_+},\dots,e_{d}\}$ under $H^1(k_p,T)\twoheadrightarrow H^1_s(k_p,T)$ gives a basis of $H^1_s(k_p,T)$.
\end{itemize}
Now use Nakayama's lemma to lift $\frak{b}$ to a basis $\{\mathbb{E}_1,\dots,\mathbb{E}_{d}\}$ of $\mathbb{V}_p$, such that $\mathbb{E}_i$ maps to $e_i$ under the augmentation map induced from  $\oo[[\Gamma\times\pmb{\Delta}]]\ra \oo$. Set
$\pmb{\al^{+}}=\textup{span}_{\oo[[\Gamma\times\pmb{\Delta}]]}\{\mathbb{E}_1,\dots,\mathbb{E}_{d_+}\}$, $\pmb{\al}=\textup{span}_{\oo[[\Gamma\times\pmb{\Delta}]]}\,\{\mathbb{E}_{1+d_+}\}$ and $\pmb{\al_s}=\textup{span}_{\oo[[\Gamma\times\pmb{\Delta}]]}\,\{\mathbb{E}_{1+d_+},\dots,\mathbb{E}_{d}\}$.

\end{proof}
\begin{define}
\label{lines}
For $k_n(\tau)=K \in \mathfrak{C}$, let $\al_{K}$ ({resp.}, $\al_K^{+}$; {resp.}, $\al_K^s$) be the image of $\pmb{\al}$ ({resp.}, $\pmb{\al^{+}}$; {resp.}, $\pmb{\al_s}$) under the (surjective) projection map \, $\vv_p \ra H^1(K_p,T)$.
\end{define}

Note that $\al_K$ ({resp.}, $\al_K^{+}$; resp., $\al_K^s$) is a free
$\ZZ_p[\textup{Gal}(K/k)]$-module of rank \emph{one} ({resp.}, of
rank $d_+$; resp., of rank $r$) for all $K \in \mathfrak{C}$, and that the restriction map
induces an isomorphism
$$\textup{res}: X_{K} \stackrel{\sim}{\lra} (X_{K^{\prime}})^{\textup{Gal}(K^{\prime}/K)}$$
 for $X=\al, \al^{+}$ and $\al^s$; for all $K\subset K^{\prime}$.  When $K=k$, note that $\al_K=\al$ and $\al_K^{+}=H^1_f(k_p,T)$ by definition (Proposition~\ref{prop:linearalgebra1}).

We write
\begin{align*}
\bigwedge^{r-1}\textbf{Hom}(\pmb{\al_s},\oo[[\Gamma\times\pmb{\Delta}]]):=
\varprojlim_{K \in \mathfrak{C}}
\bigwedge^{r-1}{}_{_{\oo[\Delta_K]}}\,\hbox{Hom}_{\oo[\Delta_K]}(\al_K^s,
\oo[\Delta_K]).
\end{align*}
 Here $\Delta_K=\Gal(K/k)$ and the inverse limit is with respect to the natural maps induced from  $$\al^s_K \lra (\al_{K^{\prime}}^s)^{\Gal(K^{\prime}/K)}$$ and the isomorphism
 \begin{align*}
  \oo[\Delta_{K^{\prime}}]^{\Gal(K^{\prime}/K)} & \tilde{\lra}\oo[\Delta_K]\\
\mathbf{N}^{K^{\prime}}_K &\longmapsto 1
 \end{align*} for $K \subset K^{\prime}$.

Localization at $p$ followed by the projection onto the ``singular quotient" $\al_K^s$ gives rise to a map
$$\textup{loc}_p^s: H^1(K,T)\stackrel{\textup{loc}_p}{\lra}H^1(K_p,T)\lra \al_K^s,$$which  induces a canonical  map
$$
\bigwedge^{r-1}\textbf{Hom}(\pmb{\al_s},\oo[[\Gamma\times\pmb{\Delta}]])\lra \varprojlim_{K\in\mathfrak{C}} \bigwedge^{r-1}\textup{Hom}_{\oo[\Delta_K]}(H^1(K,T), \oo[\Delta_K]).
$$
The image of $\Psi \in \hhh$ under this map will still be denoted by $\Psi$.
\begin{prop}
\label{def: es}
Suppose $\mathbf{c}^{(r)}=\{c_K^{(r)}\}_{_{K\in\mathfrak{C}}}$ is an Euler system of rank $r$. For any
 $$\{\psi_{K}\}_{_{K\in\mathfrak{C}}}=\Psi \in \hhh,$$ define $$H^1(K,T) \ni c_{K,\Psi}:=\psi_K(c_{K}^{(r)}).$$ Then the collection $\{c_{K,\Psi}\}_{_{K\in\mathfrak{C}}}$ is an Euler system for the $G_k$-representation $T$.
\end{prop}

We will sometimes denote the Euler system $\{c_{K,\Psi}\}_{_{K\in\mathfrak{C}}}$ by  $\{c_{k_n(\tau),\Psi}\}_{_{n,\tau}}$.
\begin{proof}
This is proved in~\cite[\S1.2.3]{pr-es}. See also~\cite[Proposition 6.6]{ru96} for the treatment in the particular case $T=\ZZ_p(1)$.
\end{proof}
 \begin{prop}
 \label{krasner-2}
 For any $K \in \mathfrak{C}$, the projection map
 \begin{align*}\hhh \lra
  \bigwedge^{r-1}\textup{Hom}_{\oo[\Delta_K]}(\al_K^s, \oo[\Delta_K])
  \end{align*}
   is surjective.
\end{prop}
\begin{proof}
Obvious since all $\al_K^s$, for $K\in \mathfrak{C}$, are free $\oo[\Delta_K]$-modules.
\end{proof}

If one applies the \emph{Euler systems to Kolyvagin systems map} of Mazur and  Rubin ({c.f.}, \cite[Theorem 5.3.3]{mr02}) on the Euler system $\{c_{K,\Psi}\}_{_{K\in \mathfrak{C}}}$ above, all one gets a priori is a Kolyvagin system for the (coarser) Selmer structure $\FFc$, and in general \emph{not} for the (finer) Selmer structure $\FF_\al$. Below, we will choose these homomorphisms $\Psi$ carefully so that the resulting Kolyvagin system is indeed a Kolyvagin system for the modified Selmer structure $\FF_{\al}$ (resp., $\FF_{\all}$) on $T$ (resp., on $\TT$).

\begin{define}
\label{hli}
We say that an element $$\{\psi_K\}_{_{K\in \mathfrak{C}}}=\Psi \in \hhh$$ satisfies $\hli$ if for any $K\in \mathfrak{C}$ one has $\psi_K(\wedge^{r} \al_K^s) \subset \al_K$.
\end{define}

We now construct a specific element
$$\Psi_0 \in \hhh$$
 that satisfies $\hli$ (and which, in a certain sense, is the best possible choice).

Fix an $\oo[[\Gamma\times\pmb{\Delta}]]$-basis
$$\{\Psi_{\pmb{\al}}^{(1)}, \dots, \Psi_{\pmb{\al}}^{(r-1)}\}$$
of the free  $\oo[[\Gamma\times\pmb{\Delta}]]$-module $\textup{Hom}_{\oo[[\Gamma\times\pmb{\Delta}]]}(\pmb{\al_s}/\pmb{\al},\oo[[\Gamma\times\pmb{\Delta}]])$ of rank $r-1$. This in turn fixes  a basis $\{\psi_{\al_K}^{(i)}\}_{i=1}^{r-1}$ for the free $\oo[\Delta_K]$-module $\textup{Hom}_{\oo[\Delta_K]}\left(\al_{K}^s/\al_K, \oo[\Delta_K]\right)$
 for all $K \in \mathfrak{C}$; such that $\{\psi_{\al_K}^{(i)}\}_{_{K\in\mathfrak{C}}}$ are compatible with respect to the surjections
 $$\xymatrix{ \textup{Hom}_{\oo[\Delta_{K^{\prime}}]}(\al_{K^{\prime}}^s/\al_{K^{\prime}},\oo[\Delta_{K^{\prime}}])\ar[r] & \textup{Hom}_{\oo[\Delta_K]}(\al_K^s/\al_K,\oo[\Delta_K])
 }$$
  for all $K\subset K^\prime$. Note that the homomorphism
  $$\bigoplus_{i=1}^{r-1}\psi_{\al_K}^{(i)}:\al_K^s/\al_K \lra \oo[\Delta_K]^{r-1}$$ is an isomorphism of $\oo[\Delta_K]$-modules, for all $K\in \mathfrak{C}$.
Let $\psi_{K}^{(i)}$ denote the image of $\psi_{K}^{(i)}$ under the canonical injection
 $$\xymatrix{
 \textup{Hom}_{\oo[\Delta_K]}(\al_K^s/\al_K,\oo[\Delta_K]) \ar @{^{(}->}[r]& \textup{Hom}_{\oo[\Delta_K]}(\al_K^s,\oo[\Delta_K]).
 }$$
 Note then that the map $$\Psi_K:=\bigoplus^{r-1}_{i=1}\psi_K^{(i)}: \al_K^s\lra \oo[\Delta_K]^{r-1}$$ is surjective and $\textup{ker}(\Psi_K)=\al_K$.


Define
$$\varphi_K := \psi_K^{(1)}\wedge\psi_K^{(2)} \wedge\dots \wedge\psi_K^{(r-1)} \in \bigwedge^{r-1}\textup{Hom}(\al_K^s,\oo[\Delta_K]). $$
For $K\subset K^{\prime}$, note that $\varphi_{K^{\prime}}$ maps to $\varphi_{K}$ under the homomorphism
$$\xymatrix{\bigwedge^{r-1}\textup{Hom}(\al_{K^{\prime}}^s,\oo[\Delta_{K^{\prime}}]) \ar@{->>}[r]& \bigwedge^{r-1}\textup{Hom}(\al_K^s,\oo[\Delta_K]).}$$
We may therefore regard $\Psi_0:=\{\varphi_K\}_{_{K\in \mathfrak{C}}}$ as an element of $\hhh$. Composing with $\textup{loc}_p^s: H^1(K,T) \ra \al_K^s$, we may further regard $\Psi_0$ as an element of $$\varprojlim_{K \in \mathfrak{C}} \bigwedge^{r-1} \textup{Hom}(H^1(K,T),\oo[\Delta_K]).$$
\begin{prop}
\label{homs-hli-1}
Suppose $\{\varphi_K\}_{_{K}}=\Psi_0$ is as above. Then $\varphi_K$ maps $\wedge^r \al_K^s$ isomorphically onto $\ker(\Psi_K)=\al_K $,  for all $K \in \mathfrak{C}$. In particular, $\Psi_0$ satisfies $\textup{H}_\al$.
\end{prop}

\begin{proof}
The proof is identical to the proof of (the easy half of)~\cite[Lemma 3.1]{kbb}, which also follows the proof of~\cite[Lemma B.1]{mr02}  almost line by line.
\end{proof}

\subsubsection{Choice of Homomorphisms: The case $T$ satisfies $\hord$}
\label{subsubsec:ordcase} Throughout~\S\ref{subsubsec:ordcase} we assume the hypotheses $\hord$, $\hntz$, $\hne$ and $\hwd$ hold true. Let $H^1_{\textup{Gr}}(k_p,\TT)$ and $\all$ be the submodules of $H^1(k_p,\TT)$ defined in~\S\ref{subsub:localconditionatpoverkinfty}.

We start with the following Proposition whose proof is identical to the proof of Proposition~\ref{prop:linearalgebra1}:
\begin{prop}
\label{prop:linearalgebra2}
There exists a decomposition of the $\oo[[\Gamma\times\pmb{\Delta}]]$-module of rank-$d$
$$\mathbb{V}_p:=\varprojlim_{n,\tau} H^1(k_n(\tau)_p,T)=\pmb{\al^{+}} \oplus\pmb{\al_s}$$
with a distinguished rank one direct summand $\pmb{\al}\subset \pmb{\al_s}$ with the following properties:

\begin{enumerate}
\item[(1)] $\pmb{\al^{+}}$ (resp., $\pmb{\al_s}$) is a free $\oo[[\Gamma\times\pmb{\Delta}]]$-modules of rank $d_+$ (resp., of rank $r$),
\end{enumerate}
Under the maps induced from the corestriction
$$\varprojlim_{n,\tau}H^1(k_n(\tau)_p,T)\lra  \varprojlim_n H^1((k_n)_p,T)=H^1(k_p,\TT),$$
\begin{enumerate}
\item[(2)] $\pmb{\al^{+}}$ projects onto $H^1_{\textup{Gr}}(k_p,\TT)$, and $\pmb{\al_s}$ onto $H^1_s(k_p,\TT):=H^1(k_p,\TT)/H^1_{\textup{Gr}}(k_p,\TT),$
\item[(3)] $\pmb{\al}$ projects onto $\all$.
\end{enumerate}
\end{prop}

Having defined $\pmb{\al^{+}}$, $\pmb{\al_s}$ and $\pmb{\al}$, we may proceed as in \S\ref{subsubsec:pss} and define $\al_K^{+}$, $\al_K^s$ and $\al_K$ as above; and use these to define a particular element
$$\Psi_0 \in \hhh$$
 in an identical fashion. We also note that
$$H^1(k_p,\TT)\supset \all =\{\al_{k_n}\} \subset \varprojlim_n H^1((k_n)_p,T).$$
\begin{define}
\label{hLi}
We say that an element $$\{\psi_K\}_{_{K\in \mathfrak{C}}}=\Psi \in \hhh$$ satisfies $\hLi$ if for any $K\in \mathfrak{C}$ one has $\psi_K(\wedge^{r} \al_K^s) \subset \al_K$.
\end{define}
Although the definition of the property $\hLi$ is identical to the definition of $\hli$ (Definition~\ref{hli}), we wish to distinguish between these two in order to remind us that we used Greenberg's local conditions as a start for one, and Bloch-Kato local conditions for the other. Finally, we note that the following (almost identical) version of Proposition~\ref{homs-hli-1} holds:

\begin{prop}
\label{prop:homshLi}
Let $\Psi_0=\{\varphi_K\}_{_{K}}$ be as above. Then $\varphi_K$ maps $\wedge^r \al_K^s$ isomorphically onto $\ker(\Psi_K)=\al_K$,  for all $K \in \mathfrak{C}$. In particular, $\Psi_0$ satisfies $\textup{H}_{\all}$.
\end{prop}
\begin{rem}
\label{rem:diagramupdate}
The diagram in \S\ref{subsec:ESKSmap} now looks as follows:
$$\xymatrix@C=.7cm{\ES^{(r)}(T) \ar@{.>}[rd]_(.45){\Psi_0}\ar[r]^{\textup{[PR98]}}&\ES(T)\ar[r]^(.4){\textup{[MR02]}}&\overline{\KS}(\TT,\FF_{\LL},\PP) \ar[r]&\overline{\KS}(T,\FFc,\PP)\\
&\Psi_0(\ES^{(r)}(T)) \ar@{^{(}->}[u]\ar@{.>}[r]^(.48){\mathcal{D}_{\LL}}\ar@{.>}@/_1.8pc/[rr]_(.48){\mathcal{D}}&\overline{\KS}(\TT,\FF_{\all},\PP)\ar[r]\ar@{^{(}->}[u]&\overline{\KS}(T,\FF_{\al},\PP)\ar@{^{(}->}[u]
}$$
 where $\Psi_0(\ES^{(r)}(T))$ stands for the collection of Euler systems (of rank one) obtained from Euler systems of rank $r$, following the procedure of Perrin-Riou and Rubin (Proposition~\ref{def: es}) with the choice $\Psi_0 \in \hhh$. In the following section, we verify that the restriction of the Euler systems to Kolyvagin systems map of~\cite{mr02} on $\Psi_0(\ES^{(r)}(T))$ really restricts to the maps $\mathcal{D}_{\LL}$ and $\mathcal{D}$.
\end{rem}
\begin{rem}\label{interchange}
Since the maps $H^1_s(k_p.\TT) \lra \al_{k_n}^s$, for $n \in \ZZ^+$, are all surjective (by our choices made in Proposition~\ref{prop:linearalgebra2})  and $H^1_s(k_p,\TT)$ ({resp.}, $\al_{k_n}^s$) is a free $\LL$-module ({resp.}, $\oo[\Gamma_n]$-module) of rank $r$, it follows that there is a canonical isomorphism
$$\varprojlim_n\wedge^r_{\oo[\Gamma_n]} \al_{k_n}^s \cong \wedge^r_{\LL} \varprojlim_n \al_{k_n}^s= \wedge^r_{\LL} H^1_s(k_p,\TT).$$
This and  Proposition~\ref{prop:homshLi} show that $\varphi_\infty=\{\varphi_{k_n}\}_{_n}$ maps $\bigwedge^r H^1_s(k_p,\TT)$ isomorphically onto $\all=\varprojlim_n \al_n$.
\end{rem}
 \subsection{Kolyvagin systems for modified Selmer structures (bis)}
  \label{subsec:kolsys2}
  We are now ready to construct  Kolyvagin systems\footnote{which we proved to exist in~\S\ref{subsec:KS1}.} for the $\al$-modified Selmer structure $\FF_{\al}$ on $T$ ({resp.}, $\all$-modified Selmer structure $\FF_{\all}$ on $\TT$) starting from an Euler system of rank $r$, for each choice of a compatible homomorphisms $\Psi \in \hhh$ that satisfies $\hli$ ({resp.}, $\hLi$). 
  \subsubsection{Kolyvagin systems over $k$ \textup{(bis)}}
  \label{subsubsec:KSoverkbis}
Theorem~\ref{thm:ESKSmain} gives a map
$$\textup{\textbf{ES}}(T)\lra \overline{\textup{\textbf{KS}}}({T},\FFc,\PP)$$
 where
 $$\overline{\textup{\textbf{KS}}}({T},\FFc,\PP):=\varprojlim_{\alpha}(\varinjlim_{j} \textup{\textbf{KS}}({T}/\mm^\alpha{T},\,\FFc,\PP \cap \PP_{j}))$$
  is the (generalized) module of Kolyvagin systems for the Selmer triple $(T,\FFc,\PP)$ and $\FFc$ is the canonical Selmer structure on ${T}$ as in Example~\ref{example:canonical selmer} (and its \emph{propagation} to the quotients of ${T}$). One of the main properties  of this map is that if an Euler system  $\left\{c_{k_n(\tau)}\right\}_{n,\tau}=\textbf{c} \in \textup{\textbf{ES}}(T)$ maps to the Kolyvagin system ${\boldsymbol\kappa}=\left\{\left\{\kappa_{\tau}(\alpha)\right\}_{\tau \in \NN_{j}}\right\}_{\alpha}$ under this map, then
\be\label{eqn:xyeqn}\xymatrix@R=.2cm @C=.30cm{
{\kappa}_1\ar@{=}[d]  & \ar@{=}[l]_(.75){\textup{def}} \varprojlim_{\alpha} \kappa_{1}(\alpha) \,\,\,\in& \varprojlim_{\alpha}H^1(k,{T}/\mm^ \alpha{T}) =H^1(k,{T}) \\
c_k &\in \,\,H^1(k,T),&
}\ee
c.f., \cite[Theorem 3.2.4]{mr02}. Let ${\boldsymbol\kappa}^{\Psi_0}=\left\{\left\{\kappa^{\Psi_0}_{\tau}(\alpha)\right\}_{_{\tau \in \NN_j}}\right\}_{\alpha}$ be the image of the Euler system $\textbf{c}^{(r)}_{\Psi_0}=\left\{c_{k_n(\tau),\Psi_0}^{(r)}\right\}_{_{n,\tau}}$, which itself is obtained from an Euler system $\textbf{c}^{(r)}=\{c_K^{(r)}\}_{K\in\mathfrak{C}}$ of rank $r$ via Proposition~\ref{def: es} applied with $\Psi_0=\{\varphi_K\}_{K\in\mathfrak{C}}$ above. Thus the equation (\ref{eqn:xyeqn}) reads
\be\label{eqn:kappa1es1}{\kappa}_1^{\Psi_0}=c_{k,\Psi_0}^{(r)}=\varphi_k(c_k^{(r)}).\ee

\begin{thm}
\label{mainkolsys}
${\boldsymbol\kappa}^{\Psi_0}:=\left\{\left\{\kappa^{\Psi_0}_{\tau}(\alpha)\right\}_{_{\tau \in \NN}}\right\}_{\alpha} \in  \overline{\textup{\textbf{KS}}}({T},\FF_{\al},\PP).$
\end{thm}

Here $$\overline{\textup{\textbf{KS}}}({T},\FF_{\al},\PP)=\varprojlim_{\alpha}(\varprojlim_{j} \textup{\textbf{KS}}({T}/\mm^ \alpha{T},\FF_{\al},\PP \cap \PP_{j}))$$is the (generalized) module of Kolyvagin systems (see~\cite[Definition 3.1.6]{mr02} for a definition) for the $\al$-modified Selmer structure $\FF_{\al}$ on ${T}$.

\begin{rem}
We could have used \emph{any} element $\Psi \in \hhh$ in~Theorem~\ref{mainkolsys} that satisfies $\hli$ (rather then the particular element $\Psi_0$) and still obtain Kolyvagin systems for the $\al$-modified Selmer structure.
 \end{rem}
For the rest of this section the integer $\alpha $ will be fixed, and we denote the element $\kappa^{\Psi_0}_{\tau}(\alpha) \in H^1(k,{T}/\mm^ \alpha{T})$ by $\kappa^{\Psi_0}_{\tau}$. Note that the statement of Theorem~\ref{mainkolsys} claims for each $\tau \in \NN_{\alpha}$ that
$$\kappa^{\Psi_0}_{\tau} \in H^1_{\FF_{\al}(\tau)}(k,{T}/\mm^ \alpha{T}),$$
 where $\FF_{\al}(\tau)$ is defined as in~\cite[Example 2.1.8]{mr02}. However,~\cite[Theorem 5.3.3]{mr02} already concludes that
 $$\kappa^{\Psi_0}_{\tau} \in  H^1_{\FFc(\tau)}(k,{T}/\mm^ \alpha{T}).$$
 Since $\FF_{\al}$ and $\FFc$ determine the same local conditions outside $p$, it suffices to prove the following in order to prove Theorem~\ref{mainkolsys}:

\begin{prop}
\label{loc-p}
Let
$$\textup{loc}_p^s: H^1(k,T/\mm^ \alpha T) \lra H^1_s(k_p,T/\mm^ \alpha T):=\frac{H^1(k_p,T/\mm^ \alpha T)}{H^1_f(k_p,T/\mm^ \alpha T)}$$
 be the localization map into the semi-local cohomology at $p$, followed by the projection onto the singular quotient. Then
 $$\textup{loc}_p^s (\kappa^{\Psi_0}_{\tau}) \in \al/\mm^ \alpha\al \subset H^1_s(k_p,T/\mm^ \alpha T).$$
\end{prop}
\begin{proof}
We first note that  $H^1_f(k_p,T/\mm^\alpha T)$ is by definition the propagation (in the sense of~\cite[Example 1.1.2]{mr02}) of $H^1_f(k_p,T)$. Similarly,
$$H^1_f(k_p,T/\mm^\alpha T)\oplus \al/\mm^\alpha\al=H^1_{\FF_{\al}(\tau)}(k_p,T/\mm^\alpha T)$$
 is the propagation of the $\al$-modifed condition $H^1_{\FF_{\al}(\tau)}(k_p,T):=H^1_{f}(k_p,T)\oplus\al$ at $p$. Let
$$\left\{\tilde{\kappa}^{\Psi_0}_{\tau}(\alpha) \in H^1(k_p,T/\mm^\alpha T) \right\}_ {\tau \in \NN_{\alpha}}$$
 be the collection that \cite[Definition 4.4.10]{r00} associates to the Euler system $\left\{c_{k_n(\tau),\Psi_0}^{(r)}\right\}_{_{n,\tau}}$. Here we write $\tilde{\kappa}^{\Psi_0}_{\tau}(\alpha)$ for the class denoted by $\kappa_{[k,\tau,\alpha]}$ in loc.cit. Since we have fixed $\alpha$  until the end of this section, we will safely drop $\alpha$ from the notation and denote $\tilde{\kappa}^{\Psi_0}_{\tau}(\alpha)$ by $\tilde{\kappa}^{\Psi_0}_{\tau}$. 
\begin{claim}
\label{reduction-1}
If $\textup{loc}_p^s(\tilde{\kappa}^{\Psi_0}_{\tau}) \in \al/\mm^\alpha \al$ then $\textup{loc}_p^s(\kappa^{\Psi_0}_{\tau}) \in \al/\mm^\alpha \al$ as well.
\end{claim}
\begin{proof}
This follows from Equation (33) in \cite[Appendix A]{mr02}, which
relates $\kappa^{\Psi_0}_{\tau}$ to
$\tilde{\kappa}^{\Psi_0}_{\tau}$. We note that what we call
$\kappa^{\Psi_0}_{\tau}$ here corresponds to $\kappa_n^\prime$ in
loc.cit.
\end{proof}
To prove Proposition, it therefore suffices to check that
$\textup{loc}_p^s(\tilde{\kappa}^{\Psi_0}_{\tau}) \in \al/\mm^\alpha
\al$, which we prove in Lemma~\ref{lined} below. Let $D_{\tau}$
denote the derivative operator of Kolyvagin, defined as in~\cite[
Definition 4.4.1]{r00}. Rubin~\cite[Definition 4.4.10]{r00} defines
$\tilde{\kappa}^{\Psi_0}_{\tau}$ as a canonical inverse image of
$D_{\tau}c_{k(\tau),\Psi_0}^{(r)}$ (mod~$\mm^\alpha$) under the
restriction map\footnote{This restriction map is an isomorphism if
we assume (H.3): First of all, it follows from hypothesis (H.3) that
$(T/\mm^\alpha T)^{G_k}=0$ (c.f.,~\cite[Lemma 3.5.2]{mr02}).
Furthermore, $\#(T/\mm^\alpha T)^{G_{k(\tau)}}\equiv \#(T/\mm^\alpha
T)^{G_{k}}=1 \mod p$ since $\Gal(k(\tau)/k)$ is a non-trivial $p$-group.
Finally, note that the kernel and cokernel of the restriction map in
question are both annihilated by $\# (T/\mm^\alpha
T)^{G_{k(\tau)}}$.}
$$H^1(k,T/\mm^\alpha T) \lra H^1(k(\tau),T/\mm^\alpha T)^{\Delta^{\tau}}.$$
Therefore, $\textup{loc}_p^s(\tilde{\kappa}^{\Psi_0}_{\tau})$ maps to $\textup{loc}_p^s\left(D_{\tau}c_{k(\tau),\Psi_0}^{(r)}\right)$ (mod~$\mm^\alpha$) under the map\footnote{This map is also an isomorphism thanks to $\hne$.}
$$H^1_s(k_p,T/\mm^\alpha T) \lra H^1_s(k(\tau)_p,T/\mm^\alpha T)^{\Delta^{\tau}}(:=\left(\al^s_{k(\tau)}/\mm^\alpha \al^s_{k(\tau)}\right)^{\Delta^{\tau}}).$$
Under this map, $\al/\mm^\alpha\al \subset H^1(k_p,T/\mm^\alpha T)$ is mapped isomorphically onto the rank one $\oo/\mm^\alpha\oo$-module  $\left(\al_{k(\tau)}/\mm^\alpha\al_{k(\tau)}\right)^{\Delta_{\tau}}$, by the definition of $\al_{k(\tau)}$ and by the fact that $\al_{k(\tau)}$ is a free $\oo[\Delta^{\tau}]$-module. The diagram below summarizes the discussion in this paragraph:
$$\xymatrix{
H^1(k_p,T/\mm^\alpha T) \ar[r]^(.45){\sim}& H^1(k(\tau)_p,T/\mm^\alpha T)^{\Delta_{\tau}}\\
\al/\mm^\alpha \al\ar[r]^(.4){\sim} \ar@{^{(}->}[u] &\left(\al_{k(\tau)}/\mm^\alpha \al_{k(\tau)}\right)^{\Delta^{\tau}}\ar@{^{(}->}[u]
}$$

\begin{lemma}
\label{lined}
$\textup{loc}_p(\tilde{\kappa}^{\Psi_0}_{\tau}) \in \al/\mm^\alpha \al.$
\end{lemma}
\begin{proof}[Proof of Lemma~\ref{lined}]
Since $\textup{loc}_p$ is Galois equivariant $\textup{loc}_p(D_{\tau}c_{k(\tau),\Psi_0}^{(r)})=D_{\tau}\textup{loc}_p(c_{k(\tau),\Psi_0}^{(r)}).$ Furthermore,
$$\textup{loc}_p\left(c_{k(\tau),\Psi_0}^{(r)}\right) \in \al_{k(\tau)},$$
 since $\Psi_0$ satisfies $\hli$. On the other hand, by \cite[Lemma 4.4.2]{r00}, the class $D_{\tau}c_{k(\tau),\Psi_0}^{(r)}$ (mod~$\mm^\alpha$) is fixed by $\Delta^{\tau}$, which in turn implies that
$$\textup{loc}_p\left(c_{k(\tau),\Psi_0}^{(r)}\right) \,(\textup{mod\,} \mm^\alpha) \in \left(\al_{k(\tau)}/\mm^\alpha\al_{k(\tau)}\right)^{\Delta^{\tau}}.$$
This shows that $\textup{loc}_p(\tilde{\kappa}^{\Psi_0}_{\tau})$ maps into $\al/\mm^\alpha \al$ by the discussion in the paragraph preceding the statement of this Proposition.
\end{proof}
This also finishes the proof of the Proposition.
\end{proof}
By the discussion preceding the statement of Proposition~\ref{loc-p}, this completes the proof of Theorem~\ref{mainkolsys}.

    \subsubsection{Kolyvagin systems over $k_\infty$ \textup{(bis)}}
  \label{subsubsec:KSoverkinftybis}
  For $\mathbb{F}=\FF_{\LL}$ or $\FF_{\all}$, recall
   $$\overline{\textup{\textbf{KS}}}(\mathbb{T},\mathbb{F},\PP):=\varprojlim_{\alpha,n}(\varinjlim_{j} \textup{\textbf{KS}}(\mathbb{T}/(\mm^\alpha,\gamma_{n}-1)\mathbb{T},\mathbb{F},\PP \cap \PP_{j})),$$
    the (generalized) module of $\LL$-adic Kolyvagin systems for the Selmer structure\footnote{As usual, we write $\mathbb{F}$ also for the \emph{propagation} of $\mathbb{T}$ to the quotients $\mathbb{T}/(\mm^\alpha,\gamma_{n}-1)\mathbb{T}$.} $\mathbb{F}$ on $\mathbb{T}$. Our definition slightly differs from that of Mazur an Rubin~\cite[Definition 3.1.6]{mr02}, however, as noted in Remark~\ref{1-kolsys}, it is possible to identify their generalized module Kolyvagin systems with ours using the fact that both $\{(\mm^\alpha,\gamma_n-1)\}_{\alpha,n}$ and $\{\mm_\LL^\beta\}_{\beta}$ (where $\mm_\LL$ is the maximal ideal of $\LL$) forms a base of neighborhoods at $0$.

  Suppose that $\Psi_0 \in \hhh$ is as in~\S\ref{subsubsec:ordcase}; in particular $\Psi_0$ satisfies $\hLi$. Let $\textbf{c}^{(r)}=\{c^{(r)}_K\}_{K\in \mathfrak{C}}$ be any Euler system of rank $r$ and let $\textbf{c}_{\Psi_0}=\{c_{k_n(\tau),\Psi_0}\}$ be the Euler system of rank one obtained from $\textbf{c}^{(r)}$ via Proposition~\ref{def: es} applied with $\Psi_0$. As before, let
  $$\pmb{\kappa}^{\Psi_0,\textup{Iw}} \in \overline{\textup{\textbf{KS}}}(\mathbb{T},\FF_\LL,\PP)$$
   be the image of $\textbf{c}_{\Psi_0}$ under the Euler system to Kolyvagin system map of  Theorem~\ref{thm:ESKSmain}. The proof of the following Theorem is very similar to the proof of Theorem~\ref{mainkolsys} above and will be skipped; see also the proofs of~\cite[Theorem 3.23]{kbbiwasawa} and~\cite[Theorem 2.19]{kbbstark}.
 \begin{thm}
  \label{thm:mainkolysys2}
$\pmb{\kappa}^{\Psi_0,\textup{Iw}} \in \overline{\textup{\textbf{KS}}}(\mathbb{T},\FF_{\all},\PP).$
  \end{thm}
  \begin{rem}
  \label{rem:KS1equalsES1}
  We know (by the definition of the Euler systems to Kolyvagin systems map) that
$$
\xymatrix@R=.2cm @C=.2cm{ {\kappa}^{\Psi_0,\textup{Iw}}_1\ar@{=}[d]  & \ar@{=}[l]_(.87){\textup{def}} \varprojlim_{\alpha,n} \kappa_{1}^{\Psi_0,\textup{Iw}}(\alpha,n) \in \varprojlim_{\alpha,n}H^1(k,\mathbb{T}/(\mm^\alpha,\gamma_n-1)\mathbb{T})=H^1(k,\mathbb{T})\\
 \{c_{k_n,\Psi_0}\}_{_n}&\ar@{=}[l]_(.72){\textup{def}}\left\{\varphi_{k_n}(c^{(r)}_{k_n})\right\}_{_n} \in \varprojlim_n H^1(k_n,T)=H^1(k,\mathbb{T}).
}$$
\end{rem}


 \section{Applications}
 \label{sec:applications}
 Throughout this section, the hypotheses $\hone$-$\hfive$, $\hne,\hwd$ are in effect.

 \subsection{Applications over $k$}
 \label{subsec:applicationk}
Aside from the hypotheses we assumed above, suppose in \S\ref{subsec:applicationk} that the hypothesis $\hpss$\, holds as well.

We start with proving a bound on the size of the dual Selmer group $H^1_{\FF_{\al}^*}(k,T^*)$, which we will use, together with the comparison result from~\S\ref{subsubsec:compareselmer1}, to obtain a bound on the classical  Selmer group.

\begin{thm}
\label{thm:mainmodifiedk-1}
Under the running hypotheses,
$$|H^1_{\FF_{\al}^*}(k,T^*)| \leq |H^1_{\FF_\al}(k,T)/\oo\cdot\kappa_1^{\Psi_0}|,$$ with equality if and only if the Kolyvagin system $\pmb{\kappa}^{\Psi_0} \in  \overline{\KS}(T,\FF_{\al},\PP)$ is primitive \textup{(}in the sense of~\cite[Definition 4.5.5]{mr02}\textup{)}.
\end{thm}
\begin{proof}
This is the standard application of $\pmb{\kappa}^{\Psi_0} \in \overline{\KS}(T,\FF_{\al},\PP)$, see~\cite[Corollary 5.2.13]{mr02}.
\end{proof}
Consider the following condition on the Euler system $\textbf{c}^{(r)}$ of rank $r$:
\begin{itemize}
\item[\textbf{(H.nV)}] 
$\textup{loc}_p^{s}\left(c_k^{(r)}\right) \neq 0.$
\end{itemize}
\begin{rem}
\label{rem:nonvanishingfirstterm}
\begin{enumerate}
\item[(i)] We give a conjectural example of an Euler system of rank $r$ in~\S\ref{subsec:PRpadic} based on Perrin-Riou's conjectures on $p$-adic $L$-functions. We will see then that the hypothesis $\textbf{H.nV}$ above amounts to saying that an associated $L$-value does not vanish. See the proof of Theorem~\ref{thm:mainconjappl} for details.
\item[(ii)] A similar hypothesis also appears in~\cite{kbbstark,kbbiwasawa} (as \textbf{H-F} in~\cite[\S3]{kbbstark}; see also ~\cite[Remark 4.1]{kbbiwasawa}). Roughly speaking, Leopoldt's conjecture ensures that \textbf{(H.nV)} holds for the Euler system of rank $r$ obtained from the Rubin-Stark elements in~\cite{kbbstark,kbbiwasawa}.
\end{enumerate}
\end{rem}
\begin{lemma}
\label{lem:nonvanishingKS}
Suppose \textbf{\textup{(H.nV)}} holds. Then $\textup{loc}_p^{s}(\kappa_1^{\Psi_0}) \neq 0$, in particular, $\kappa_1^{\Psi_0} \neq 0$.
\end{lemma}
\begin{proof}
The following equalities follow from the definitions:
\be\label{eqn:longone}\textup{loc}_p^s(\kappa_1^{\Psi_0})=\textup{loc}_p^s\left(c_{k,\Psi_0}\right)=\textup{loc}_p^s\left(\varphi_k\left(c_k^{(r)}\right)\right)=\varphi_k \left(\textup{loc}_p^s\left(c_k^{(r)}\right)\right).
\ee
Since $\varphi_k:\wedge^r H^1_s(k_p,T) \ra \al$ is an isomorphism and since we assumed  \textbf{\textup{(H.nV)}}, Lemma follows.
\end{proof}
\begin{cor}
\label{cor:mainNV}
If  \textbf{\textup{(H.nV)}} holds, then
\begin{itemize}
\item[(i)] $H^1_{\FF_{\al}^*}(k,T^*)$ is finite,
\item[(ii)] $H^1_{\FF_\al}(k,T)$ is a free $\oo$-module of rank one.
\end{itemize}
\end{cor}
\begin{proof}
By Lemma~\ref{lem:nonvanishingKS} and~\cite[Corollary 5.2.13(i)]{mr02} applied to $\pmb{\kappa}^{\Psi_0}\in \overline{\KS}(T,\FF_{\al},\PP)$, it follows that  $H^1_{\FF_{\al}^*}(k,T^*)$ is finite. We note that the theorem of Mazur and Rubin applies thanks to Proposition~\ref{modifiedcorerank}.

We have $H^1(k,T)_{\textup{tors}}\cong H^0(k,T\otimes\Phi/\oo)$ for the $\oo$-torsion submodule $H^1(k,T)_{\textup{tors}}$. As explained in~\cite[Lemma 3.5.2]{mr02},  it follows from our hypothesis $\hthree$ that $H^0(k,T\otimes\Phi/\oo)=0$. We therefore conclude that $H^1_{\FF_{\al}}(k,T)\subset H^1(k,T)$ is $\oo$-torsion free, hence it is a free $\oo$-module. Using~\cite[Corollary 5.2.6]{mr02}, we conclude that
\begin{align*}
\textup{rank}_{\oo}\left(H^1_{\FF_{\al}}(k,T)\right)&=\textup{rank}_{\oo}\left(H^1_{\FF_{\al}}(k,T)\right)-\textup{corank}_{\oo}\left(H^1_{\FF_{\al}^*}(k,T^*)\right)\\
&=\XX(T,\FF_\al)-\XX(T^*,\FF_{\al}^*)=1,
\end{align*}
where $\XX(T,\FF_\al)$ and $\XX(T^*,\FF_\al^*)$ denote the core Selmer rank, see~\S\ref{subsec:KS1}.
\end{proof}
\begin{cor}
\label{cor:mainvanishingselmer}
If \textbf{\textup{(H.nV)}} holds, then $H^1_{\FF_{\textup{BK}}}(k,T)=0$.
\end{cor}

\begin{proof}
By Lemma~\ref{lem:nonvanishingKS}, we have $\textup{loc}_p^s(\kappa_1^{\Psi_0}) \neq 0$, in particular, the map $\textup{loc}_p^s: H^1_{\FF_{\al}}(k,T)\ra \al$ is non-trivial. Since both $H^1_{\FF_\al}(k,T)$ and $\al$ are free $\oo$-modules of rank one, it follows that $\textup{loc}_p^s$ is injective, i.e.,
$$H^1_{\FF_{\textup{BK}}}(k,T)=\ker\left(H^1_{\FF_{\al}}(k,T)\stackrel{\textup{loc}_p^s}{\lra} \al \right)=0.$$
\end{proof}
\begin{thm}
\label{thm:maink}
Under the hypothesis \textbf{\textup{(H.nV)}},
$$|H^1_{\FF_{\textup{BK}}^*}(k,T^*)|\leq |\al/\oo\cdot\textup{loc}_p^s(\kappa_1^{\Psi_0})|,$$
and we have equality if and only if $\pmb{\kappa}^{\Psi_0}\in \overline{\KS}(T,\FF_{\al},\PP)$ is primitive.
\end{thm}

\begin{proof}
This follows from Theorem~\ref{thm:mainmodifiedk-1} and Corollary~\ref{cor:compareoverkwithclass} applied with the class $c=\kappa_1^{\Psi_0} \in H^1_{\FF_\al}(k,T)$. Note that Corollary~\ref{cor:compareoverkwithclass} applies thanks to Corollary~\ref{cor:mainvanishingselmer}.
\end{proof}

\begin{cor}
\label{cor:maink}
\begin{itemize}
\item[(i)] $| H^1_{\FF_{\textup{BK}}^*}(k,T^*)|\,\leq\, |\wedge^r H^1_s(k,T)/\oo\cdot\textup{loc}_p^s(c_k^{(r)})\mid.$
\item[(ii)] Suppose  \textbf{\textup{(H.nV)}} holds. We then have equality in \textup{(i)} if and only if the inequality of Theorem~\ref{thm:maink} is an equality.
\end{itemize}
\end{cor}

\begin{proof}
By construction,
\be
\label{eqn:nodiagram}
\xymatrix @C=.4cm @R=.2cm{ \varphi_k:\wedge^r H^1_s(k_p,T) \ar[r]^(.73){\sim}& \al\\
\textup{loc}_p^s(c_k^{(r)})\ar@{{|}->}[r]&\textup{loc}_p^s(\kappa_1^{\Psi_0})
}
\ee
If \textbf{\textup{(H.nV)}} fails, then there is nothing to prove, hence we may assume without loss of generality that \textbf{\textup{(H.nV)}} holds. In this case, Corollary follows from Theorem~\ref{thm:maink} and (\ref{eqn:nodiagram}) above.

\end{proof}
\subsection{Applications over $k_\infty$}
\label{subsec:mainthmkinfty}
Along with the hypotheses we set at the beginning of~\S\ref{sec:applications}, suppose also that $\htam,\,\hord$\, and $\hntz$\, hold. Recall that we write $\textup{char}(M)$ for the characteristic ideal of a finitely generated $\LL$-module $M$, with the convention that $\textup{char}(M)=0$ unless $M$ is $\LL$-torsion.

We proceed as in the previous section: We first prove a bound for the characteristic ideal of the dual Selmer group $H^1_{\FF_{\all}^*}(k,\mathbb{T}^*)^{\vee}$, which we use, together with Proposition~\ref{prop:compare selmer over k_infty}, to obtain a bound on the characteristic ideal of the (Pontryagin dual of the) classical Selmer group.

Let $\pmb{\kappa}^{\Psi_0,\textup{Iw}} \in \overline{\KS}(\TT,\FF_{\all},\PP)$ be the $\LL$-adic Kolyvagin system obtained from an Euler system of rank $r$ as in \S\ref{subsubsec:KSoverkinftybis}. Note that $\pmb{\kappa}^{\Psi_0,\textup{Iw}}$ maps to $\pmb{\kappa}^{\Psi_0} \in \overline{\KS}(T,\FF_\al,\PP)$ under the map
$$\overline{\KS}(\TT,\FF_{\all},\PP) \lra  \overline{\KS}(T,\FF_{\al},\PP).$$
We note that  $\FF_\al$ in this section is defined using the Greenberg local condition (see Remark~\ref{rem:BKvsGr}), whereas $\FF_{\al}$ that we used in the previous section is defined by relaxing Bloch-Kato local conditions (see \S\ref{subsub:lock}).
\begin{thm}
\label{thm:mainmodifiedkinfty-2}
Under the running hypotheses:
\begin{itemize}
\item[(i)] $\textup{char}\left(H^1_{\FF_{\all}^*}(k,\TT^*)^\vee\right) \Big{|} \,\textup{char}\left( H^1_{\FF_{\all}}(k,\TT)/\LL\cdot\kappa_1^{\Psi_0,\textup{Iw}}\right).$\\
\item[(ii)] The divisibility in \textup{(i)} is an equality if $\pmb{\kappa}^{\Psi_0} \in \overline{\KS}(T,\FF_\al,\PP)$ is primitive.
\end{itemize}
\end{thm}
\begin{proof}
(i) is \cite[Theorem 5.3.10(i)]{mr02}, and the assertion (ii) follows from~\cite[Theorem 5.3.10(iii)]{mr02}, once we check that $\pmb{\kappa}^{\Psi_0,\textup{Iw}}$ is $\LL$-primitive (in the sense of~\cite[Definition 5.3.9]{mr02}), provided that  $\pmb{\kappa}^{\Psi_0} \in \overline{\KS}(T,\FF_\al,\PP)$ is primitive. This is what we verify now.

Let $\overline{T}$ be the residual representation $\mathbb{T}/\mm_\LL\mathbb{T}=T/pT$. For a Kolyvagin system $\pmb{\kappa} \in \overline{\KS}(\mathbb{T})$ (resp., $\kappa \in \overline{\KS}({T})$), let $\overline{\pmb{\kappa}}$ (resp., $\overline{\kappa}$) denote the image of $\pmb{\kappa}$ (resp., $\kappa$) under the map $\overline{\KS}(\mathbb{T})\ra{\KS}(\overline{T})$ (resp., under the map $\overline{\KS}({T})\ra{\KS}(\overline{T})$). Since ${\pmb\kappa}^{\Psi_0, \textup{Iw}}$ maps to the element ${\pmb\kappa}^{\Psi_0}$ under the map $\overline{\KS}(\mathbb{T})\ra \overline{\KS}({T}),$ it is clear that $\overline{\pmb\kappa}^{\Psi_0,\textup{Iw}}=\overline{\pmb{\kappa}}^{\Psi_0}$, and we henceforth write $\overline{\kappa}$ for both. By our assumption that $\pmb{\kappa}^{\Psi_0}$ is primitive, it follows that $\overline{\kappa}\neq 0$. This proves that the image of ${\pmb\kappa}^{\Psi_0,\textup{Iw}}$  under the map
$\overline{\KS}(\mathbb{T}) \ra \overline{\KS}(\mathbb{T}/\frak{p}\mathbb{T})$
 is non-zero for any height-one prime $\frak{p} \subset \LL$; since we have a commutative diagram
 $$\xymatrix@C=.05pt@R=.5pt{\pmb{\kappa}^{\Psi_0,\textup{Iw}}\ar@{{|}->}[dd] &\in& \overline{\KS}(\mathbb{T})\ar[dd]\ar[rrrrd]&&&&\\
 &&&&&&\overline{\KS}(\mathbb{T}/\frak{p}\mathbb{T})\ar[lllld]\\
 \overline{\kappa}&\in&{\KS}(\overline{T})&&&&
 }$$ and $\overline{\kappa}\neq0$.
\end{proof}
\begin{cor}
\label{cor:mainkinfty-2}
Suppose the hypothesis \textup{\textbf{(H.nV)}} holds.
\begin{itemize}
\item[(i)] $\textup{char}\left(H^1_{\FF_{\textup{Gr}}^*}(k,\TT^*)^\vee\right) \Big{|}\, \textup{char}\left( \all/\LL\cdot\textup{loc}_p^s(\kappa_1^{\Psi_0,\textup{Iw}})\right).$
\item[(ii)] The inequality of \textup{(i)} is an equality if and only if $\pmb{\kappa}^{\Psi_0}$ is primitive.
\end{itemize}
\end{cor}
\begin{proof}
As in Corollary~\ref{cor:mainvanishingselmer}, \textup{\textbf{(H.nV)}} implies that $H^1_{\FF_{\textup{Gr}}}(k,T)$ vanishes. (i) now follows from Theorem~\ref{thm:mainmodifiedkinfty-2}(i) and Proposition~\ref{prop:compare selmer over k_infty}(ii) applied with the class $c=\kappa_1^{\Psi_0,\infty} \in H^1_{\FF_{\all}}(k,\TT)$. The assertion (ii) is immediate from Theorem~\ref{thm:mainmodifiedkinfty-2}(ii).
\end{proof}

Define $c_{k_\infty}^{(r)}:=\{c_{k_n}^{(r)}\}_{_n} \in \varprojlim_n \wedge^r_0 H^1(k_n,T)$. Recall that the subscript `0' here is to remind us that the elements $\{c_{k_n}^{(r)}\}$ are allowed to have \emph{denominators}. As explained in Remark~\ref{rem:locallyintegral}, the singular projections of these elements have no \emph{denominators}: $\textup{loc}_p^s(c_{k_n}^{(r)}) \in \wedge^r H^1_s((k_n)_p,T)$. Hence,
$$\textup{loc}_p^s(c_{k_\infty}^{(r)}):=\{c_{k_n}^{(r)}\} \in \varprojlim_{n} \wedge^r H^1_s((k_n)_p,T)=\wedge^r H^1_s(k_p,\TT),$$
where the last equality is because each $H^1_s((k_n)_p,T)$ is a free $\oo[\Gamma_n]$-module of rank $r$ and the maps $H^1_s(k_p,\TT) \ra H^1_s((k_n)_p,T)$ are all surjective.

\begin{thm}
\label{thm:mainkinfty}
Under the hypotheses of Corollary~\ref{cor:mainkinfty-2},
\begin{itemize}
\item[(i)] $\textup{char}\left(H^1_{\FF_{\textup{Gr}}^*}(k,\TT^*)^\vee\right) \Big{|}\, \textup{char}\left( \wedge^r H^1_s(k_p,\TT)/\LL\cdot\textup{loc}_p^s(c_{k_{\infty}}^{(r)})\right),$
\item[(ii)] the divisibility in (i) is an equality if and only if  $\pmb{\kappa}^{\Psi_0}$ is primitive.
\end{itemize}
\end{thm}

\begin{proof}
Recall $\varphi_{\infty}=\{\varphi_{k_n}\}_{_n}$, which we defined in Remark~\ref{interchange}. By definition, we have the following diagram:
$$
\xymatrix @C=.4cm @R=.2cm{ \varphi_\infty:\wedge^r H^1_s(k_p,\TT) \ar[r]^(.71){\sim}& \all\\
\textup{loc}_p^s(c_{k_\infty}^{(r)})\ar@{{|}->}[r]&\textup{loc}_p^s(\kappa_1^{\Psi_0,\infty})
}$$
(i) now follows from Corollary~\ref{cor:mainkinfty-2}(i) and the diagram above, and (ii) is immediate after Corollary~\ref{cor:mainkinfty-2}(ii).
\end{proof}
\subsection{Perrin-Riou's (conjectural) $p$-adic $L$-functions}
\label{subsec:PRpadic}
Rubin~ \cite[\S VIII]{r00} sets up a connection between Euler systems of rank $r$ and $p$-adic $L$-functions via the work of Perrin-Riou \cite{pr, pr-ast}. We will apply the results of~\S\ref{subsec:applicationk} and \S\ref{subsec:mainthmkinfty} with the (conjectural) Euler system of Perrin-Riou and Rubin. Since these Euler systems arise from $p$-adic $L$-functions, Corollary~\ref{cor:maink} and Theorem~\ref{thm:mainkinfty} will relate Selmer groups to $L$-values.
\subsubsection{The setting}
\label{subsubsec:settingPR}
For notational convenience, we restrict ourselves to the case $\Phi=\QQ_p$ and $\oo=\ZZ_p$. Let $\QQ_p(1)=\QQ_p\otimes\ZZ_p(1)$ and  $\QQ(j)=\QQ_p(1)^{\otimes j}$ for every $j \in \ZZ$. We also write $V(j)=V\otimes\QQ_p(j)$ for a Galois representation $V$, and $V^*=\Hom(V,\QQ_p(1))$. Throughout this section, we assume that the $G_k$-representation $V=T\otimes\QQ_p$ is the $p$-adic realization $M_p$ of a (pure) motive $M_{/k}$  in the sense of~\cite[\S III.2.1.1]{FPR91}.  Write $w=w(M)$ for the weight of $M$ and let $L(M,s)$ denote the $L$-function of $M$. This is defined as an  Euler product $$L(M,s)=\prod_{\ell} L_\ell(M,s)$$ which is absolutely convergent in the half-plane $\Re(s)>1+\frac{w}{2}$. We will assume without loss that $k=\QQ$; as in general one could consider the induced representation $\textup{Ind}_{k/\QQ}T$ in place of $T$. We will suppose further that the representation $V=M_p$ is crystalline at $p$.

Write $\check{M}$ for the dual motive. For the sake of simplicity, we shall be interested in the case of a \emph{self-dual motive} $M \stackrel{\sim}{\ra} \check{M}(1)$. In this case,  we have $w=-1$, and $s=0$ is the center of symmetry of the conjectural functional equation that the associated complex $L$-function $L(M,s)$ satisfies.  Serre's~\cite[\S3]{serre} general recipe implies that the Archimedean factor $L_\infty(M,s)$ at infinity is non-vanishing at $s=0$, hence the central point $s=0$ is critical in the sense of Deligne~\cite{deligne3}.

\begin{example} In the examples below, suppose $k$ is an arbitrary totally real field.
\begin{itemize}
\item[1.] Let $A$ be an abelian variety over $k$. Let $M$ be the motive $h^1(A)(1)$. The $p$-adic realization of $M$ is given by $M_p=\QQ_p\otimes T_p(A)$. Falting's~\cite{faltingstate} proof of the Tate conjecture implies that the motive $M$ determines the abelian variety $A$ up to an isogeny over $k$. Let $A^{\vee}$ denote the dual abelian variety, and fix a polarization $f:A\ra A^{\vee}$. This isogeny induces an isomorphism of motives $h^1(A) \stackrel{\sim}{\ra}h^1(A^{\vee})$ and  the Weil pairing  shows that $M \stackrel{\sim}{\lra} \check{M}(1)$, i.e., $M$ is self-dual. One has $L(M,s)=L(A_{/k},s+1)$, where $L(A_{/k},s)$ is the Hasse-Weil $L$-function attached to $A$. The study of $L(M,s)$ at the central critical point $s=0$ therefore amounts to the study of $L(A_{/k},s)$ at $s=1$. The representation $V$ is crystalline at $p$ if and only if $A$ has good reduction at $p$~(by the work of Fontaine \cite{fontaine1} for the ``if" part of this statement; and the ``only if" part by Coleman and Iovita~\cite{colemaniovita}, see also~\cite{mokrane} for the case when $A_{/\QQ_p}$ is potentially a product of Jacobians).
\\
\item[2.] Suppose that $f$ is a cuspidal Hilbert eigenform of even parallel weight $(w,w,\dots,w)$ (for brevity, we say of weight $w \in 2\ZZ^+$), of level $\frak{n}\subset \oo_k$ and central character $\varphi$. Thanks to~\cite[Proposition 1.3]{shimura}, there exists a number field $L_f$ such that its ring of integers $\oo_f:=\oo_{L_f}$ contains the values of $\varphi$ and all Hecke eigenvalues $\theta_f(\frak{a})$ for $(\frak{a},\frak{n})=1$. Let $\frak{p}$ be any prime of $L_f$ above $p$. The work of Carayol~\cite{carayol86}, Wiles~\cite{wileslambda}, Taylor~\cite{taylor89} and Blasius and Rogawski~\cite{blasiushilbert} attaches to $f$ a motive $M$ such that the $\frak{p}$-adic realization $M_{\frak{p}}=V_{\frak{p}}(f)$ is an irreducible~\cite{taylor95} two dimensional representation of $G_k$ over $L_{f,\frak{p}}$. (When $k=\QQ$, this construction is due to Eichler, Shimura, Deligne~\cite{deligne69} and Scholl~\cite{scholl90}.) The works of D. Blasius and J. Rogawski, C. Breuil, M. Kisin and T. Saito (under certain conditions) and T. Liu (in general) show that $V_{\frak{p}}(f)$ is crystalline at $\frak{p}$ if $(\frak{p},\frak{n})=1$. Let $\mathbb{A}_k$ denote the id\'eles of $k$, and suppose $\chi:\mathbb{A}_k/k^{\times} \ra L_f^{\times}$ is a character such that $\varphi=\chi^{-2}$. As Nekov\'a\v{r} explains in~\cite[\S12.5.5]{nekovar06}, the $G_k$-representation $V=V_{\frak{p}}(f)(w/2)\otimes\chi$ is self-dual in the sense that $V\stackrel{\sim}{\ra} \Hom(V,\QQ_p(1))$.
\end{itemize}

\end{example}

Let $B_{\textup{dR}}$ denote Fontaine's~\cite{fontaine82} field of $p$-adic periods; it is a discretely valued field whose valuation ring contains $\overline{\QQ}_p$. There is a natural descending filtration
$$\dots\supset B^i_{\textup{dR}}\supset B^{i+1}_{\textup{dR}} \supset \dots,$$
 which is obtained by letting $B^i_{\textup{dR}} \subset B_{\textup{dR}}$ to be the set of elements whose valuation is at least $i$. For an arbitrary Galois representation $W$ (which is finite dimensional over $\QQ_p$) and a finite extension $\frak{L}$ of $\QQ_p$, write $D_{\textup{dR}}(\frak{L},W)=H^0(\frak{L},B_{\textup{dR}}\otimes W)$, and $D_{\textup{dR}}(\QQ_p,W)=D_{\textup{dR}}(W)$. The filtration on $B_{\textup{dR}}$ induces a decreasing filtration $\{D^i_{\textup{dR}}(W)\}_{i\in \ZZ}$ on $D_{\textup{dR}}(W)$. One always has
$$\textup{dim}_{\frak{L}}\,D_{\textup{dR}}(\frak{L},W)\leq \textup{dim}_{\QQ_p}W$$ by \cite[\S 5.1]{fontaine82} and the $G_{\frak{L}}$-representation $W$ is called De Rham if $\textup{dim}_{\frak{L}}(D_{\textup{dR}}(\frak{L},W))=\textup{dim}_{\QQ_p}(W)$. A $G_{\QQ_p}$-representation $W$ is De Rham if and only if  it is De Rham as a $G_{\frak{L}}$-representation; and one has
$$\frak{L}\otimes_{\QQ_p} D_{\textup{dR}}(W) \stackrel{\sim}{\lra}D_{\textup{dR}}(\frak{L},W),$$
if $W$ is De Rham.

For any De Rham representation $W$ of $G_{\frak{L}}$ as above, Bloch and Kato~\cite{bk} construct a canonical homomorphism
$$\textup{exp}^*: H^1(\frak{L},W) \lra D_{\textup{dR}}^0(\frak{L},W)$$ called the dual exponential map. By its construction, it factors through the singular quotient $H^1_s(\frak{L},W)$. In Section~\ref{subsubsec:PRlog} below, we will explain Perrin-Riou's~\cite{pr} interpolation of the dual exponential maps for \emph{crystalline}\footnote{Kato claims in~\cite[Remark 16.5]{ka1} that this assumption is not necessary and refers to his preprint with Kurihara and Tsuji~\cite{kkt}.} representations (which we define next), as one climbs up the cyclotomic tower.

Let $B_{\textup{cris}}$ be Fontaine's crystalline period ring, see~\cite{fontaine94} for its construction and other properties we note here. For a $G_{\QQ_p}$-representation $W$ as above, let $D_{\textup{cris}}(W)=H^0(\QQ_p, B_{\textup{cris}}\otimes W)$ be Fontaine's filtered vector space associated to $W$ which is endowed by a Frobenius action. If $W$ is also a $G_{\QQ}$-representation, we set
 $$D_{\textup{cris}}(F,W)=D_{\textup{cris}}(\textup{Ind}_{F/\QQ}W)$$
 for a finite abelian extension $F$ of $\QQ$ which is unramified above $p$.

 For any $G_{\QQ_p}$-representation $W$, it is known that $D_{\textup{cris}}(W) \subset D_{\textup{dR}}(W)$, and hence
 $$\textup{dim}_{\QQ_p}\,D_{\textup{cris}}(W) \leq\textup{dim}_{\QQ_p}\,D_{\textup{dR}}(W)\leq \textup{dim}_{\QQ_p}W,$$ and we say that $W$ is crystalline if $\textup{dim}_{\QQ_p}\,D_{\textup{cris}}(W)= \textup{dim}_{\QQ_p}W.$ Hence, if $W$ is crystalline, then $W$ is De Rham as well, and one has $D_{\textup{cris}}(W)=D_{\textup{dR}}(W)$.

 We define one final ring which plays an important role in what follows. Define $G_\infty:=\Gal(\QQ(\pmb{\mu}_{p^\infty})/\QQ)=\Delta\times\Gamma$ where $\Delta=\Gal(\QQ(\pmb{\mu}_p)/\QQ)$ is a finite group of order prime to $p$, and $\Gamma$ is defined as before. Set $G_n=\Gal(\QQ(\pmb{\mu}_{p^n})/\QQ)$. Fixing a topological generator $\gamma$ of $\Gamma$, we may identify $\ZZ_p[[G_\infty]]$ with the power series ring $\ZZ_p[\Delta][[\gamma-1]]$ over the group ring $\ZZ_p[\Delta]$. For any integer $h\geq 1$, set

\begin{align*}
\mathcal{H}_{h}=\Bigg\{\sum_{\substack{ n\geq 0,\\ \delta \in \Delta}} a_{n,\delta}\cdot\delta\cdot(\gamma-1)^n &\in \QQ_p[\Delta][[\gamma-1]]: \\
& \lim_{n\ra\infty} |a_{n,\delta}|_p\cdot n^{-h}=0,\hbox{ for every } \delta \in \Delta \Bigg\},
\end{align*}
 where $|\cdot|_p$ is the $p$-adic norm on $\QQ_p$, normalized by setting $|p|_p=\frac{1}{p}$. Define $\mathcal{H}_\infty=\varinjlim_{h} \mathcal{H}_h$. Any continuous character $\chi: G_\infty \ra \overline{\QQ}_p^\times$ induces a homomorphism $\mathcal{H}_\infty \ra \overline{\QQ}_p^\times$, which we still denote by $\chi$. We write $\rho_{\textup{cyc}}$ for the cyclotomic character
 $$\rho_{\textup{cyc}}: G_{\infty} \stackrel{\sim}{\lra} \ZZ_p^\times$$
 and following~\cite[\S 4.1.5]{pr},  we say that $\rho$ is a geometric character of $G_\infty$ if there is an integer $j=j_\rho$ such that $\rho_{\textup{cyc}}^{-j}\cdot\rho=\chi_\rho$ is a character of finite order.

 Finally, for every field $F$ and a $G_F$-module $T$ which is free of finite rank over $\ZZ_p$, write $$H^1_\infty(F,T)=\varprojlim_{n} H^1(F(\pmb{\mu}_{p^n}),T),$$ and if $W=T\otimes\QQ_p$, write $H^1_\infty(F,W)=\QQ_p\otimes H^1_\infty(F,T)$.
\subsubsection{Perrin-Riou's extended logarithm and conjectures}
\label{subsubsec:PRlog} We are now ready to state a theorem due to
Perrin-Riou~\cite{pr}. The statement below follows ~\cite[Theorem \S
16.4]{ka1}, which is Perrin-Riou's Theorem enhanced by Kato,
Kurihara and Tsuji~\cite{kkt} to cover any De Rham representation at
$p$ that satisfies $D_{\textup{cris}}(W^*) \subset
D^0_{\textup{dR}}(W^*)$.
\begin{thm}[Perrin-Riou, Kato-Kurihara-Tsuji]
\label{thm:PRslog}
Suppose $W$ is a $G_{\QQ}$-representation which is finite dimensional as a $\QQ_p$-vector space.
Assume $W$ is De Rham at $p$ and  $D_{\textup{cris}}(W^*) \subset D^0_{\textup{dR}}(W^*)$. Then for every finite extension $F/\QQ$ in which $p$ is unramified, there is a unique homomorphism
$$\frak{Log}^F: H^1_\infty(F,W) \lra \mathcal{H}_\infty \otimes_{\QQ_p}D_{\textup{cris}}(F,W)$$
 which  satisfies the following properties \textup{(i)-(ii)}, for every $\eta \in D_{\textup{cris}}(W^*)$ and for every integer $j\geq 1$:
 \begin{itemize}
 \item[(i)] Let $\frak{Log}^F_\eta$ be the composite map
 $$\frak{Log}^F_\eta: H^1_\infty(F,W) \stackrel{\frak{Log}^F}{\lra}  \mathcal{H}_\infty \otimes_{\QQ_p}D_{\textup{cris}}(F,W) \stackrel{\eta}{\lra}  \mathcal{H}_\infty \otimes_{\QQ_p} F,$$
  where the second map is induced from the canonical pairing
  $$D_{\textup{dR}}(W) \times D_{\textup{cris}}(W^*) \lra \QQ_p$$
  and from
  $$D_{\textup{cris}}(F,W) \subset D_{\textup{dR}}(F,W)\cong F\otimes D_{\textup{dR}}(W).$$
  Then for $n\geq 1$, for every character $\chi:G_n\ra \overline{\QQ}_p^\times$ which does not factor through $G_{n-1}$ and for any $x \in  H^1_\infty(F,W)$, we have
  $$\rho_{\textup{cyc}}^j\chi^{-1}\left(\frak{Log}^F_\eta \right) \approx  \sum_{\sigma \in G_n}\chi(\sigma)\langle \sigma(\textup{exp}^*(x_{-j,n})), (p^{-j}\varphi)^{-n}(\eta)\rangle.$$
 Here:
 \begin{itemize}
  \item\textup{`}$\approx$\textup{'} means equality up to simple non-zero factors which are omitted for brevity,
  \item$\varphi$ is the crystalline Frobenius,
  \item  $x_{-j,n}$ is the image of $x$ under the composite
 $$H^1_{\infty}(F_p,W) \stackrel{\sim}{\lra} H^1_{\infty}(F_p,W(-j)) \stackrel{\textup{proj}}{\lra} H^1(F(\pmb{\mu}_{p^n})_p, W(-j)), $$
 \item $\textup{exp}^*$ is the semi-local Bloch-Kato dual exponential
 \begin{align*}\textup{exp}^*: H^1(F(\pmb{\mu}_{p^n}),W(-j)) &\lra  D^0_{\textup{dR}}(F(\pmb{\mu}_{p^n}),W(-j))\\
 &\,\,\,\,\subset D_{\textup{dR}}(F(\pmb{\mu}_{p^n}),W(-j))\\
 &\,\,\,\,=D_{\textup{dR}}(F(\pmb{\mu}_{p^n}),W)\\
 &\,\,\,\, =F(\pmb{\mu}_{p^n})\otimes D_{\textup{dR}}(W)
 \end{align*}
 \item $\langle\,,\,\rangle$ is the pairing
 $$F(\pmb{\mu}_{p^n})\otimes D_{\textup{dR}}(W) \times D_{\textup{cris}}(W^*) \lra F(\pmb{\mu}_{p^n})\otimes\QQ_p$$
 induced from the pairing $D_{\textup{dR}}(W) \times D_{\textup{cris}}(W^*) \lra \QQ_p$.
 \end{itemize}
 \medskip
 \item[(ii)] Suppose $\eta=(1-p^{-j}\varphi)\eta^{\prime}$,
  with $\eta^{\prime} \in  D_{\textup{cris}}(W^*)$, and let $\frak{Log}^F_\eta$ be as in \textup{(i)}. Then for any $x \in H^1_{\infty}(F_p,W)$ we have
 $$\rho_{cyc}^{j}(\frak{Log}^F_\eta(x))=(j-1)! \cdot \langle\textup{exp}^*(x_{-j,0}), (1-p^{j-1}\varphi^{-1})\eta^{\prime} \rangle.$$
 \end{itemize}
\end{thm}

Let $M_{/\QQ}$ be a pure motive. For a geometric character $\rho$ of $G_\infty$, set $M(\rho)=M(j_\rho)(\chi_\rho)$. For every positive integer $\frak{f}$, one can then attach $M(\rho)$ a complex $L$-function with Euler factors at primes dividing $\frak{f}$ removed:
$$L_{\frak{f}}(M(\rho),s)=\prod_{\ell \nmid \frak{f}} L_{\ell} (M(\rho),s)^{-1}.$$
Here, for a prime $\ell\neq p$ at which the $p$-adic realization $M(\rho)_p$ is unramified, the Euler factor at $\ell$ is given by
$$L_\ell(M,s)=\textup{det}\left(1-\textup{Fr}_{\ell}^{-1}x\mid M(\rho)_p\right)\Big{|}_{s=\ell^{-s}}.$$

Let $\mathfrak{K}=\textup{Frac}(\mathcal{H}_\infty)$, the fraction field of $\mathcal{H}_\infty$. Write $d_{-}=\textup{dim}\,M_p^{-}$ for the dimension of the $(-1)$-eigenspace of a complex conjugation acting on the $p$-adic realization $M_p$ which we henceforth assume to be crystalline.

\begin{conj}[Perrin-Riou~\cite{pr-ast} \S4.2.2]
\label{conj:prpadicL}
For every positive integer $\frak{f}$ which is prime to $p$ and to every prime at which $M_p$ is ramified, there exists an element $\frak{l}_{\frak{f}}(M) \in \frak{K}\otimes\wedge^{d_{-}}D_{\textup{cris}}(M_p)$ and $\overline{\eta}=\eta_1\wedge \dots \wedge \eta_{d_{-}} \in \wedge^{d_{-}}D_{\textup{cris}}(M_p^*)$ such that
$$\mathbf{L}_{\frak{f}}^{(p)}(M)=\overline{\eta}(\frak{l}_{\frak{f}}(M)) \in \frak{K}$$
 is the \textup{`}$p$-adic $L$-function\textup{'} attached to $M$, which interpolates the special values of the complex $L$-functions attached to twists of $M$ by geometric characters, with their Euler factors at primes dividing $\frak{f}$ removed.
\end{conj}
See~\cite[\S 4.2]{pr-ast} for a detailed description of the properties which characterize this $p$-adic $L$-function. The statement above is Rubin's extrapolation~\cite[Conjecture VIII.2.1]{r00} of Perrin-Riou's conjecture by introducing the level $\frak{f}$. The interpolation property alluded to above (roughly) reads as follows:

For every geometric character $\rho$ of $G_\infty$ such that $\chi_\rho(p)\cdot p^{j_\rho}$ and $\overline{\chi}_{\rho}(p)\cdot p^{-j_\rho-1}$ are not eigenvalues of $\varphi$ on $D_{\textup{cris}}(M_p)$,
\be\label{eq:interpolatiponprop}
\rho^{-1}(\mathbf{L}_{\frak{f}}^{(p)}(M))=\mathcal{E}_p (M(\rho)) \times \frac{L_{\frak{f}}(M(\rho),0)}{\textup{Per}_{\infty}(M(\rho))}\times \textup{Per}_p(M(\rho))
\ee
where $\mathcal{E}_p(M(\rho))$ is the Euler factor at $p$ and $\textup{Per}_{\infty}(M(\rho))$ (resp., $\textup{Per}_{p}(M(\rho))$) is the archimedean (resp., $p$-adic) period attached to $M(\rho)$, see~\cite[\S 3.1 and 4.1.4]{pr-ast}.
\subsubsection{Connection with Euler systems of rank $r$}
\label{subsubsec:PRpadicES}
Let $M_{/\QQ}$ be a pure motive as above, and let $M_p$ be its $p$-adic realization which is crystalline. Fix a $G_\QQ$-stable lattice $\mathcal{T} \subset M_p$ and an integer $\mathcal{B}=\mathcal{B}(\mathcal{T})$ which is divisible by $p$ and all bad primes for $M_p$.
Until the end of this section we assume the following conditions hold for $\mathcal{T}$:
\begin{itemize}
\item[(A)] $H^0(\QQ_p(\pmb{\mu}_p),\mathcal{T}^*)=0$,
\item[(B)] $H^0(\QQ_{p}(\pmb{\mu}_{p^{\infty}}),\mathcal{T})=0$.
\end{itemize}
where $\mathcal{T}^*=\Hom(\mathcal{T},\pmb{\mu}_{p^{\infty}})$ is as before. The conditions above are the hypotheses $\hne$ and $\hwd$ with $k=\QQ(\pmb{\mu}_p)$, and as in \S\ref{subsub:lock}, one may prove under these conditions that:
\begin{enumerate}
\item[(i)] $H^1_{\infty}(\QQ_p,\mathcal{T})$ is a free $\ZZ_p[[G_\infty]]$-module of rank $d=\dim M_p$,
\item[(ii)] the canonical projection $H^1_{\infty}(\QQ_p,\mathcal{T}) \ra H^1(\QQ_p(\pmb{\mu}_{p^n}),\mathcal{T})$ is surjective,
\item[(iii)] $H^1(\QQ_p(\pmb{\mu}_{p^n}),\mathcal{T})$ is a free $\ZZ_p[G_n]$-module of rank $d$.
\end{enumerate}
Furthermore, as noted in Remark~\ref{rem:locallyintegral}, these together with~\cite[Example (1), page 38]{ru96} show that
\begin{itemize}
\item[(1)] $\wedge_0^r H^1(\QQ_p(\pmb{\mu}_{p^n}),\mathcal{T})=\wedge^r H^1(\QQ_p(\pmb{\mu}_{p^n}),\mathcal{T})$,
\item[(2)] $\wedge^r H^1_\infty(\QQ_p,\mathcal{T})=\varprojlim_n \wedge^r H^1(\QQ_p(\pmb{\mu}_{p^n}),\mathcal{T})$,
\end{itemize}
where the exterior products in (1) is taken in the category of $\ZZ_p[G_n]$-modules, whereas in (2), the exterior products are taken in the category of $\ZZ_p[[G_\infty]]$-modules.

 Finally, assume that the weak Leopoldt  conjecture (see~\cite{pr-ast} \S1.3) holds for the representation $\Hom_{\ZZ_p}(\mathcal{T},\ZZ_p(1))$.

For any integer $\frak{f}$, write $R_{\frak{f}}=\QQ(\pmb{\mu}_{\frak{f}})^+$ for the maximal real field of $\QQ(\pmb{\mu}_{\frak{f}})$ and define
$$\mathcal{C}=\bigcup_{\substack{(\frak{f},\mathcal{B})=1\\ n\geq 1}} R_{\frak{f}}(\pmb{\mu}_{p^n}).$$
For notational consistency, we write $r=d_{-}=d_{-}(M_p)$. Recall that an Euler system of rank $r$ for the pair $(\mathcal{T},\mathcal{C})$ is a collection $\textbf{c}^{(r)}=\left\{c^{(r)}_K\right\}_{K\subset \mathcal{C}}$ with the properties that
\begin{itemize}
\item $c^{(r)}_K \in \wedge_0^r H^1(K,\mathcal{T})$,
\item  for $K\subset K^{\prime} \subset \mathcal{C}$ such that $K^{\prime}/\QQ$ is a finite extension,
$$\textup{Cor}^r_{K^{\prime}/K}\left(c_{K^{\prime}}\right)=\left(\prod_{\qq} P_{\qq}(\textup{Fr}_{\qq}^{-1})\right)c_{K},$$
where the product is over all rational primes $\qq \nmid \mathcal{B}$ which does not ramify in $K/\QQ$, but does ramify in $K^{\prime}/\QQ$.
\end{itemize}
See \S\ref{sec:ES} above for further details.

For any number field $K$, write as usual
$$\textup{loc}_p: H^1(K,\mathcal{T}) \lra H^1(K_p,\mathcal{T})$$
 for the localization map at $p$.  If $\mathbf{c}^{(r)}=\left\{c^{(r)}_K\right\}_{K\in \kk}$ is an Euler system of rank $r$ for $(\mathcal{T},\mathcal{C})$, we may regard $\textup{loc}_p\left(c^{(r)}_\infty\right):=\left\{\textup{loc}_p\left(c_{\QQ(\pmb{\mu}_{p^n})}^{(r)}\right)\right\}_n$ as an element of $\wedge^r H^1_\infty(\QQ_p,\mathcal{T})$, and apply Perrin-Riou's extended logarithm
 $$\frak{Log}^{\otimes r}: \wedge^r H^1_\infty(\QQ_p,\mathcal{T}) \lra \frak{K} \otimes \wedge^r D_{\textup{cris}}(\mathcal{M}_p)$$
  on it. Here we write $\frak{Log}$ for $\frak{Log}^{\QQ_p}$ above.

\begin{conj}[\cite{pr-ast} \S 4.4 and \cite{r00} Lemma VIII.5.1]
\label{conj:specialelts}
Assuming the hypotheses above, there exists an Euler system $\mathbf{c}^{(r)}=\left\{c^{(r)}_K\right\}_{K\in \mathcal{C}}$ of rank $r$ for $(\mathcal{T},\mathcal{C})$ so that
$$\overline{\eta}\left(\frak{Log}^{\otimes r} \left(\textup{loc}_p\left(c^{(r)}_\infty\right)\right)\right)=\mathbf{L}^{(p)}(M),$$
where $\overline{\eta}=\eta_1\wedge \dots \wedge \eta_r \in \wedge^r D_{\textup{cris}}(M_p^*)$, and $\mathbf{L}^{(p)}(M)=\mathbf{L}^{(p)}_1(M)$ is as in Conjecture~\ref{conj:prpadicL}.
\end{conj}
We will write $\frak{Log}^{\otimes r}_{\overline{\eta}}$ as a short-cut for the composite $$\overline{\eta}\circ\frak{Log}^{\otimes r}:\wedge^r H^1_\infty(\QQ_p, \mathcal{T}) \ra \mathcal{H}_\infty.$$

\subsubsection{Applications}
\label{subsec:applications} We apply the results of~\S\ref{sec:applications} together with the (conjectural) Euler system of rank $r$ given in~Conjecture~\ref{conj:specialelts}.

Suppose $V$ is the $p$-adic realization of a fixed self-dual pure motive $\MM\stackrel{\sim}{\ra}\check{\MM}(1)$ defined over $k=\QQ$ and with coefficients in $L=\QQ$. As remarked before, taking $k=\QQ$ is not too serious as one may always consider $\textup{Ind}_{k/\QQ}\MM$ in place of $\MM$; and the assumption that $L=\QQ$ is only made for notational convenience. The $p$-adic realization $V$ is then a finite dimensional $\QQ_p$-vector space endowed with a $G_\QQ$-action, which is unramified outside a finite set of places. We will also assume that $V$ is crystalline at $p$. Fix a $G_\QQ$-stable lattice $T \subset V$. We assume until the end of this paper that $T$ satisfies the hypotheses (A) and (B) from the previous section, as well as $\hone$-$\hfive$ from the introduction.

Along with the motive $\MM_{/\QQ}$, we will consider its Tate-twists $\MM(j)$ for very large integers $j$; the $p$-adic realization $\MM(j)_p$ of $\MM(j)$ is $V(j)=V\otimes \QQ_p(j)$. The $G_{\QQ}$-representation $V(j)$ is also unramified outside a finite set of places and is crystalline at $p$. We write $T(j)=T\otimes\ZZ_p(j)$ which is naturally a lattice inside $V(j)$.
\begin{lemma}
For any $j$, the hypotheses \textup{(A)} and \textup{(B)} hold for $T(j)$. 
\end{lemma}
\begin{proof}
(B) obviously holds for $T(j)$ if it holds for $T$. Let $\LL=\ZZ_p[[\Gamma]]$ with $\Gamma=\Gal(\QQ(\pmb{\mu}_{p^\infty})/\QQ(\pmb{\mu}_{p}))$ as usual. The statement of (A) for $T$ is equivalent to the vanishing of $H^2(\QQ_p(\pmb{\mu}_p),T\otimes\LL)=0$ (see the proof of Lemma~\ref{lemma:free for k}), and the proof of Lemma follows using the natural isomorphism
$$H^2(\QQ_p(\pmb{\mu}_p),T\otimes\LL) \lra H^2(\QQ_p(\pmb{\mu}_p),T(j)\otimes\LL).$$
\end{proof}

Fix a large enough $j \in 2\ZZ$ so that $D^0_{\textup{dR}}(V(j)^*)=D_{\textup{dR}}(V(j)^*)$. Such an integer $j$ exists because
$$D^0_{\textup{dR}}(V(j)^*)=D^0_{\textup{dR}}(V^*(-j))=D^{-j}_{\textup{dR}}(V^*).$$
Since we insist that $j$ is even, it follows that $r=\dim(V^-)=\dim (V(j)^-)$.

Assume that the weak Leopoldt conjecture holds for the representation $\Hom(T(j),\ZZ_p(1))\cong T(-j)$, and suppose that the Conjecture~\ref{conj:specialelts} holds for $M=\MM(j)$.
\begin{thm}
\label{thm:mainconjappl}
Suppose $1$ is not an eigenvalue for the action of $\varphi$ on $D_{\textup{cris}}(V)$, and assume that $L(\mathcal{M},0)\neq 0$. Then the Bloch-Kato Selmer group $H^1_{\FF_{\textup{BK}}^*}(\QQ,T^*)$ is finite.
\end{thm}
 \begin{rem}
 \label{rem:froboncrisexceptional} Since $V$ is self-dual, it follows that $1$ is an eigenvalue of $\varphi$ acting on $D_{\textup{cris}}(V)$  if and only if  $p^{-1}$ is an eigenvalue (as there is a perfect pairing $D_{\textup{cris}}(V)\otimes D_{\textup{cris}}(V^*) \ra \QQ_p(1)$ (c.f.,~\cite[Lemme 1.4.3]{benoisberger}) and $\varphi$ acts on $\QQ_p(1)$ by multiplication by $p^{-1}$). The assumption that $1$ is not an eigenvalue (therefore neither $p^{-1}$) is to rule out the possibility that the $p$-adic $L$-function $\mathbf{L}^{(p)}(\MM)$ may have an \emph{exceptional zero} at the trivial character of $G_\infty$ (c.f., \cite[Remark 0.4]{benoistrivial}).

 Note also that $1$ (resp., $p^{-1}$) is an eigenvalue of $\varphi$ acting on $D_{\textup{cris}}(V^*)=D_{\textup{cris}}(V)$ if and only if $p^{j}$ (resp., $p^{j-1}$) is an eigenvalue of $\varphi$ acting on $D_{\textup{cris}}(V(j)^*)$. 
 In particular, under the assumption that  $1$ is not an eigenvalue for $\varphi \big|_{ D_{\textup{cris}}(V)}$, the operators $1-p^{-j}\varphi$ and $1-p^{j-1}\varphi^{-1}$ acting on $D_{\textup{cris}}(V(j)^*)$ (which appear in the statement of Theorem~\ref{thm:PRslog}(ii)) are both invertible.
  \end{rem}

 \begin{proof}[Proof of Theorem~\ref{thm:mainconjappl}]
 Let $\mathbf{c}^{(r)}(j)$ denote the Euler system of rank $r$ for the pair $(T(j),\mathcal{C})$ predicted by Conjecture~\ref{conj:specialelts}, where $\mathcal{C}$ is as in the previous section. Applying Rubin's twisting formalism~\cite[\S VI]{r00}, we obtain an Euler system $\mathbf{c}^{(r)}=\{c^{(r)}_K\}_{K\in \mathcal{C}}$ of rank $r$ for $(T,\mathcal{C})$. Corollary~\ref{cor:maink} gives an inequality
 $$| H^1_{\FF_{\textup{BK}}^*}(\QQ,T^*)|\,\leq\, |\wedge^r H^1_s(\QQ,T)/\ZZ_p\cdot\textup{loc}_p^s(c_{\QQ}^{(r)})\mid,$$
 and the theorem is proved once we verify that $\textup{loc}_p^s(c_{\QQ}^{(r)}) \neq 0$.

 Let $c_\infty^{(r)}(j)=\{c_{\QQ(\pmb{\mu}_{p^n})}^{(r)}(j)\}_n \in H^1_\infty(\QQ,T(j))$, and consider
 \be\label{eqn:recalldef}\rho_{\textup{cyc}}^{j}\mathbf{L}^{(p)}(\MM(j))=\rho_{\textup{cyc}}^{j}\frak{Log}_{\overline{\eta}}^{\otimes r} \left(\textup{loc}_p\left(c^{(r)}_\infty(j)\right)\right)\ee
where the equality follows from the defining property of $\mathbf{c}^{(r)}(j)$. If we take $j$ large enough 
 and assume that $1$ is not an eigenvalue for $\varphi\big|_{D_{\textup{cris}}(V)}$, one may calculate  $\rho_{\textup{cyc}}^{j}\mathbf{L}^{(p)}(\MM(j))$ using the interpolation property of the (conjectural) $p$-adic $L$-function $\mathbf{L}^{(p)}(\MM(j))$ and conclude that
\be\label{eqn:nonvanishingpL}\rho_{\textup{cyc}}^{j}\mathbf{L}^{(p)}(\MM(j))\neq 0\ee 
by our assumption that  $L(\mathcal{M},0)\neq 0$. On the other hand, the interpolation property of Perrin-Riou's extended logarithm (see Theorem~\ref{thm:PRslog}(ii)) shows that the image of $\textup{loc}_p\left(c^{(r)}_\infty(j)\right)$ under
$$\wedge^r H^1_\infty(\QQ_p,T(j)) \stackrel{\frak{Log}^{\otimes r}_{\overline{\eta}}}{\lra} \mathcal{H_\infty}\stackrel{\rho_{\textup{cyc}}^j}{\lra} \QQ_p$$
coincides with the image of $\textup{loc}_p\left(c_\QQ^{(r)}\right)$ under
$$\xymatrix@C=1.2cm{\wedge^r H^1(\QQ_p,T)\ar[r]^(.55){\left(\textup{exp}^{*}\right)^{\otimes r}}& \wedge^r D_{\textup{dR}}(V) \ar[r]^(.6){\alpha^{-1}\beta \cdot\overline{\eta}} &\QQ_p},$$
and since the Bloch-Kato dual exponential $\textup{exp}^*$ factors through the singular quotient $H^1_\textup{s}(\QQ_p,T):=H^1(\QQ_p,T)/H^1_{f}(\QQ_p,T)$, this agrees with the image of $\textup{loc}_p^{\textup{s}}\left(c_\QQ^{(r)}\right)$ under the composite
\be\label{eqn:final}\xymatrix@C=1.2cm{\wedge^r H^1_{\textup{s}}(\QQ_p,T)\ar[r]^(.55){\left(\textup{exp}^{*}\right)^{\otimes r}}& \wedge^r D_{\textup{dR}}(V) \ar[r]^(.6){\alpha^{-1}\beta \cdot\overline{\eta}}&\QQ_p}.\ee
Here $$\alpha=\det(1-p^{-j}\varphi|D_{\textup{cris}}(V(j)^*))\,\, \hbox{ and }\,\,\beta=\det(1-p^{j-1}\varphi^{-1}|D_{\textup{cris}}(V(j)^*)).$$ Both $\alpha$ and $\beta$ are non-zero thanks to our assumption that $1$ is not an eigenvalue for $\varphi\big|_{D_{\textup{cris}}(V)}$ (see Remark~\ref{rem:froboncrisexceptional}).

It then follows from (\ref{eqn:recalldef}) and
(\ref{eqn:nonvanishingpL}) that the image of
$\textup{loc}_p^{\textup{s}}\left(c_\QQ^{(r)}\right)$ under the map
(\ref{eqn:final}) is non-zero, which in turn implies that
$\textup{loc}_p^{\textup{s}}\left(c_\QQ^{(r)}\right)\neq 0$ and the
Theorem is proved.
 \end{proof}

 \begin{rem}
 The proof of Theorem~\ref{thm:mainconjappl} gives a  bound on the Bloch-Kato Selmer group $H^1_{\FF_{\textup{BK}}^*}(\QQ,T^*)$ which is closely related to $L$-values. This lends further evidence to Bloch-Kato conjectures.
  \end{rem}
One may possibly prove an Iwasawa theoretic version of Theorem~\ref{thm:mainconjappl}. Let $\QQ_{p,\infty}$ be the cyclotomic $\ZZ_p$-extension of $\QQ_p$. For every finite sub-extension $\QQ_p\subset K \subset \QQ_{p,\infty}$, suppose that the assumptions of Remark~\ref{rem:greenbergvsBK} hold, as well as that $H^0(K,\textup{F}_p^-T\otimes \Phi/\oo)=0$, so as to ensure that 
\be
\label{eqn:greenbergvsBK}
H^1(K, \textup{F}_p^+T) \stackrel{\sim}{\lra} H^1_f(K,T).
\ee
Assume that Conjecture~\ref{conj:specialelts} holds for the motive $\mathcal{M}$, and that the weak Leopoldt conjecture holds for $T$.
\begin{thm}
\label{thm:cotorsion}
If $\mathbf{L}^{(p)}(\mathcal{M})\neq 0$, then the module $H^1_{\FF_{\textup{Gr}}^*}(\QQ,\mathbb{T}^*)$ is $\LL$-cotorsion. 
\end{thm}
\begin{proof}
Since the $\LL$-module $\wedge^r H^1_s(\QQ_p,\mathbb{T})$ is free of rank one, this will follow from Theorem~\ref{thm:mainkinfty}, once we verify that $\textup{loc}^s_p(c^{(r)}_{\infty})\neq 0$. This however follows from our assumption that $\mathbf{L}^{(p)}(\mathcal{M})\neq 0$, the conjectural description of $c^{(r)}_{\infty}$ via Conjecture~\ref{conj:specialelts} (which we assume) and the fact that the dual exponential map by definition factors through the singular quotient 
$$\varprojlim_{\QQ_p\subset K \subset \QQ_{p,\infty}} \frac{H^1(K,T)}{H^1_f(K,T)}\stackrel{\sim}{\lra}\varprojlim_{K \subset \QQ_{p,\infty}} \frac{H^1(K,T)}{H^1(K,\textup{F}_p^+T)}\stackrel{\sim}{\lra}H^1(\QQ_p,\textup{F}_p^-\mathbb{T})=H^1_s(\QQ_p,\mathbb{T}),$$
where the first isomorphism is the inverse of the isomorphism~(\ref{eqn:greenbergvsBK}) above.
\end{proof}

\subsection*{Acknowledgements.} The main idea of this work has occurred to the author during the period when he held a William Hodge post-doctoral fellowship at IH\'ES, and the paper was finalized during his stay at MPIM-Bonn and his post at Ko\c{c} University in Istanbul. The author thanks heartily all three institutes for their hospitality. The author also thanks Karl Rubin for sharing his insights with the author, which essentially led him to write this paper, and thanks Tadashi Ochiai for helpful correspondence. He finally thanks the anonymous referee, who has kindly pointed out and corrected several inaccuracies in an earlier version. He also acknowledges an EU-FP7 grant and a TUBITAK-Career grant which partially supported the author when this research was conducted.

{\scriptsize
\bibliographystyle{halpha}
\bibliography{references}
}
\end{document}